\documentclass[a4paper,notitlepage,twoside,reqno,10pt]{amsart}

\usepackage{bbm,pifont,latexsym}

\usepackage{dcolumn,indentfirst}
\usepackage[hypertexnames=false,bookmarksopen=true,linktocpage=true,pdfstartview={XYZ null null 1.25}]{hyperref}
\usepackage{amsmath,amssymb,amscd,amsthm,amsfonts,mathrsfs}
\usepackage{color,graphicx,xcolor,graphics,subfigure,extarrows,caption2}
\usepackage{titletoc}

\newtheorem{thm}{Theorem}[section]
\newtheorem{lema}[thm]{Lemma}
\newtheorem{cor}[thm]{Corollary}
\newtheorem{prop}[thm]{Proposition}
\theoremstyle{definition}
\newtheorem*{defi}{Definition}

\newtheorem*{rmk}{Remark}
\newtheorem{ques}{Question}

\newcommand{\D}{\mathbb{D}}
\newcommand{\T}{\mathbb{T}}
\newcommand{\R}{\mathbb{R}}
\newcommand{\Z}{\mathbb{Z}}
\newcommand{\N}{\mathbb{N}}

\newcommand{\C}{\mathbb{C}}
\newcommand{\EC}{\widehat{\mathbb{C}}}
\newcommand{\MA}{\mathcal{A}}

\newcommand{\A}{\mathbb{A}}

\newcommand{\MV}{\mathcal{V}}

\newcommand{\Mod}{\textup{mod}}
\newcommand{\ii}{\textup{i}}

\makeatletter\@addtoreset{equation}{section}\makeatother

\begin{document}

\author{YOUMING WANG}
\address{Department of Applied Mathematics, Hunan Agricultural University, Changsha, 410128, P. R. China}
\email{wangym2002@163.com}

\author{FEI YANG}
\address{Department of Mathematics, Nanjing University, Nanjing, 210093, P. R. China}
\email{yangfei@nju.edu.cn}

%---------------------------------------------------------------------------------------------------------------
\title[Buried Julia components]{JULIA SETS AS BURIED JULIA COMPONENTS}

\begin{abstract}
Let $f$ be a rational map with degree $d\geq 2$ whose Julia set is connected but not equal to the whole Riemann sphere. It is proved that there exists a rational map $g$ such that $g$ contains a buried Julia component on which the dynamics is quasiconformally conjugate to that of $f$ on the Julia set if and only if $f$ does not have parabolic basins and Siegel disks. If such $g$ exists, then the degree can be chosen such that $\deg(g)\leq 7d-2$. In particular, if $f$ is a polynomial, then $g$ can be chosen such that $\deg(g)\leq 4d+4$. Moreover, some quartic and cubic rational maps whose Julia sets contain buried Jordan curves are also constructed.
\end{abstract}

% AMS subject classifications (used in AMS journals)
\subjclass[2010]{Primary: 37F45; Secondary: 37F10}

% AMS keywords (used in AMS journals)
%\keywords{Julia sets; buried component; singularly perturbations}

% today's date, or fill in whatever date you prefer
\date{\today}

% acknowledge support, etc
% \thanks{This research was partially supported by NSF grant DOA-123456789.}
% \thanks{We would like to thank our colleagues for their helpful criticism.}

% dedication
% \dedicatory{Dedicated to Professor Donald Knuth on the occasion of his $100$th birthday}

\maketitle

%----------------------------------------------------------------------------------------------------------------
%\vskip1.0cm
\tableofcontents
%----------------------------------------------------------------------------------------------------------------

%----------------------------------------------------------------------------------------------------------------
\section{Introduction}

Let $f:\EC\to\EC$ be a rational map with degree at least two. The \textit{Julia set} $J(f)$ of $f$ is the set of the points which fail to be normal in the sense of Montel. Or equivalently, $J(f)$ is the closure of the repelling periodic points of $f$. The complement of $J(f)$ is the \textit{Fatou set} of $f$ which we denote by $F(f)$. Each connected component of $J(f)$ (resp. $F(f)$) is called a \textit{Julia} (resp. \textit{Fatou}) \textit{component}. It was known that the Julia components can exhibit several kinds of shapes: the singletons, Jordan curves and some other complex topologies.

A Julia component (or a point on the Julia set) is called \textit{buried} provided it is disjoint with the boundary of any Fatou component. In particular, buried Julia components cannot occur in the polynomials since the Julia set coincides with the boundary of the unbounded Fatou component. The first example of buried Julia component was constructed by McMullen in \cite{McM88}. Consider the family of rational maps which is given by
\begin{equation*}
f_{c,\lambda}(z)=z^\ell+c+\frac{\lambda}{z^m}, \text{ where } \ell\geq 2,m\geq 1 \text{ and } c,\lambda\in\C.
\end{equation*}
McMullen proved that if $c=0$, $1/\ell+1/m<1$ and $\lambda\in\C\setminus\{0\}$ is small enough, then $J(f_{c,\lambda})$ is a Cantor set of circles and contains infinitely many buried Julia components which are Jordan curves (see also \cite{DLU05}). In particular, the same phenomenon occurs if $c\neq 0$ is small (see \cite{XQY14}). In \cite{PT00}, Pilgrim and Tan studied an example which is slightly different from $f_{c,\lambda}$:
\begin{equation*}
\widetilde{f}_{-1,\lambda}(z)=\frac{1}{(\frac{1}{z})^2-1}+\frac{\lambda}{z^3}=\frac{z^2}{1-z^2}+\frac{\lambda}{z^3}, \text{ where } \lambda\in\C.
\end{equation*}
They proved that if $\lambda\neq 0$ is small enough, then $J(\widetilde{f}_{-1,\lambda})$ contains a Julia component which is homeomorphic to the Julia of $z\mapsto z^2-1$. Moreover, the Julia components of $J(\widetilde{f}_{-1,\lambda})$ which are not eventually periodic are buried Jordan curves.

Let $f$ be a rational map with degree at least two. Beardon proved that the Julia set $J(f)$ has buried components if it is disconnected and every component of the Fatou set $F(f)$ has finite connectivity \cite{Bea91b}. Then Qiao proved that $J(f)$ has buried components is equivalent to it is disconnected and $F(f)$ has no completely invariant component \cite{Qia95}. Although the existence of buried Julia components which are singletons was known very early, the first specific example was given in \cite{Qia95} by applying Beardon's criterion. He showed that the Julia set of Herman's example contains some buried components which are singletons.

In 2008 a specific example in \cite{BDGR08} shows that if $c\neq 0$ is the center of a hyperbolic component of the Multibrot set $M_q$, then $J(f_{c,\lambda})$ contains both Jordan curves and singletons which are buried Julia components, where $q=\ell=m\geq 3$ and $\lambda\neq 0$ is small enough (see also \cite{GMR13}).

At the time when McMullen gave the first example of buried Julia components \cite[p.\,55]{McM88}, he also raised the following question:
\begin{ques}\label{ques-1}
Does there exist a rational map with degree less than $5$ having a buried Julia component which is not a singleton?
\end{ques}
In 2015, Godillon constructed a family of \textit{cubic} rational maps
\begin{equation}\label{equ-Godillon}
F_\lambda(z)=\frac{(1-\lambda)[(1-4\lambda+6\lambda^2-\lambda^3)z-2\lambda^3]}{(z-1)^2[(1-\lambda-\lambda^2)z-2\lambda^2(1-\lambda)]}, \text{ where }\lambda\in\C,
\end{equation}
and proved that $J(F_\lambda)$ contains a buried Julia component which is neither a Jordan curve nor a singleton if $\lambda\neq 0$ is small enough \cite{God15}. In particular, his buried Julia component is homeomorphic to the Julia set of $z\mapsto 1/(z-1)^2$, which has a super-attracting periodic orbit $0\mapsto 1\mapsto \infty\mapsto 0$ with period $3$. According to \cite{Yin92} and \cite[Lemma 8.2]{Mil93}, the Julia set of any quadratic rational map is either connected or a Cantor set. This implies that Godillon's example is optimal in terms of degree since the rational maps with degree at most two cannot contain any buried Julia components.

For the buried \textit{points} in the Julia sets of rational maps, the problem of existence has not been solved completely. Makienko conjectured that the Julia set of a rational map $f$ has buried points if and only if there is no completely invariant component of the Fatou set of $f^{\circ 2}$ (see \cite{Mak87}). One can refer to \cite{Mor97}, \cite{Qia97}, \cite{Mor00}, \cite{SY03}, \cite{CMMR09}, \cite{CMT13} and the references therein for the progress.

\subsection{Statement of the main results}

In \cite[Theorem 3.4]{McM88}, McMullen proved the following result: Let $J_0$ be a Julia component of a rational map $f$ (other than a single point) such that $f(J_0)=J_0$. Then there exists another rational map $g$ such that $f:J_0\to J_0$ is quasiconformally conjugate to $g:J(g)\to J(g)$. By quasiconformally conjugate of $f|_{J_0}$ and $g|_{J(g)}$ we mean that there exists a quasiconformal mapping $\varphi:\EC\to\EC$ such that $\varphi(J_0)=J(g)$ and the restriction of $\varphi$ in a small open neighborhood of $J_0$ is a conjugacy of $f$ and $g$. This implies that the dynamics of $f$ can be decomposed into some small pieces. Actually, in complex dynamics, how to decompose a dynamical systems is an important problem.

\medskip

In this article, we are interested in the inverse procedure of McMullen:
\begin{ques}
Could any connected Julia set appear as a buried Julia component of a higher degree rational map?
\end{ques}
As a complete answer to this question, we prove the following result.

\begin{thm}\label{thm-main}
Let $f:\EC\to\EC$ be a rational map with degree at least two whose Julia set $J(f)\neq\EC$ is connected. Then there is a rational map $g$ such that $g$ has a buried Julia component on which $g$ is quasiconformally conjugate to $f$ on $J(f)$ if and only if $f$ does not have parabolic basins and Siegel disks.

If such $g$ exists, then the degree of $g$ can be chosen such that $\deg(g)\leq 7d-2$, where $d=\deg(f)$. In particular, $\deg(g)\leq 4d+4$ if $f$ is a polynomial.
\end{thm}

Theorem \ref{thm-main} implies that a rational map can contain buried Julia components which are ``almost" arbitrary. Indeed, the assumptions in Theorem \ref{thm-main} allow the presence of critical points on the buried Julia components. The upper bound of the degree of $g$ is not sharp in general. In fact, if $J(f)$ is not a dendrite (see \S\S\ref{sec-second-pert}, \ref{sec-one-orbit}), then $\deg(g)$ can be chosen at most $7d-6$ (and at most $4d$ if $f$ is a polynomial).

The proof of Theorem \ref{thm-main} is based on successive perturbations and quasiconformal surgery. See Figure \ref{Fig_J-buried} for an example, which illustrates the process of perturbations\footnote{Note that in Figure \ref{Fig_J-buried}, the degree of the last rational map is $12$. However, by the proof of a special case of Theorem \ref{thm-main} in \S\ref{sec-second-pert}, there should exist a rational map with degree $8$ whose Julia set contains a buried copy of basilica. This kind of rational maps exist indeed. But in order to obtain the buried Julia component, the second parameter $\mu$ there should be chosen extremely small. In this case we cannot obtain a good picture.} of the map $z\mapsto z^2-1$.

\begin{figure}[!htpb]
  \setlength{\unitlength}{1mm}
  \centering
  \includegraphics[width=110mm]{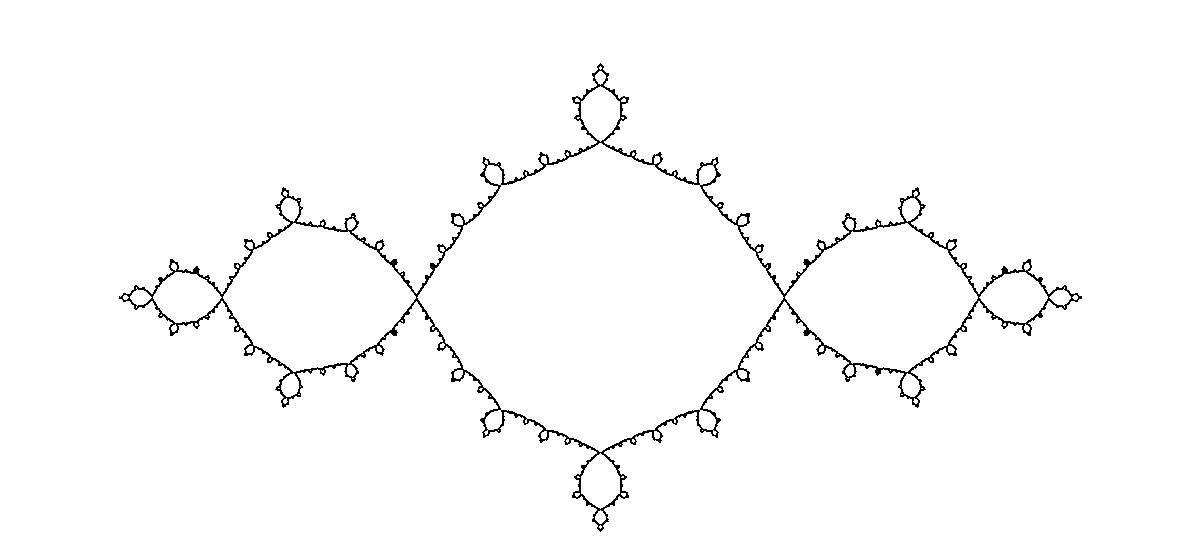}\vskip0.2cm
  \includegraphics[width=110mm]{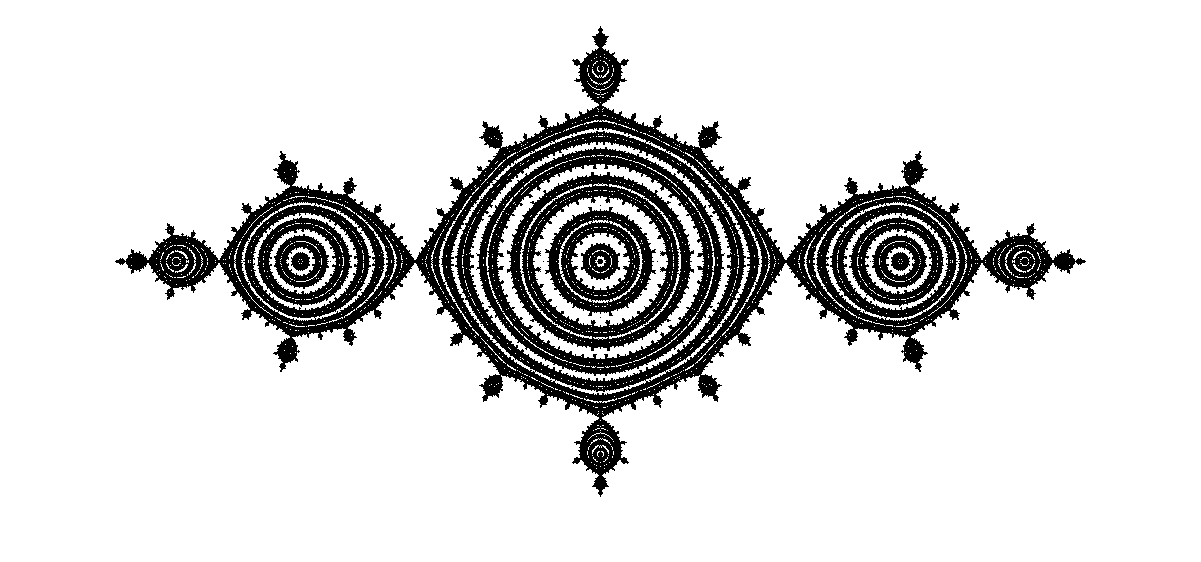}\vskip0.2cm
  \includegraphics[width=110mm]{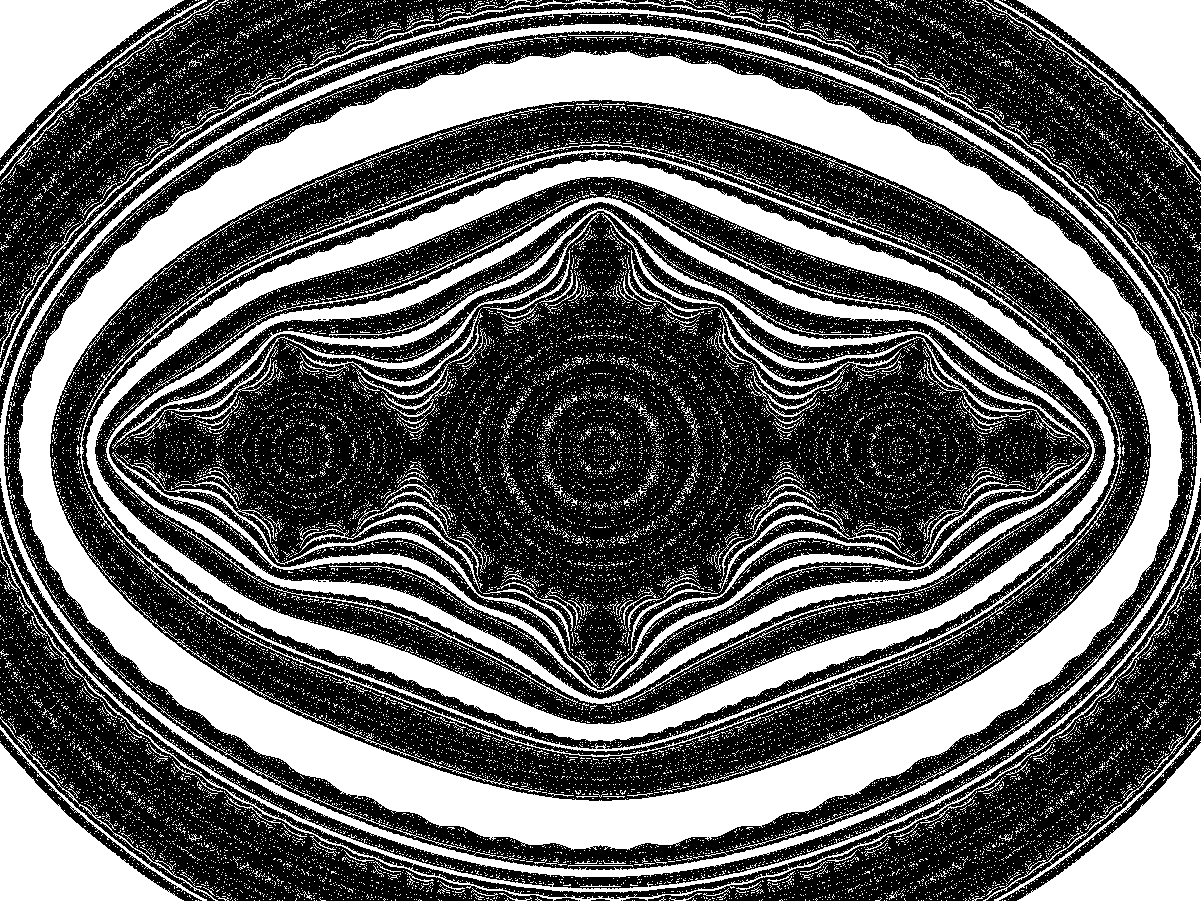}
  \caption[]{The Julia sets of \\
  \begin{minipage}{\linewidth}
  \begin{equation*}
    z\mapsto z^2-1, ~z\mapsto z^2-1+\frac{\lambda}{z^3} \text{ and } z\mapsto 1/\Big(\big(z^2-1+\frac{\lambda}{z^3}\big)^{-1}+\mu z^7\Big)
  \end{equation*}
  \end{minipage}
  (from the top down), where $\lambda=10^{-10}$ and $\mu=10^{-14}$. The middle and bottom Julia sets, respectively, contain a semi-buried and buried component which are homeomorphic to the top Julia set: the basilica.}
  \label{Fig_J-buried}
\end{figure}

A parameter $c\in\C$ (resp. $b\in\C$) is called the \textit{center} of a hyperbolic component of the Mandelbrot set (resp. Multibrot set $M_q$) if the critical point $0$ is a periodic point with period $p\geq 1$ of $P_c(z)=z^2+c$ (resp. $z\mapsto z^q+b$ with $q\geq 3$). In this article, we consider the singular perturbation of this kind of unicritical polynomials and prove the following result.

\begin{thm}\label{thm-exam}
Let $c\neq 0$ (resp. $b\neq 0$) be the center of a hyperbolic component of the Mandelbrot set (resp. Multibrot set $M_q$ with $q\geq 3$). If $\lambda\neq 0$ (resp. $\mu\neq 0$) is small enough, then the Julia sets of
\begin{equation}\label{equ-quartic}
f_\lambda(z)=z^2+c+\frac{\lambda}{(z-c)^2} \quad\text{and}\quad g_\mu(z)=z^q+b+\frac{\mu}{z-b}
\end{equation}
contain infinitely many buried Julia components which are Jordan curves.
\end{thm}

This answers Question \ref{ques-1} of McMullen proposed in \cite{McM88} by simple examples. See Figure \ref{Fig_J-buried-2} for two specific examples. A similar singular perturbation of $z\mapsto z^q+b$ with the form $z\mapsto z^q+b+\mu/z^q$ was considered in \cite{BDGR08}, where $b\neq 0$ is the center of a hyperbolic component of the Multibrot set $M_q$ with $q\geq 3$. However, the perturbation there is made at the \textit{critical point} $0$. In order to obtain the buried Julia components, that perturbation makes the degree of the resulting rational map at least $6$. Here our perturbation is made at the \textit{critical value}. Hence the degree of the perturbed rational map can be reduced to $4$ (if $q=3$).

On the other hand, a singular perturbation of $z\mapsto z^2+c$ with the form $z\mapsto z^2+c+\lambda/z^2$ was considered in \cite{Mar08}, where $c\neq 0$ is the center of a hyperbolic component of the Mandelbrot set. Similarly, the perturbation is made also at the critical point $0$. The degree of the resulting rational map is $4$ but this map has connected Julia set and hence cannot contain any buried Julia components. See also \cite{GMR13} for a singular perturbation of $z\mapsto z^2+c$ on the corresponding bounded super-attracting cycle by adding one pole to each point in the cycle.

\begin{figure}[!htpb]
  \setlength{\unitlength}{1mm}
  \centering
  \includegraphics[height=50mm]{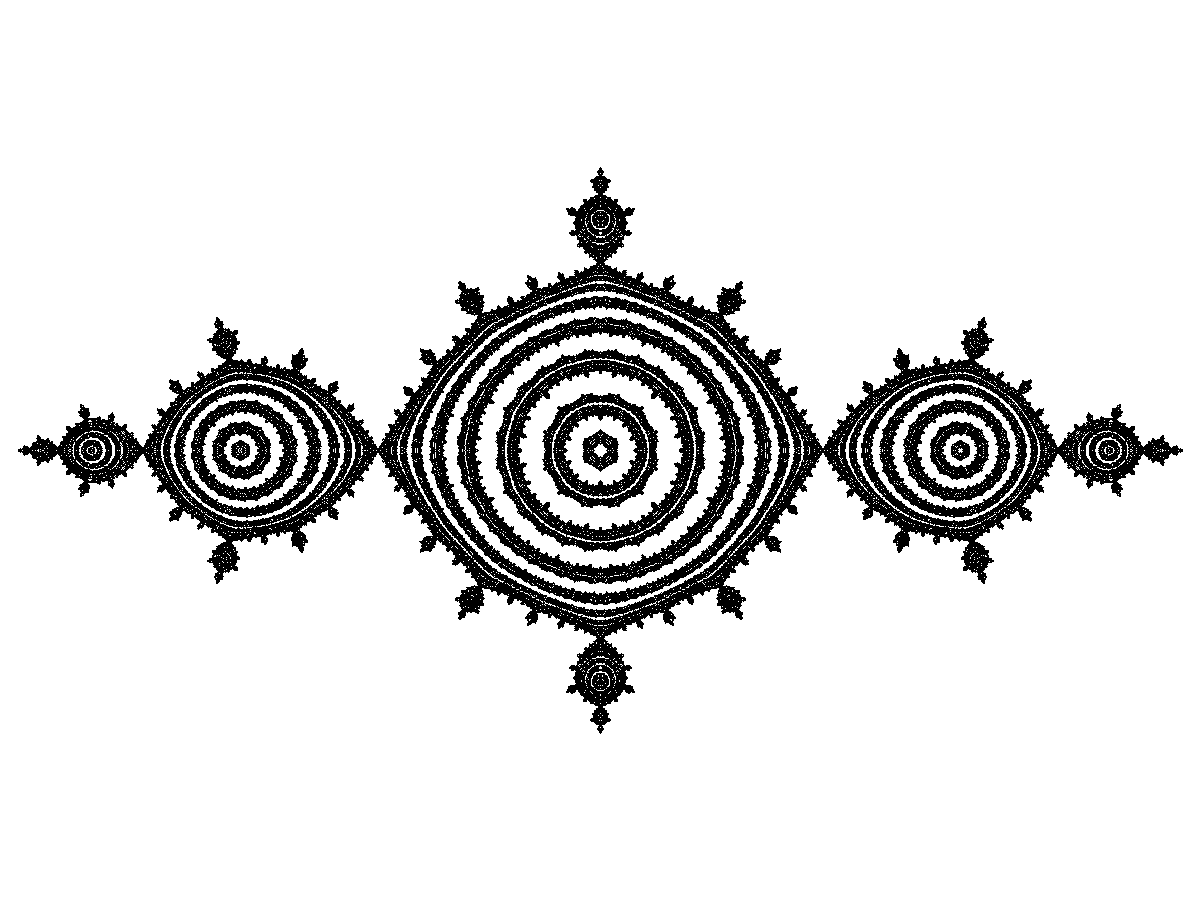}\hskip0.1cm
  \includegraphics[height=50mm]{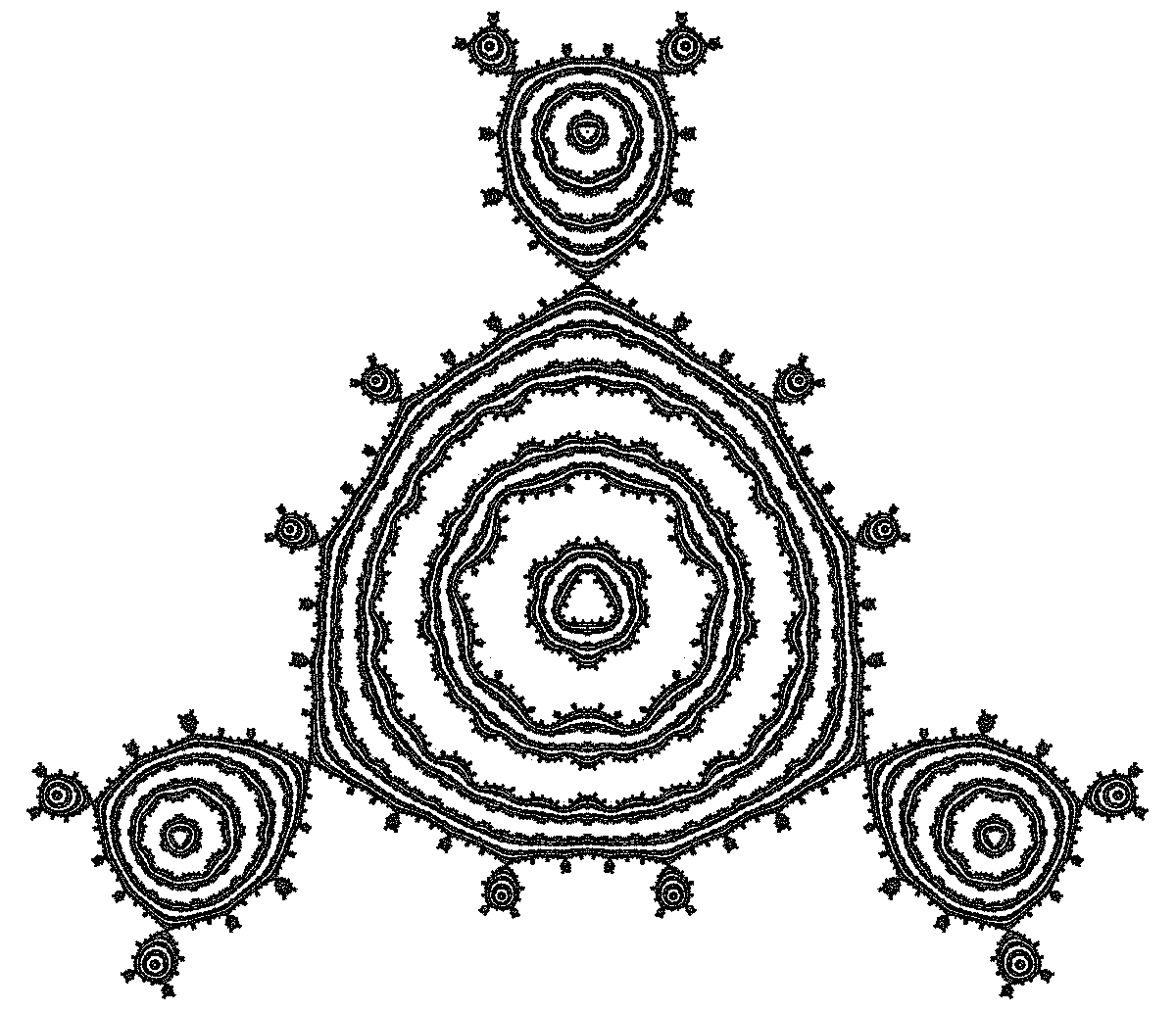}
  \caption[]{The Julia sets of two quartic rational maps \\
  \begin{minipage}{\linewidth}
  \begin{equation*}
    f_\lambda(z)=z^2-1+\frac{\lambda}{(z+1)^2} \text{\quad and\quad} g_\mu(z)=z^3+\ii+\frac{\mu}{z-\ii},
  \end{equation*}
  \end{minipage}
  where $\lambda=10^{-7}$ and $\mu=10^{-4}$. Both of these two Julia sets contain infinitely many buried Julia components which are Jordan curves.}
  \label{Fig_J-buried-2}
\end{figure}

Godillon's example in \eqref{equ-Godillon} is a cubic rational map but the construction depends heavily on the complicated combinatorics which are encoded by a weighted dynamical tree. Here we adopt the idea in the proof of Theorem \ref{thm-main} and the construction of Theorem \ref{thm-exam}, to give a simple example of cubic rational map such that the corresponding Julia set contains a buried Julia component which is neither a Jordan curve nor a singleton.

\begin{thm}\label{thm-exam-2}
Let $a\neq -(3+\sqrt{5})/2$ be the center of a hyperbolic component of
\begin{equation*}
Q_a(z)=1+a/z^2, \quad\text{where } a\in\C\setminus\{0\}.
\end{equation*}
If $\nu\neq 0$ is small enough, then the Julia set of the cubic rational map
\begin{equation}\label{equ-cubic}
f_\nu(z)=1+\frac{a}{z^2}+\frac{(1+a+\sqrt{-a})\,\nu}{z-1-\nu}
\end{equation}
contains a buried Julia component which is homeomorphic to $J(Q_a)$ and infinitely many buried Julia components which are Jordan curves.
\end{thm}

This gives also an answer to the question of McMullen in terms of different cubic rational maps from Godillon's. See Figure \ref{Fig_J-buried-3} for a specific example.
The family $Q_a(z)=1+a/z^2$, where $a\in\C\setminus\{0\}$ was first studied by Lyubich in \cite[\S 2.4]{Lyu86}. Note that $Q_a$ has two critical points $0$, $\infty$ and exactly one critical orbit
\begin{equation*}
0\overset{(2)}{\longmapsto}\infty\overset{(2)}{\longmapsto}1\overset{(1)}{\longmapsto} 1+a \longmapsto\cdots
\quad\text{and}\quad \pm\sqrt{-a}\overset{(1)}{\longmapsto}0.
\end{equation*}
Lyubich asked whether there exist $a\neq 0$ such that $Q_a$ has a Herman ring. In 1987 Shishikura gave a negative answer to this question \cite{Shi87} and later another proof was provided by Bam\'{o}n and Bobenrieth \cite{BB99} (see also \cite{Yan17}).

\begin{figure}[!htpb]
  \setlength{\unitlength}{1mm}
  \centering
  \includegraphics[width=115mm]{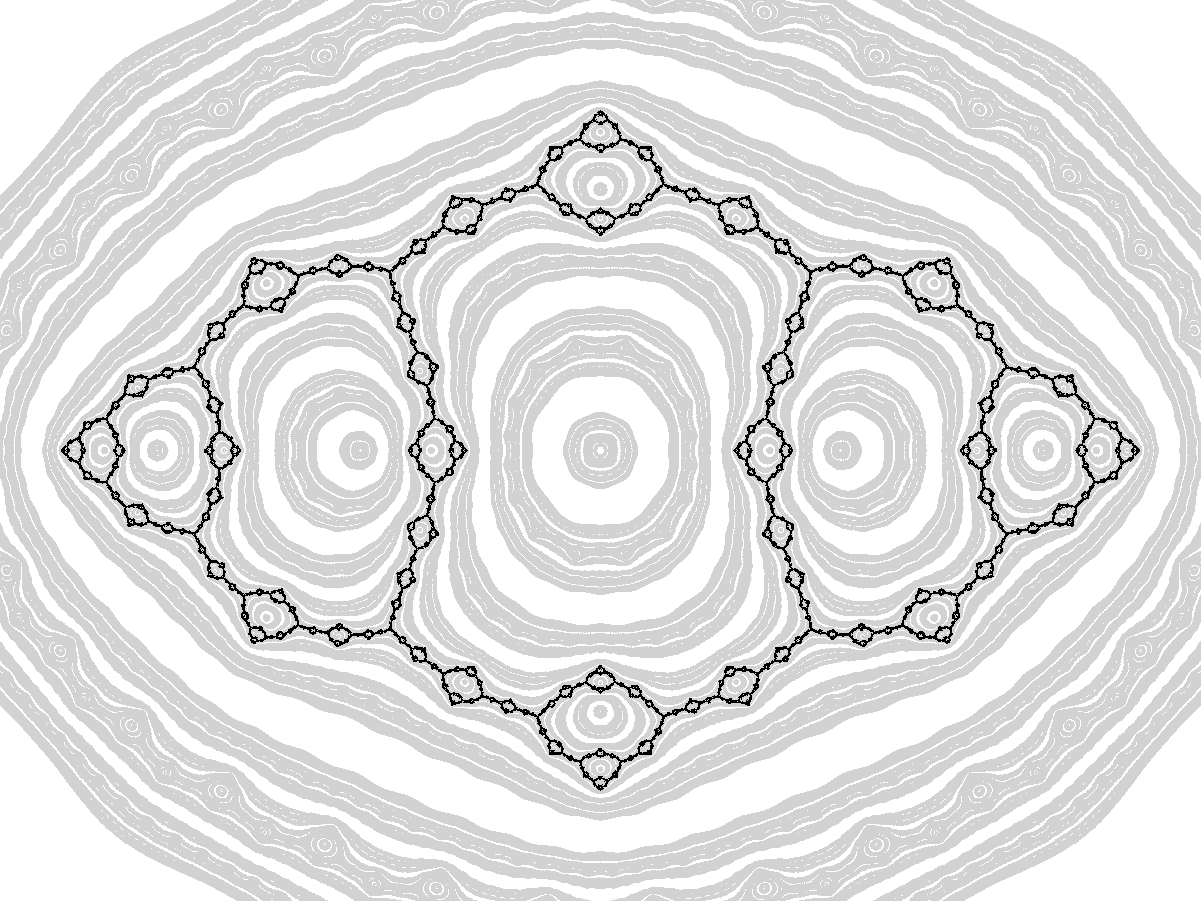}
  \caption{The Julia set of cubic rational map $f_\nu(z)=1-1/z^2+\nu/(z-1-\nu)$, where $\nu=10^{-5}$. This Julia set contains infinitely many buried Julia components which are Jordan curves (included in gray parts) and contains also a buried Julia component which is homeomorphic to the Julia set of $z\mapsto 1-1/z^2$ (the black part).}
  \label{Fig_J-buried-3}
\end{figure}

In this article we adopt the convention that $\arg\sqrt{z}\in(-\frac{\pi}{2},\frac{\pi}{2}]$ for $z\in\C\setminus\{0\}$. In Theorem \ref{thm-exam-2} we exclude $a=-(3+\sqrt{5})/2$ since this is equivalent to $1+a+\sqrt{-a}=0$ and in this case we would not obtain a singular perturbation.

Let us compare our example with Godillon's in \eqref{equ-Godillon}. His formula is a bit complicated but can be written as family with only one free critical point hence one can draw the bifurcation locus in the complex plane. Our formula is simpler but there are several free critical points.

\subsection{The idea of the proofs}

Let us give a sketch of the proof of Theorem \ref{thm-main}. For the sufficiency, suppose that $f$ has a connected Julia set and that it does not have parabolic basins and Siegel disks. A standard quasiconformal surgery guarantees that one can assume that all periodic Fatou components of $f$ are super-attracting and $f$ is post-critically finite in the Fatou set (see \cite[p.\,106]{CG93}). The proof will be divided into two main steps: The first one is to perturb all super-attracting periodic orbits (except one) and obtain a quasiregular map $F$ (if $f$ has at least two disjoint super-attracting periodic orbits). Then this map $F$ is conjugate to a rational map $h$ by quasiconformal surgery principle. The second step is to perturb the remaining super-attracting periodic orbit and obtain a quasiregular map $H$ which can be conjugated to a rational map $g$  (if $f$ has exactly one super-attracting periodic orbit we omit the first step). We will prove that the Julia set of the first rational map $h$ has a ``semi-buried" Julia component. Then we show that the second rational map $g$ contains a fully buried Julia component which is a copy of $J(f)$.

For the necessity in Theorem \ref{thm-main} we use a proof of contradiction. Since topological conjugacy preserves the multiplier at indifferent periodic points, one can exclude the existence of parabolic periodic points in the buried Julia components. If $f$ has a Siegel disk $U$ and let $\varphi:J(f)\to J_0(g)$ be a quasiconformal conjugacy defined from a neighborhood of $J(f)$ to that of a Julia component $J_0(g)$ of $g$. It is easy to see that $U$ contains an annular neighborhood $A$ such that $\varphi(A\cap U)$ is contained in a rotation domain of $g$. However, this is impossible if $\varphi(J(f))$ is a buried Julia component of $g$. For details, see \S\ref{sec-second-pert}.

As stated above, the constructions in Theorems \ref{thm-exam} and \ref{thm-exam-2} are inspired by the proof of Theorem \ref{thm-main}. Perturbing at the critical values not only provides the annulus-to-disk dynamical behavior but also reduces the degrees as much as possible. For details, see \S\ref{sec-quartic} and \S\ref{sec-cubic}.

\vskip0.2cm
\noindent\textbf{Notations.} We collect some notations which will be used throughout of this article. Let $\N$, $\Z$, $\R$ and $\C$, respectively, be the set of natural numbers, integers, real numbers and complex numbers. For $r>0$, we use $\D(a,r):=\{z\in\C:|z-a|<r\}$ to denote the round disk centered at $a\in\C$ with radius $r$, and $\T_r:=\{z\in\C:|z|=r\}$ be the boundary of $\D(0,r)$. For $0<r<1$, we denote the annulus $\A_r:=\{z\in\C:r<|z|<1\}$.

\vskip0.2cm
\noindent\textbf{Acknowledgements.} This work is supported by the National Natural Science Foundation of China (grant No.\,11671092), the Natural Science Foundation of Jiangsu Province (No.\,BK20191246), the Fundamental Research Funds for the Central Universities (grant No. 0203-14380025) and EDF of Hunan Province (grant No.\,16C0763). We would like to thank Guizhen Cui and Yongcheng Yin for helpful comments and suggestions. We are also very grateful to the referee for insightful and detailed comments, suggestions and corrections; and to the editor for his patient replies on our several inquiries on the status of this paper.

\section{Semi-buried Julia components}\label{sec-semi-buried}

In this section, we will operate the first quasiconformal surgery on the given rational map $f$ and obtain a new rational map $h$. This map $h$ has a ``semi-buried" Julia component on which the restriction of $h$ is quasiconformally conjugate to $f$ on the Julia set $J(f)$.

Let $f$ be a rational map with degree at least two whose Julia set is connected. Further, we assume that the Fatou set of $f$ is non-empty and does not contain parabolic basins and Siegel disks. This implies that all the periodic Fatou components of $f$ are attracting or super-attracting. By a standard quasiconformal surgery, see \cite[Theorem 5.1, p.\,106]{CG93} for example, we assume in the following that $f$ is a post-critically finite rational map in the Fatou set. This implies that all the periodic Fatou components of $f$ are super-attracting and each critical point in the Fatou set will be iterated to the super-attracting cycle eventually.

Let $O_1$, $O_2$, $\dots$, $O_n$ be the collection of different cycles of super-attracting periodic Fatou components of $f$. It is well known that $1\leq n\leq 2d-2$ (see \cite{Shi87} or \cite[Theorem 9.4.1]{Bea91a}). In the remaining of this section, we assume that $n\geq 2$ and the case for $n=1$ will be handled in \S\ref{sec-one-orbit} and \S\ref{sec-dendrite}.

For $1\leq i\leq n$, let $p_i\geq 1$ be the (minimal) period of $O_i$ and $a_i$ a super-attracting periodic point of $f$ contained in $O_i$. For $1\leq j\leq p_i$, let $B_{i,j}$ be the Fatou component containing the point $f^{\circ (j-1)}(a_i)$. Hence we have $O_i=\{B_{i,j}:1\leq j\leq p_i\}$ for all $1\leq i\leq n$. Moreover, we use $\MA_i$ to denote the super-attracting basin of $f$ containing $O_i$ (note that each $B_{i,j}$ is connected but $\MA_i$ may be disconnected). Without loss of generality, we assume that $a_1=0$ and $a_n=\infty$.

\subsection{Cutting along the equipotentials I}\label{subsec-equi}

For performing the quasiconformal surgery, we need to divide $\EC$ into several pieces on which a quasiregular map $F$ will be piecewisely defined. This partition comes from some equipotential curves of $f$ in the super-attracting periodic orbits.

According to B\"{o}ttcher's theorem, each super-attracting periodic orbit $O_i$ provides the Riemann mappings $\phi_{i,j}:B_{i,j}\to\D$, where $1\leq i\leq n-1$ and $1\leq j\leq p_i$, such that $\phi_{i,j}(f^{\circ (j-1)}(a_i))=0$ and the following diagram commutes:
$$\begin{CD}
B_{i,1} @>f>> B_{i,2} @>f>> \cdots  @>f>> B_{i,p_i} @>f>> B_{i,1}\\
@VV\phi_{i,1}V   @VV\phi_{i,2}V        @VVV @VV\phi_{i,p_i}V  @VV\phi_{i,1}V\\
\D @>z\mapsto z^{d_{i,1}}>> \D @>z\mapsto z^{d_{i,2}}>> \cdots  @>z\mapsto z^{d_{i,p_i-1}}>> \D @>z\mapsto z^{d_{i,p_i}}>> \D,
\end{CD}$$
where\footnote{We assume that $d_{i,p_i}\geq 2$ and $d_{i,j}\geq 1$ for $1\leq j\leq p_i-1$ rather than $d_{i,1}\geq 2$ and $d_{i,j}\geq 1$ for $2\leq j\leq p_i$ since this normalization can reduce the degree of the new rational map after the surgery as much as possible. See Lemma \ref{lema:cover-disk-to-disk}.} $d_{i,p_i}\geq 2$ and $d_{i,j}\geq 1$ are positive integers for $1\leq j\leq p_i-1$. For convenience, we denote
\begin{equation}\label{equ:d_i}
d_i:=\prod_{j=1}^{p_i}d_{i,j}\geq 2.
\end{equation}
Hence $f^{\circ p_i}:B_{i,j}\to B_{i,j}$ is holomorphically conjugate to $z\mapsto z^{d_i}$ from $\D$ to itself, where $1\leq i\leq n-1$ and $1\leq j\leq p_i$.

An \textit{equipotential curve} (or \textit{equipotential} in short) $\gamma$ in $B_{i,j}$ is the preimage by $\phi_{i,j}$ of an Euclidean circle in $\D$ centered at 0. The radius of this circle is called the \textit{level} of $\gamma$ and is denoted by $L_{i,j}(\gamma)\in(0,1)$, i.e. $\gamma=\{z\in B_{i,j}:|\phi_{i,j}(z)|=L_{i,j}(\gamma)\}$.

We will define the corresponding potentials and levels in the basin $B_{n,j}$ (actually is its conformal image of the conjugacy map) with $1\leq j\leq p_n$ in the next section since they will not be used in the present section. The definitions will be modified slightly such that the levels of the equipotentials in $B_{n,j}$ are contained in $(1,+\infty)$ since $\infty\in B_{n,1}$.

Let $\gamma$ be a Jordan curve contained in $\EC\setminus\{\infty\}$. We use $D(\gamma)$ to denote the connected component of $\EC\setminus\gamma$ which does not contain $\infty$. For two disjoint connected compact sets $\gamma_1$ and $\gamma_2$ in $\EC\setminus\{\infty\}$ which are not singletons, we use $A(\gamma_1,\gamma_2)$ to denote the unique annular component\footnote{In order to emphasize the dynamical correspondences, we may modify the order of $\gamma_1$ and $\gamma_2$ defining an annulus. That means sometimes the annulus $A(\gamma_1,\gamma_2)$ will be denoted by $A(\gamma_2,\gamma_1)$.}
of $\EC\setminus(\gamma_1\cup\gamma_2)$. Moreover, $A(\gamma_1,\gamma_2)$ is biholomorphically equivalent to a standard annulus $\mathbb{A}_r:=\{z\in\C :r<|z|<1\}$ with $r\in(0,1)$ whose \textit{conformal modulus} is $\Mod(\mathbb{A}_r)=(1/2\pi)\log(1/r)$.

\begin{lema}[{holomorphic covering from disk to disk}]\label{lema:cover-disk-to-disk}
For each $1\leq i\leq n-1$, let $m_i\geq 1$ be an integer satisfying $m_i>d_{i,1}/(d_i-1)$. Then there exist equipotentials\footnote{If $p_i=1$ for some $1\leq i\leq n-1$, then $B_{i,2}=B_{i,1}$ and $\gamma_{i,2}$ is regarded as a curve in $B_{i,1}$.} $\gamma_{i,1}$, $\gamma_{i,p_i+1}$, $\xi_{i,1}\subset B_{i,1}$, $\gamma_{i,2}\subset B_{i,2}$ and a holomorphic branched mapping $F|_{D(\xi_{i,1})}:D(\xi_{i,1})\to\EC\setminus\overline{D}(\gamma_{i,2})$ satisfying the following conditions:
\begin{enumerate}
\item $L_{i,1}(\gamma_{i,p_i+1})<L_{i,1}(\xi_{i,1})<L_{i,1}(\gamma_{i,1})$;
\item $F(a_i)=\infty$ and $F:D(\xi_{i,1})\setminus\{a_i\}\to\C\setminus\overline{D}(\gamma_{i,2})$ is a degree $m_i$ covering map; and
\item $F(\xi_{i,1})=\gamma_{i,2}$ and $F(\gamma_{i,p_i+1})=\eta_i$, where $\eta_i$ is a real-analytic Jordan curve in $B_{n,1}$ separating $\infty$ from $\partial B_{n,1}$.
\end{enumerate}
\end{lema}

\begin{proof}
Without loss of generality, we prove this lemma only for $i=1$ since the other cases are completely similar. According to the normalization, we have $a_1=0$. For each small $r\in(0,1)$, let $\gamma_{1,1}$, $\gamma_{1,p_1+1}$ be the equipotentials in $B_{1,1}$ such that $L_{1,1}(\gamma_{1,1})=r$ and $L_{1,1}(\gamma_{1,p_1+1})=r^{d_1}$. For $2\leq j\leq p_1$, we denote $\gamma_{1,j}:=f^{\circ (j-1)}(\gamma_{1,1})$. Then $\gamma_{1,j}$ is an equipotential in $B_{1,j}$ for all $1\leq j\leq p_1$. See Figure \ref{Fig_surgery} for a partial illustration.

\begin{figure}[!htpb]
  \setlength{\unitlength}{1mm}
  \centering
  \includegraphics[width=0.95\textwidth]{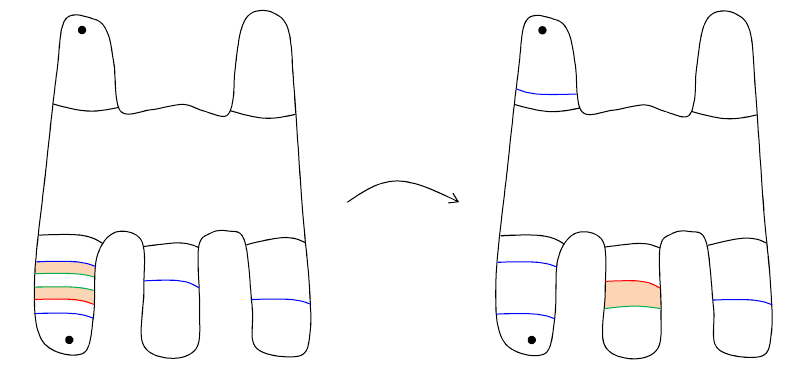}
  \put(-113,5){$0$}
  \put(-110,48){$\infty$}
  \put(-105,16){$\gamma_{1,1}$}
  \put(-106,8){$\gamma_{1,p_1+1}$}
  \put(-90,13){$\gamma_{1,2}$}
  \put(-73,10){$\gamma_{1,p_1}$}
  \put(-121,11){\footnotesize{$\xi_{1,1}$}}
  \put(-121,14){\footnotesize{$\alpha_{1,1}$}}
  \put(-121,16){\footnotesize{$\beta_{1,1}$}}
  \put(-113,23){\footnotesize{$\partial B_{1,1}$}}
  \put(-99,22){\footnotesize{$\partial B_{1,2}$}}
  \put(-84,22){\footnotesize{$\partial B_{1,p_1}$}}
  \put(-111,36){\footnotesize{$\partial B_{n,1}$}}
  \put(-86,36){\footnotesize{$\partial B_{n,p_n}$}}
  \put(-96,47){$\EC$}
  \put(-62,31){$F$}
  \put(-27,47){$\EC$}
  \put(-44,5){$0$}
  \put(-41,48){$\infty$}
  \put(-37,16){$\gamma_{1,1}$}
  \put(-38,8){$\gamma_{1,p_1+1}$}
  \put(-21,13){$\gamma_{1,2}$}
  \put(-21,9){$\eta_{1,2}$}
  \put(-5,10){$\gamma_{1,p_1}$}
  \put(-44,23){\footnotesize{$\partial B_{1,1}$}}
  \put(-30,22){\footnotesize{$\partial B_{1,2}$}}
  \put(-15,22){\footnotesize{$\partial B_{1,p_1}$}}
  \put(-42,36){\footnotesize{$\partial B_{n,1}$}}
  \put(-17,36){\footnotesize{$\partial B_{n,p_n}$}}
  \put(-34,42){$\eta_1$}
  \caption{The sketch of the definition of the map $F|_{D(\xi_{1,1})}:D(\xi_{1,1})\to\EC\setminus\overline{D}(\gamma_{1,2})$. For further references, some other equipotentials are drawn also in the super-attracting basins and the boundaries of immediate basins are also marked. For simplicity, we only draw the first super-attracting orbit $B_{1,1}$, $\cdots$, $B_{1,p_1}$.}
  \label{Fig_surgery}
\end{figure}

Let $\xi_{1,1}$ be an equipotential in $B_{1,1}$ such that $L_{1,1}(\xi_{1,1})=s\in (r^{d_1},r)$. Hence $L_{1,1}(\gamma_{1,p_1+1})<L_{1,1}(\xi_{1,1})<L_{1,1}(\gamma_{1,1})$. Recall that $\phi_{1,1}:D(\xi_{1,1})\to\D(0,s)$ is the restriction of the B\"{o}ttcher map. For $m_1\geq 1$, we define $Q_{m_1}(z)=z^{m_1}/s^{m_1}:\D(0,s)\to\D$. Let $\psi_1:\EC\setminus\overline{D}(\gamma_{1,2})\to\D$ be a conformal map such that $\psi_1(\infty)=0$. Define \begin{equation*}
F:=\psi_1^{-1}\circ Q_{m_1}\circ\phi_{1,1}.
\end{equation*}
Then $F:D(\xi_{1,1})\to\EC\setminus\overline{D}(\gamma_{1,2})$ is a holomorphic branched mapping with degree $m_1$, $F(0)=\infty$ and $0$ is the unique possible critical point. Since
\begin{equation}\label{equ-Q-phi-gamma}
Q_{m_1}\circ\phi_{1,1}(\gamma_{1,p_1+1})=\{z\in\C:|z|=r^{d_1 m_1}/s^{m_1}\}\subset\D
\end{equation}
is an Euclidean circle, it follows that $\eta_1:=F(\gamma_{1,p_1+1})=\psi_1^{-1}\{z:|z|=r^{d_1 m_1}/s^{m_1}\}$ is a real-analytic Jordan curve separating $\infty$ from $\gamma_{1,2}$.

We need to find a sufficient condition to guarantee that $\eta_1$ is contained in $B_{n,1}$. Since the degree of the restriction of $f$ on $B_{1,1}$ is $d_{1,1}$ and $a_1=0$, the map $f$ can be written near the origin as
\begin{equation*}
f(z)=f(0)+b_1 z^{d_{1,1}}+O(z^{d_{1,1}+1}),
\end{equation*}
where $b_1\neq 0$ is a constant depending only on $f$. If $r>0$ is sufficiently small, then there exists a constant $C_1>0$ independent of $r$ such that $D(\gamma_{1,2})=f(D(\gamma_{1,1}))$ is a Jordan disk centered at $f(0)$ with radius about $C_1r^{d_{1,1}}$. More specifically, $r$ can be chosen small enough such that \begin{equation}\label{equ:f0-posi}
\D(f(0),C_1r^{d_{1,1}}/2) \subset D(\gamma_{1,2})\subset\D(f(0),2C_1r^{d_{1,1}}).
\end{equation}
Then there exists a constant $C_2>0$ depending on $C_1$ but independent of $r>0$ such that the conformal modulus satisfies
\begin{equation}\label{equ:f0-posi-1}
\Mod(\D\setminus \overline{D}(\psi_1(\partial B_{n,1})))=\Mod(A(\gamma_{1,2},\partial B_{n,1}))\leq \frac{1}{2\pi}\log \frac{1}{r^{d_{1,1}}}+C_2.
\end{equation}

Note that the conformal map $\psi_1:\EC\setminus\overline{D}(\gamma_{1,2})\to\D$ can be written as $\psi_1=\psi_{1,2}\circ\psi_{1,1}$, where
\begin{equation*}
\psi_{1,1}(z)=\frac{1}{z-f(0)} \text{\quad and\quad} \psi_{1,2}:\psi_{1,1}(\EC\setminus\overline{D}(\gamma_{1,2}))\to\D
\end{equation*}
are conformal maps. Both $\psi_1$ and $\psi_{1,2}$ can be extended homeomorphically to the boundaries. Since $\psi_1(\infty)=0$, we have $\psi_{1,2}(0)=0$. Moreover, $\psi_{1,1}(\EC\setminus\overline{D}(\gamma_{1,2}))$ is a bounded Jordan disk in $\C$ containing the Jordan curve $\psi_{1,1}(\partial B_{n,1})$ which surrounds the origin. See Figure \ref{Fig_Koebe}.

\begin{figure}[!htpb]
  \setlength{\unitlength}{1mm}
  \centering
  \includegraphics[width=0.9\textwidth]{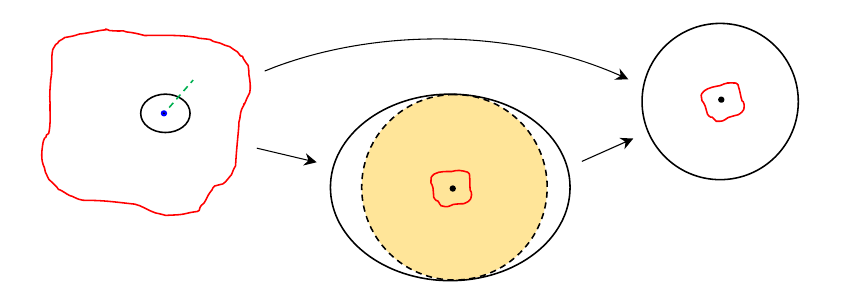}
  \put(-55,38){$\psi_1$}
  \put(-78,23){$\psi_{1,1}$}
  \put(-36,24){$\psi_{1,2}$}
  \put(-98,22){$\gamma_{1,2}$}
  \put(-102,10){$\partial B_{n,1}$}
  \put(-87.5,30){$f(0)$}
  \put(-52.3,15){$0$}
  \put(-15.8,27){$0$}
  \put(-18,13.5){$\D$}
  \put(-41,5){$\psi_{1,1}(\gamma_{1,2})$}
  \put(-61,10){$\psi_{1,1}(\partial B_{n,1})$}
  \caption{The map $\psi_1:\EC\setminus\overline{D}(\gamma_{1,2})\to\D$ is written as $\psi_1=\psi_{1,2}\circ\psi_{1,1}$. To obtain the bounded shape of $\psi_1(\partial B_{n,1})$, we apply Koebe's distortion theorem to $\psi_{1,2}$ in the round disk $\D(0,1/(2C_1 r^{d_{1,1}}))$, which is a subset of $\psi_{1,1}(\EC\setminus\overline{D}(\gamma_{1,2}))$.}
  \label{Fig_Koebe}
\end{figure}

Let $C_2'$ and $C_2''$ be two positive constants such that
\begin{equation*}
C_2'\leq |z-f(0)| \leq C_2'' \text{\quad for all\quad} z\in\partial B_{n,1}.
\end{equation*}
Then we have
\begin{equation*}
\frac{1}{C_2''}\leq |\zeta|\leq \frac{1}{C_2'} \text{\quad for all\quad} \zeta\in \psi_{1,1}(\partial B_{n,1}).
\end{equation*}
By \eqref{equ:f0-posi} we have
\begin{equation*}
\frac{1}{2 C_1 r^{d_{1,1}}}\leq |\zeta|\leq \frac{2}{C_1 r^{d_{1,1}}} \text{\quad for all\quad} \zeta\in \psi_{1,1}(\gamma_{1,2}).
\end{equation*}

Let $r>0$ be small enough such that $2/C_2'\leq 1/(2 C_1 r^{d_{1,1}})$. Then we obtain a univalent map
\begin{equation*}
\psi_{1,2}:\D(0,1/(2C_1 r^{d_{1,1}}))\to\D,  \text{\quad where\quad} \psi_{1,2}(0)=0.
\end{equation*}
By Koebe's distortion theorem (see \cite[Theorem 1.6, p.\,21]{Pom75} or \cite[Theorem 2.6, p.\,33]{Dur83}), there exists a constant $C_3> 1$ independent of small $r>0$ such that
\begin{equation}\label{equ:est-C3}
C_3^{-1}|w_2|\leq |w_1|\leq C_3|w_2| \text{\quad for all\quad} w_1, w_2\in \psi_1(\partial B_{n,1}).
\end{equation}
In particular, $C_3$ can be chosen as
\begin{equation*}
C_3:=\frac{1/C_2'}{(1-\frac{1}{2})^2}\Big/\frac{1/C_2''}{(1+\frac{1}{2})^2}=\frac{16 C_2''}{9 C_2'}.
\end{equation*}

By \eqref{equ:f0-posi-1} and \eqref{equ:est-C3}, it follows that there exists a constant $C_4>0$ independent of small $r>0$ such that for any $w\in\psi_1(\partial B_{n,1})$, we have
\begin{equation*}
\log|w|\geq\log r^{d_{1,1}}-C_4.
\end{equation*}
In order to guarantee that $\eta_1\subset B_{n,1}$, by \eqref{equ-Q-phi-gamma}, it is sufficient to obtain the inequality
\begin{equation*}
\log\frac{r^{d_1 m_1}}{s^{m_1}}\leq\log r^{d_{1,1}}-C_4.
\end{equation*}
Since $d_1$, $d_{1,1}$, $m_1$ are positive integers, $r>0$ can be chosen arbitrarily small and $s\in (r^{d_1},r)$ can be chosen arbitrarily close to $r$, it implies that we only need to guarantee that
\begin{equation*}
\frac{r^{d_1 m_1}}{r^{m_1}}<r^{d_{1,1}}.
\end{equation*}
This is equivalent to $(d_1-1)m_1>d_{1,1}$, i.e. $m_1>d_{1,1}/(d_1-1)$, as desired.
\end{proof}

\begin{rmk}
Since $1\leq d_{i,1}\leq d_i$ and $d_i\geq 2$ for all $1\leq i\leq n-1$, the inequality $m_i>d_{i,1}/(d_i-1)$ is always satisfied if we set $m_i=3$.

If $d_{i,1}=1$ and $d_i=2$ for some $1\leq i\leq n-1$, then $m_i$ can be chosen to be $2$. This observation will allow us to construct a family of quartic rational functions whose Julia sets contain buried components which are Jordan curves. See \S\ref{sec-quartic}.
\end{rmk}

Denote $W:=\EC\setminus \cup_{i=1}^{n-1}D(\gamma_{i,1})$. We define $F|_W:=f|_W$. Then $F|_{\gamma_{i,1}}:\gamma_{i,1}\to\gamma_{i,2}$ is a degree $d_{i,1}$ covering map for all $1\leq i\leq n-1$.

\subsection{Holomorphic covering from an annulus to a disk I}\label{subsec-annu-to-disk}

Note that the map $F$ has been defined on the Riemann sphere except on the annulus $A(\xi_{i,1},\gamma_{i,1})$, where $1\leq i\leq n-1$. In order to use the quasiconformal surgery principle, we need to consider a holomorphic covering from an annulus to a disk.

The first statement of the following lemma was stated in \cite[Lemma A.3]{God15} and a sketch of the proof was given. For completeness, we include a detailed proof here and add a complementary result.

\begin{lema}[{holomorphic covering from annulus to disk}]\label{lema:ann-to-disk-std}
Let $\ell$ and $m$ be two positive integers. Then there exist a constant $r>0$ and a holomorphic branched covering map $\psi:\A_r\to \D$ with degree $\ell+m$ such that
\begin{enumerate}
\item $\psi$ has $\ell+m$ critical points in $\A_r$;
\item $\psi$ can be extended continuously to $\partial\A_r$ by a degree $\ell$ covering $\psi|_{\partial\D}:\partial\D\to\partial\D$ and a degree $m$ covering $\psi|_{\T_r}:\T_r\to\partial\D$; and
\item $\Mod(\A_r)<\tfrac{2}{\pi}\log 2$.
\end{enumerate}
Further, suppose that $\ell=m$ and $K>1$. Then $\psi:\A_r\to \D$ can be chosen such that
\begin{enumerate}\setcounter{enumi}{3}
\item $\psi$ has exactly $2$ critical values at $\pm\, 1/K$;
\item $\Mod(\A_r)<\tfrac{1}{m\pi}\log (2K+1)$; and
\item $\psi^{-1}(\Gamma)\subset\A_r$ is a connected compact set separating $0$ from $\infty$, where $\Gamma\subset\D$ is a connected compact set connecting $1/K$ with $-1/K$.
\end{enumerate}
\end{lema}

\begin{proof}
The proof is based on the study of the dynamical properties of the McMullen maps
\begin{equation*}
g_\lambda(z)=z^\ell+\lambda/z^m, \text{ where } \lambda>0.
\end{equation*}
Except $0$ and $\infty$, it is easy to see that $g_\lambda$ has $\ell+m$ critical points and the critical values of $g_\lambda$ are
\begin{equation}\label{equ-cv}
v_j=\Big(1+\frac{\ell}{m}\Big)\Big(\frac{m}{\ell}\lambda\Big)^{\frac{\ell}{\ell+m}}e^{\frac{\ell j}{\ell+m}2\pi\ii}, \text{ where }1\leq j\leq \ell+m.
\end{equation}
Note that $v_j$ may be equal to $v_k$ if $1\leq j\neq k\leq \ell+m$. Let $v:=v_{\ell+m}>0$ be the modulus of each $v_j$, where $1\leq j\leq \ell+m$.

Let $K>1$ be a constant. By Riemann-Hurwitz's formula, it follows that $g_\lambda^{-1}(\EC\setminus\overline{\D}(0,Kv))$ consists of two Jordan disks $V_0$ and $V_\infty$ such that
\begin{enumerate}
\item $0\in V_0$ and $\infty\in V_\infty$;
\item $g_\lambda|_{\partial V_\infty}:\partial V_\infty\to \T_{Kv}$ is a covering map with degree $\ell$; and
\item $g_\lambda|_{\partial V_0}:\partial V_0\to \T_{Kv}$ is a covering map with degree $m$.
\end{enumerate}
Therefore, $g_\lambda:A(\partial V_0,\partial V_\infty)\to\D(0,Kv)$ is a holomorphic branched covering map with degree $\ell+m$.
For all $\lambda>0$, then it is easy to check that
\begin{equation*}
|z|< \Big(\frac{m}{\ell}\lambda\Big)^{\frac{1}{\ell+m}}\Big(K+\frac{(K+1)\ell}{m}\Big)^{\frac{1}{\ell}} \text{ for every } z\in\partial V_\infty,
\end{equation*}
and
\begin{equation*}
|z|> \Big(\frac{m}{\ell}\lambda\Big)^{\frac{1}{\ell+m}}\Big(\frac{\ell}{m}\Big)^{\frac{1}{m}}\Big(K+1+\frac{K\ell}{m}\Big)^{-\frac{1}{m}} \text{ for every } z\in\partial V_0.
\end{equation*}
Since $x\mapsto \tfrac{1}{x}\log(K+(K+1)x)$ is decreasing on $[1,+\infty)$ for all $K>1$, this implies that
\begin{equation*}
\begin{split}
&~2\pi\,\Mod(A(\partial V_0,\partial V_\infty))\\
< &~ \frac{1}{\ell}\log\Big(K+\frac{(K+1)\ell}{m}\Big)+\frac{1}{m}\log\Big(K+\frac{(K+1)m}{\ell}\Big)\\
\leq &~ \frac{1}{\ell}\log\big(K+(K+1)\ell\big)+\frac{1}{m}\log\big(K+(K+1)m\big)\leq 2\log(2K+1).
\end{split}
\end{equation*}
In particular, if $K=3/2$, then we have $\Mod(A(\partial V_0,\partial V_\infty))< \tfrac{2}{\pi}\log 2$. Moreover, if $\ell=m\geq 1$, then
\begin{equation*}
2\pi\,\Mod(A(\partial V_0,\partial V_\infty))< \frac{2}{m}\log(2K+1).
\end{equation*}

Let $0<r<1$ be the number such that $\Mod(\A_r)=\Mod(A(\partial V_0,\partial V_\infty))$. Then there exist two conformal mappings $\psi_1:\A_r\to A(\partial V_0,\partial V_\infty)$ and $\psi_2(z)=z/(Kv):\D(0,Kv)\to\D$. Then $\psi:=\psi_2\circ g_\lambda\circ\psi_1:\A_r\to\D$ is the required holomorphic function. Indeed, the statement (f) holds since $g_\lambda^{-1}([-v,v])\subset A(\partial V_0,\partial V_\infty)$ is a connected curve separating $0$ from $\infty$.
\end{proof}

\begin{rmk}
Lemma \ref{lema:ann-to-disk-std} is similar to the key lemma in \cite[Lemma 2.1]{PT99} (see also \cite[Lemma 7.47]{BF14}) about the annulus-to-disk branched coverings. However, here the covering from annulus to disk is required to be holomorphic. We will use the properties (d), (e) and (f) of Lemma \ref{lema:ann-to-disk-std} in \S\ref{sec-dendrite}.
\end{rmk}

Let us continue the construction of $F$. We will use the properties (a), (b) and (c) of Lemma \ref{lema:ann-to-disk-std} to prove the following result.

\begin{lema}\label{lema:annu-disk}
For each $1\leq i\leq n-1$, there exist two equipotentials $\alpha_{i,1}$, $\beta_{i,1}$ in $B_{i,1}$, an equipotential $\eta_{i,2}$ in $B_{i,2}$ and a holomorphic branched covering map $F|_{A(\alpha_{i,1},\beta_{i,1})}:A(\alpha_{i,1},\beta_{i,1})\to D(\eta_{i,2})$ with degree $d_{i,1}+m_i$ such that
\begin{enumerate}
\item $L_{i,1}(\xi_{i,1})<L_{i,1}(\alpha_{i,1})<L_{i,1}(\beta_{i,1})<L_{i,1}(\gamma_{i,1})$ and $L_{i,2}(\eta_{i,2})<L_{i,2}(\gamma_{i,2})$;
\item $F|_{A(\alpha_{i,1},\beta_{i,1})}$ has $d_{i,1}+m_i$ critical points in $A(\alpha_{i,1},\beta_{i,1})$; and
\item $F|_{A(\alpha_{i,1},\beta_{i,1})}$ can be extended continuously to $\alpha_{i,1}\cup\beta_{i,1}$ by a degree $m_i$ covering $F|_{\alpha_{i,1}}:\alpha_{i,1}\to\eta_{i,2}$ and a degree $d_{i,1}$ covering $F|_{\beta_{i,1}}:\beta_{i,1}\to\eta_{i,2}$.
\end{enumerate}
\end{lema}

\begin{proof}
By the first half result of Lemma \ref{lema:ann-to-disk-std}, it is sufficient to prove the existence of $\alpha_{i,1}$ and $\beta_{i,1}$ in $B_{i,1}$ such that $L_{i,1}(\xi_{i,1})<L_{i,1}(\alpha_{i,1})<L_{i,1}(\beta_{i,1})<L_{i,1}(\gamma_{i,1})$ and $\Mod (A(\xi_{i,1},\gamma_{i,1}))\geq \tfrac{2}{\pi}\log 2$.
See Figure \ref{Fig_surgery}. This is obvious if we choose the level $r>0$ of $\gamma_{i,1}$ in the proof of Lemma \ref{lema:cover-disk-to-disk} small enough.
\end{proof}

Now $F$ is defined on the Riemann sphere except on the annuli $A(\xi_{i,1},\alpha_{i,1})$ and $A(\beta_{i,1},\gamma_{i,1})$, where $1\leq i\leq n-1$. Since all of the connected components of the boundaries of these annuli, together with their images $\gamma_{i,2}$ and $\eta_{i,2}$, are quasicircles (actually are analytic curves), one can make an interpolation such that the resulting map $F$ satisfies
\begin{enumerate}
\item $F|_{A(\xi_{i,1},\alpha_{i,1})}:A(\xi_{i,1},\alpha_{i,1})\to A(\gamma_{i,2},\eta_{i,2})$ is a degree $m_i$ covering map;
\item $F|_{A(\beta_{i,1},\gamma_{i,1})}:A(\beta_{i,1},\gamma_{i,1})\to A(\eta_{i,2},\gamma_{i,2})$ is a degree $d_{i,1}$ covering map; and
\item $F|_{A(\xi_{i,1},\alpha_{i,1})}$ and $F|_{A(\beta_{i,1},\gamma_{i,1})}$ are local quasiconformal.
\end{enumerate}
In particular, for quasiconformal interpolation in annuli, we refer to \cite[\S 2.3.2]{BF14}.

\begin{rmk}
As mentioned in \S\ref{subsec-equi}, we modify the order of the curves defining an annulus following the dynamics. Specifically, if we mention a covering map $F:A(\gamma_1,\gamma_2)\to A(\eta_1,\eta_2)$ between two annuli whose boundaries are Jordan curves, this means that $F$ has a continuous extension which maps $\gamma_1$ onto $\eta_1$ and $\gamma_2$ onto $\eta_2$. Such setting will be also adopted in the rest of this article.
\end{rmk}

\subsection{Uniformization I}

Now we have a quasiregular map $F$ defined from the Riemann sphere to itself whose dynamics is sketched in Figure \ref{Fig_surgery}. We need to find a quasiconformal homeomorphism to conjugate $F$ to a rational map. For this, we will apply Shishikura's fundamental lemma for quasiconformal surgery.

\begin{lema}[{Fundamental lemma for qc surgery, \cite{Shi87}}]\label{lema:qc}
Let $G:\EC\to\EC$ be a quasiregular map. Suppose that there exist an open set $E\subset\EC$ and an integer $N\geq 0$ satisfying the following two conditions:
\begin{enumerate}
\item $G(E)\subset E$; and
\item $\partial G/\partial \overline{z}=0$ holds on $E$ and a.e. on $\EC\setminus G^{-N}(E)$.
\end{enumerate}
Then there exists a quasiconformal map $\varphi:\EC\to\EC$ such that $\varphi\circ G\circ \varphi^{-1}$ is rational.
\end{lema}

The lemma above established firstly by Shishikura although its original statement is more general. The reader can refer to \cite[\S 3]{Shi87}, \cite[Lemma 9.6.2]{Bea91a} and \cite[Proposition 5.2]{BF14} for a proof and more details.

\begin{cor}\label{cor:home}
There is a quasiconformal map $\varphi_1:\EC\to\EC$ such that
\begin{equation*}
h:=\varphi_1\circ F\circ \varphi_1^{-1}
\end{equation*}
is rational map satisfying $\varphi_1(0)=0$, $\varphi_1(\infty)=\infty$ and $\deg(h)=d+\sum_{i=1}^{n-1}m_i$.
\end{cor}

\begin{proof}
Recall the definition of the equipotential $\gamma_{i,j}$ in \S\ref{subsec-equi} and Lemma \ref{lema:cover-disk-to-disk}, where $1\leq i\leq n-1$ and $1\leq j\leq p_i+1$.
For the quasiregular map $F$, we define an open set
\begin{equation}\label{equ-E-1}
E_1:=\bigcup_{j=1}^{p_n}B_{n,j}\cup\bigcup_{i=1}^{n-1}\bigcup_{j=2}^{p_i+1}D(\gamma_{i,j}).
\end{equation}
According to Lemma \ref{lema:cover-disk-to-disk}, we have $F(D(\gamma_{i,p_i+1}))\subset B_{n,1}$ for all $1\leq i\leq n-1$. This implies that $F(E_1)\subset E_1$ since $F(D(\gamma_{i,j}))=D(\gamma_{i,j+1})$ for $1\leq i\leq n-1$ and $2\leq j\leq p_i$. On the other hand, $A(\xi_{i,1},\alpha_{i,1})\cup A(\beta_{i,1},\gamma_{i,1})$ is contained in $F^{-1}(D(\gamma_{i,2}))$ and $F$ is analytic except on the annuli $A(\xi_{i,1},\alpha_{i,1})$ and $A(\beta_{i,1},\gamma_{i,1})$, where $1\leq i\leq n-1$. Therefore, we have $\partial F/\partial \overline{z}=0$ on $E_1$ and a.e. on $\EC\setminus F^{-1}(E_1)$.
The result then follows immediately by Lemma \ref{lema:qc}.

To conclude the proof it is sufficient to verify the statement on the degree of $h$. Note that $\deg(f)=d$ and $a_n=\infty$ has $d$ preimages (counted with multiplicity) which are not contained in the attracting basins of $a_i$ with $1\leq i\leq n-1$. By Lemma \ref{lema:cover-disk-to-disk}, $\infty$ has a unique preimage $a_i$ (with multiplicity $m_i$) in $O_i$ under the quasiregular map $F$, where $1\leq i\leq n-1$, and $\infty$ has no other preimages by Lemma \ref{lema:annu-disk}. This implies that $\deg(h)=\deg(F)=d+\sum_{i=1}^{n-1}m_i$.
\end{proof}

\subsection{The semi-buried property}

By the construction of surgery, $h$ has a cycle of super-attracting periodic Fatou components $\varphi_1(O_n)=\{\varphi_1(B_{n,j}):1\leq j\leq p_n\}$. Recall that $\MA_n$ is the super-attracting basin of $f$ containing $O_n$. In this subsection we show that the boundary of $\varphi_1(\MA_n)$ is ``semi-buried".

\begin{defi}[{Semi-buried Julia components}]
Let $J_0\neq\EC$ be a Julia component of a rational map $R$ and $U$ a connected component of $\EC\setminus J_0$. If $J_0$ is disjoint with the boundary of any Fatou component of $R$ in $U$, then $J_0$ is called \textit{semi-buried} from $U$. In particular, $J_0$ is a buried Julia component if and only if it is semi-buried from any component of $\EC\setminus J_0$.
\end{defi}

Let $\MV:=F(f)\setminus \MA_n$, where $F(f)$ is the Fatou set of $f$. Since we have assumed that $n\geq 2$ in this section, it implies that $\MV\neq\emptyset$. Recall that $\varphi_1:\EC\to\EC$ is the quasiconformal mapping introduced in Corollary \ref{cor:home}.

\begin{prop}\label{prop-homeo}
The set $\varphi_1(J(f))$ is a semi-buried Julia component of $h$. In particular, $\varphi_1(J(f))$ is semi-buried from every component of $\varphi_1(\MV)$.
\end{prop}

\begin{proof}
We first show that $J':=\varphi_1(J(f))$ is contained in the Julia set of $h$. Indeed, for any $z_0\in J'$ and any open neighborhood $U$ of $z_0$, there is a point $z_1\in U\cap\varphi_1(\MA_n)$ since $J(f)=\partial\MA_n$. Since $z_1$ will be iterated into the periodic orbit $O_n$ eventually while the orbit of $z_0$ is contained in $J'$, it implies that $\{h^{\circ k}\}_{k\in\N}$ is not equi-continuous in $U$. Therefore, $z_0$ is contained in the Julia set of $h$ and hence $J'\subset J(h)$.

Next we prove that for each component $V$ of $\varphi_1(\MV)$, there exists a sequence of Julia components $\{J_k\}_{k\in\N}$ of $h$ in $V$ which converges to $\partial V$ in the Hausdorff metric. Without loss of generality, we only consider the case that $V\subset \varphi_1(\MA_i)$ for $i=1$ since the remaining case is completely similar. By Lemma \ref{lema:cover-disk-to-disk}, $F(D(\xi_{1,1})\setminus\{0\})=\C\setminus\overline{D}(\gamma_{1,2})$ is a covering map with degree $m_1\geq 1$ and $F(\gamma_{1,p_1+1})=\eta_1\subset B_{n,1}$, it follows that the annulus $\varphi_1(A(\gamma_{1,p_1+1},\xi_{1,1}))$ contains a Julia component $J_0$ of $h$ separating $0$ from $\infty$. On the other hand, by the surgery construction in \S\ref{subsec-annu-to-disk}, it follows that the annulus $\varphi_1(A(\xi_{1,1},\gamma_{1,1}))$ is contained in a Fatou component $U_0$ of $h$ and it separates $0$ from $\infty$ also.

Since $F^{\circ p_1}=f^{\circ p_1}:A(\gamma_{1,1},\partial B_{1,1})\to A(\gamma_{1,p_1+1},\partial B_{1,1})$ is a covering map with degree $d_1$, it follows that the annulus $\varphi_1(A(\gamma_{1,1},\partial B_{1,1}))$ contains a Julia component $J_1:=h^{-p_1}(J_0)\cap \varphi_1(A(\gamma_{1,1},\partial B_{1,1}))$ and a Fatou component $U_1:=h^{-p_1}(U_0)\cap \varphi_1(A(\gamma_{1,1},\partial B_{1,1}))$ such that both $J_1$ and $U_1$ separate $0$ from $\infty$. Inductively, one can obtain a sequence of Julia components $\{J_k\}_{k\geq 1}$ and a sequence of Fatou components $\{U_k\}_{k\geq 1}$ in $\varphi_1(A(\gamma_{1,1},\partial B_{1,1}))$ such that $h^{\circ p_1}(J_k)=J_{k-1}$, $h^{\circ p_1}(U_k)=U_{k-1}$ and each $J_k$ and $U_k$ separates $0$ from $\infty$.

Recall that the levels of $\gamma_{1,1}$ and $\gamma_{1,p_1+1}$ in $B_{1,1}$ are $r$ and $r^{d_1}$ respectively, where $r>0$ is small enough. This implies that for every $z\in \varphi_1^{-1}(J_k\cup U_k)$, the level of $z$ in $B_{1,1}$ satisfies
\begin{equation*}
1>L_{1,1}(z)=|\phi_{1,1}(z)|\geq\sqrt[d_1^{k-1}]{r}, \text{ where }k\in\N.
\end{equation*}
Since both $\varphi_1^{-1}(J_k)$ and $\varphi^{-1}(U_k)$ separate $0$ from $\infty$, it implies that the Hausdorff distance between $\partial B_{1,1}$ and $\varphi_1^{-1}(J_k)$ (resp. $\varphi^{-1}(U_k)$) tends to zero as $k\to\infty$. Equivalently, the sequence of Julia components $\{J_k\}_{k\in\N}$ of $h$ in $\varphi_1(B_{1,1})$ converges to $\varphi_1(\partial B_{1,1})$ in the Hausdorff metric.

Now we show that $J'=\varphi_1(J(f))$ is a Julia component of $h$. Let $J''$ be the Julia component of $h$ containing $J'$. Suppose that there exists a point $z_0\in J''\setminus J'$. Then we have $z_0\in\varphi_1(\MV)$. By iterating $z_0$ several times if necessary, we assume that $z_0\in\varphi_1(B_{1,1})$. However this is a contradiction since $\varphi_1(B_{1,1})$ contains a sequence of Julia components $\{J_k\}_{k\in\N}$ which converges to $\varphi_1(\partial B_{1,1})$ in the Hausdorff metric. This implies that $z_0$ does not exist and $J'$ is a semi-buried Julia component of $h$.

As a connected subset of the Julia component $J'$, we know that $\varphi_1(\partial B_{1,1})$ is semi-buried from $\varphi_1(B_{1,1})$. By considering the preimages of $\varphi_1(B_{1,1})$ under $h$, it follows that $\partial V$ is semi-buried from $V$ for every component $V$ of $\varphi_1(\MA_1)$. In particular, $\varphi_1(J(f))$ is a semi-buried Julia component of $h$.
\end{proof}

\begin{rmk}
By the construction of $F$, it follows that $f|_{J(f)}:J(f)\to J(f)$ is conjugate to $h|_{\varphi_1(J(f))}:\varphi_1(J(f))\to\varphi_1(J(f))$ by a restriction of a quasiconformal mapping, where $\varphi_1(J(f))$ is a semi-buried Julia component of $h$.
\end{rmk}

\section{From semi-buried to buried}\label{sec-second-pert}

In this section, we perform the surgery of the second stage on $f$ based on $h$ and prove that the semi-buried Julia component can be transferred to a fully buried component. According to Proposition \ref{prop-homeo}, $\varphi_1(J(f))$ is semi-buried from every component of $\varphi_1(\MV)$. Hence a natural idea is to perform the surgery in the immediate super-attracting basins $\varphi_1(O_n)$. Note that the properties of $h$ are quite different from that of $f$ since $h$ is no longer post-critically finite in the Fatou set and the Julia set of $h$ is not connected.

By definition we know that $\varphi_1(E_1)$ is contained in the super-attracting basin of $\infty$, where $E_1$ is defined in \eqref{equ-E-1}. For $1\leq i\leq n-1$ and $1\leq j\leq p_i$, we use $U_{i,j}$ to denote the Fatou component of $h$ containing $\varphi_1(f^{\circ (j-1)}(a_i))\in\varphi_1(B_{i,j})$. For saving the notations, we still use $O_n$ and $B_{n,j}$ etc, respectively, to denote the super-attracting periodic orbit $\varphi_1(O_n)$ and the immediate super-attracting basin $\varphi_1(B_{n,j})\ni h^{\circ (j-1)}(\infty)$ of $h$, etc (i.e. the quasiconformal prefix $\varphi_1$ is omitted).

\subsection{Cutting along the equipotentials II}\label{subsec-equi-2}

According to B\"{o}ttcher's theorem, the super-attracting periodic orbit $O_n$ provides the Riemann mappings $\phi_{n,j}:B_{n,j}\to\EC\setminus\overline{\D}$, where $1\leq j\leq p_n$, such that $\phi_{n,j}(h^{\circ (j-1)}(\infty))=\infty$ and the following diagram commutes:
$$\begin{CD}
B_{n,1} @>h>> B_{n,2} @>h>> \cdots  @>h>> B_{n,p_n} @>h>> B_{n,1}\\
@VV\phi_{n,1}V   @VV\phi_{n,2}V        @VVV @VV\phi_{n,p_n}V  @VV\phi_{n,1}V\\
\EC\setminus\overline{\D} @>z\mapsto z^{d_{n,1}}>> \EC\setminus\overline{\D} @>z\mapsto z^{d_{n,2}}>> \cdots  @>z\mapsto z^{d_{n,p_n-1}}>> \EC\setminus\overline{\D} @>z\mapsto z^{d_{n,p_n}}>> \EC\setminus\overline{\D},
\end{CD}$$
where $d_{n,p_n}\geq 2$ and $d_{n,j}\geq 1$ are positive integers for $1\leq j\leq p_n-1$. For convenience, we denote $d_n:=\prod_{j=1}^{p_n}d_{n,j}\geq 2$.
Hence $h^{\circ p_n}:B_{n,j}\to B_{n,j}$ is holomorphically conjugate to $z\mapsto z^{d_n}$ from $\EC\setminus\overline{\D}$ to itself for $1\leq j\leq p_n$.

An \textit{equipotential} $\gamma$ in $B_{n,j}$ is the preimage by $\phi_{n,j}$ of an Euclidean circle in $\EC\setminus\overline{\D}$ centered at 0. The radius of this circle is called the \textit{level} of $\gamma$ and is denoted by $L_{n,j}(\gamma)\in(1,+\infty)$. The \textit{level} of a point $z\in B_{n,j}$ (is contained in some equipotential curve) can be defined similarly as $L_{n,j}(z):=|\phi_{n,j}(z)|\in(1,+\infty)$. Recall that $m_1\geq 1$ is an integer introduced in Lemma \ref{lema:cover-disk-to-disk}.

\begin{lema}[{holomorphic covering from disk to disk II}]\label{lema:cover-disk-to-disk-2}
Let $\ell\geq 1$ be an integer satisfying $\ell>(d_{n,1}+\tfrac{d_n}{m_1})/(d_n-1)$. Then there exist equipotentials\footnote{If $p_n=1$, then $B_{n,2}=B_{n,1}$ and $\gamma_{n,2}$ is regarded as an equipotential in $B_{n,1}$.} $\gamma_{n,1}$, $\gamma_{n,p_n+1}$, $\xi_{n,1}\subset B_{n,1}$, $\gamma_{n,2}\subset B_{n,2}$ and a holomorphic branched mapping $H|_{\EC\setminus \overline{D}(\xi_{n,1})}:\EC\setminus \overline{D}(\xi_{n,1})\to\EC\setminus\overline{D}(\gamma_{n,2})$ (if $p_n\geq 2$) or\footnote{Recall that $D(\gamma_{n,2})$ is the connected component of $\EC\setminus\gamma_{n,2}$ which does not contain $\infty$. If $p_n\geq 2$, then $D(\gamma_{n,2})$ is a Jordan disk which contains neither $\infty$ nor $0$. See Figure \ref{Fig_surgery-2}.} $H|_{\EC\setminus \overline{D}(\xi_{n,1})}:\EC\setminus \overline{D}(\xi_{n,1})\to D(\gamma_{n,2})$ (if $p_n=1$) satisfying the following conditions:
\begin{enumerate}
\item $H(\infty)=0$ and $H:\C\setminus \overline{D}(\xi_{n,1})\to\EC\setminus(\overline{D}(\gamma_{n,2})\cup\{0\})$ (or $H:\C\setminus \overline{D}(\xi_{n,1})\to D(\gamma_{n,2})\setminus\{0\}$) is a degree $\ell$ covering map;
\item $H(\xi_{n,1})=\gamma_{n,2}$ and $H(\gamma_{n,p_n+1})=\eta_{1,1}$, where $\eta_{1,1}$ is a real-analytic Jordan curve in $U_{1,1}$ separating $0$ from $\partial U_{1,1}$;
\item $h$ maps $\eta_{1,1}$ to a curve $\kappa_{n,1}$; and
\item $L_{n,1}(\gamma_{n,1})<L_{n,1}(\xi_{n,1})<L_{n,1}(\gamma_{n,p_n+1})<L_{n,1}(z)$ for all $z\in\kappa_{n,1}$.
\end{enumerate}
\end{lema}

\begin{proof}
The proof is a bit similar to that of Lemma \ref{lema:cover-disk-to-disk}. Without loss of generality, we assume that $p_n\geq 2$ since the proof for $p_n=1$ is similar.
For each large $R\in(1,+\infty)$, let $\gamma_{n,1}$, $\gamma_{n,p_n+1}$ be the equipotentials in $B_{n,1}$ such that $L_{n,1}(\gamma_{n,1})=R$ and $L_{n,1}(\gamma_{n,p_n+1})=R^{d_n}$. For $2\leq j\leq p_n$, we denote $\gamma_{n,j}:=h^{\circ (j-1)}(\gamma_{n,1})$. Then $\gamma_{n,j}$ is an equipotential in $B_{n,j}$ for all $1\leq j\leq p_n$. See Figure \ref{Fig_surgery-2} for a partial illustration.

\begin{figure}[!htpb]
  \setlength{\unitlength}{1mm}
  \centering
  \includegraphics[width=0.95\textwidth]{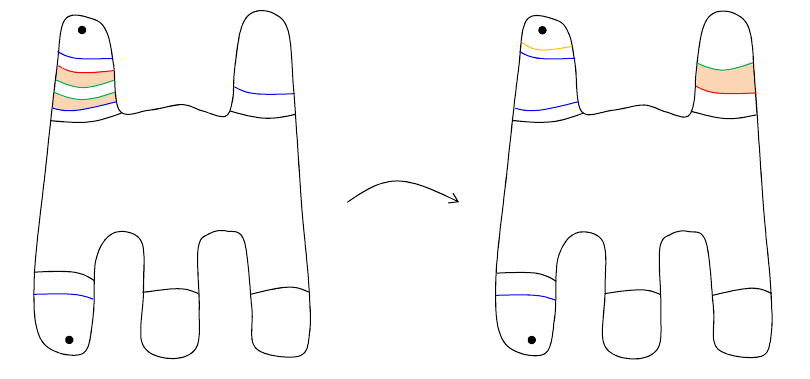}
  \put(-113,5){$0$}
  \put(-110,49){$\infty$}
  \put(-103,47){$\gamma_{n,p_n+1}$}
  \put(-102,41){$\gamma_{n,1}$}
  \put(-76,42){$\gamma_{n,2}$}
  \put(-121,12){$\eta_{1,1}$}
  \put(-119,42){\footnotesize{$\beta_{n,1}$}}
  \put(-118,44){\footnotesize{$\alpha_{n,1}$}}
  \put(-118,46){\footnotesize{$\xi_{n,1}$}}
  \put(-114,17){\footnotesize{$\partial U_{1,1}$}}
  \put(-99,15){\footnotesize{$\partial U_{1,2}$}}
  \put(-83,15){\footnotesize{$\partial U_{1,p_1}$}}
  \put(-111,35){\footnotesize{$\partial B_{n,1}$}}
  \put(-86,36){\footnotesize{$\partial B_{n,2}$}}
  \put(-96,50){$\EC$}
  \put(-62,31){$H$}
  \put(-27,50){$\EC$}
  \put(-44,5){$0$}
  \put(-41,49){$\infty$}
  \put(-49,50){$\kappa_{n,1}$}
  \put(-35,47){$\gamma_{n,p_n+1}$}
  \put(-34,41){$\gamma_{n,1}$}
  \put(-7,42){$\gamma_{n,2}$}
  \put(-7,46){$\eta_{n,2}$}
  \put(-53,12){$\eta_{1,1}$}
  \put(-45,17){\footnotesize{$\partial U_{1,1}$}}
  \put(-30,15){\footnotesize{$\partial U_{1,2}$}}
  \put(-14,15){\footnotesize{$\partial U_{1,p_1}$}}
  \put(-42,35){\footnotesize{$\partial B_{n,1}$}}
  \put(-17,36){\footnotesize{$\partial B_{n,2}$}}
  \caption{The sketch of the definition of $H|_{\EC\setminus \overline{D}(\xi_{n,1})}:\EC\setminus \overline{D}(\xi_{n,1})\to\EC\setminus\overline{D}(\gamma_{n,2})$, where $p_n\geq 2$. For further references, some other equipotentials are drawn also in the super-attracting basins of $\infty$. The boundaries of immediate basins of the cycle $O_n$ and parts of their first preimages are marked also.}
  \label{Fig_surgery-2}
\end{figure}

Let $\xi_{n,1}$ be an equipotential in $B_{n,1}$ such that $L_{n,1}(\xi_{n,1})=S\in (R,R^{d_n})$. Recall that $\phi_{n,1}:\EC\setminus \overline{D}(\xi_{n,1})\to\EC\setminus\overline{\D}(0,S)$ is the restriction of the B\"{o}ttcher map. For $\ell\geq 1$, we define $Q_{\ell}(z)=z^{\ell}/S^{\ell}:\EC\setminus\overline{\D}(0,S)\to\EC\setminus\overline{\D}$. Let $\psi_n:\EC\setminus\overline{D}(\gamma_{n,2})\to\EC\setminus\overline{\D}$ be a conformal map such that $\psi_n(0)=\infty$. Define
\begin{equation*}
H:=\psi_n^{-1}\circ Q_{\ell}\circ\phi_{n,1}.
\end{equation*}
Then $H:\EC\setminus \overline{D}(\xi_{n,1})\to\EC\setminus\overline{D}(\gamma_{n,2})$ is a holomorphic branched mapping with degree $\ell$, $H(\infty)=0$ and $\infty$ is the unique critical point. Since
\begin{equation}\label{equ-Q-phi-gamma-1}
Q_{\ell}\circ\phi_{n,1}(\gamma_{n,p_n+1})=\{z\in\C:|z|=R^{d_n \ell}/S^{\ell}\}\subset\EC\setminus\overline{\D}
\end{equation}
is an Euclidean circle, it follows that $\eta_{1,1}:=H(\gamma_{n,p_n+1})=\psi_n^{-1}(\{z:|z|=R^{d_n \ell}/S^{\ell}\})$ is a real-analytic Jordan curve separating $0$ from $\gamma_{n,2}$.

We need to find a sufficient condition to guarantee that $\eta_{1,1}$ is contained in $U_{1,1}$ and $L_{n,1}(\gamma_{n,p_n+1})<L_{n,1}(z)$, where $z\in\kappa_{n,1}:=h(\eta_{1,1})$. Note that $\kappa_{n,1}$ is not necessarily an equipotential in $B_{n,1}$. Since the degree of the restriction of $h$ on $B_{n,1}$ is $d_{n,1}$, the map $h$ can be written near the infinity as
\begin{equation*}
h(z)=h(\infty)+b_n/z^{d_{n,1}}+O(1/z^{d_{n,1}+1}),
\end{equation*}
where $b_n\neq 0$ is a constant depending only on $h$. If $R>1$ is large enough, then $D(\gamma_{n,2})=h(\EC\setminus \overline{D}(\gamma_{n,1}))$ is a Jordan disk centered at $h(\infty)$ with radius about $C_1/R^{d_{n,1}}$, where $C_1>0$ is a constant independent of $R$. More specifically, $R$ can be chosen large enough such that $\D(h(\infty),C_1/ (2R^{d_{n,1}})) \subset D(\gamma_{n,2})\subset\D(h(\infty),2C_1/R^{d_{n,1}})$.

Let $\widetilde{\eta}_{1,1}\subset U_{1,1}$ be a sufficiently small round circle separating $0$ from $\partial U_{1,1}$ with radius $r>0$. There exists a constant $C_2>0$ depending on $C_1$ but independent of the large $R>1$ and small $r>0$ such that the conformal modulus satisfies
\begin{equation*}
\Mod(D(\psi_n(\widetilde{\eta}_{1,1}))\setminus\overline{\D})=\Mod(A(\widetilde{\eta}_{1,1},\gamma_{n,2}))\leq \frac{1}{2\pi}\log \frac{R^{d_{n,1}}}{r}+C_2.
\end{equation*}
Similar to the argument in the proof of Lemma \ref{lema:cover-disk-to-disk}, by Koebe's distortion theorem, there exists a constant $C_3\geq 1$ independent of large $R>1$ and small $r>0$ such that
\begin{equation*}
C_3^{-1}|w_2|\leq |w_1|\leq C_3|w_2| \text{ for all } w_1, w_2\in \psi_n(\widetilde{\eta}_{1,1}).
\end{equation*}
This implies that there exists a constant $C_4>0$ independent of $R>0$ such that for any $w\in\psi_n(\widetilde{\eta}_{1,1})$, we have
\begin{equation*}
\log|w|\leq\log \frac{R^{d_{n,1}}}{r}+C_4.
\end{equation*}
In order to guarantee that $\eta_{1,1}=H(\gamma_{n,p_n+1})\subset U_{1,1}$, by \eqref{equ-Q-phi-gamma-1}, it is sufficient to obtain the inequality
\begin{equation}\label{equ-R-r-1}
\log\frac{R^{d_n\ell}}{S^\ell}\geq\log \frac{R^{d_{n,1}}}{r}+C_4.
\end{equation}

Since the local degree of $h$ at $0$ is $m_1$ and $h(0)=\infty$, it implies that $h$ can be written near the origin as
\begin{equation*}
1/h(z)=b_0 z^{m_1}+O(z^{m_1+1}),
\end{equation*}
where $b_0\neq 0$ is a constant depending only on $h$.
In order to guarantee that $L_{n,1}(\gamma_{n,p_n+1})<L_{n,1}(z)$ for all $z\in\kappa_{n,1}=h(\eta_{1,1})$, it is sufficient to obtain the inequality
\begin{equation}\label{equ-R-r-2}
\log\frac{1}{r^{m_1}}>\log R^{d_n}+C_5,
\end{equation}
where $C_5>0$ is a constant depending on $h$ but independent of large $R$ and small $r$.
Note that $d_n$, $d_{n,1}$, $m_1$, $\ell$ are positive integers and $S\in (R,R^{d_n})$. Since $R>1$ can be arbitrarily large (and hence $r>0$ should be sufficiently small) and $S$ can be arbitrarily close to $R$, by \eqref{equ-R-r-1} and \eqref{equ-R-r-2}, it is sufficient to guarantee that
\begin{equation*}
\frac{R^{d_n \ell}}{R^{\ell}}>\frac{R^{d_{n,1}}}{r} \quad\text{and}\quad \frac{1}{r}>R^{\frac{d_n}{m_1}}.
\end{equation*}
This is equivalent to $(d_n-1)\ell-d_{n,1}>\tfrac{d_n}{m_1}$, i.e. $\ell>(d_{n,1}+\tfrac{d_n}{m_1})/(d_n-1)$, as desired.
\end{proof}

\begin{rmk}
Since $d_{n,1}\leq d_n$, $d_n\geq 2$ and $m_1\geq 1$, the inequality $\ell>(d_{n,1}+\tfrac{d_n}{m_1})/(d_n-1)$ is always satisfied if we set $\ell=5$. Moreover, $\ell$ can be chosen as $3$ if $m_1\geq 3$.

One can obtain a similar result as Lemma \ref{lema:cover-disk-to-disk-2} if $\ell$ is chosen such that $\ell>(d_{n,1}+\tfrac{d_n}{m_i})/(d_n-1)$ for some $1\leq i\leq n-1$.
\end{rmk}

Denote $W:=D(\gamma_{n,1})$. We define $H|_W:=h|_W$. Then $H|_{\gamma_{n,1}}:\gamma_{n,1}\to\gamma_{n,2}$ is a degree $d_{n,1}$ covering map and $H|_W$ maps $W$ onto $\EC$.

\subsection{Annulus-to-disk and annulus-to-annulus coverings II}\label{subsec-annu-to-disk-2}

Similar to the construction of $F$, we also need to construct some holomorphic branched covering maps from annulus to disk and quasiregular covering maps from annulus to annulus for $H$. Similar to Lemma \ref{lema:annu-disk}, the following result is an immediate corollary of Lemma \ref{lema:ann-to-disk-std}. See Figure \ref{Fig_surgery-2}.

\begin{lema}\label{lema:annu-disk-2}
There exist two equipotentials $\alpha_{n,1}$, $\beta_{n,1}$ in $B_{n,1}$, an equipotential $\eta_{n,2}$ in $B_{n,2}$ and a holomorphic branched covering map $H|_{A(\alpha_{n,1},\beta_{n,1})}:A(\alpha_{n,1},\beta_{n,1})\to D(\eta_{n,2})$ (if $p_n\geq 2$) or $H|_{A(\alpha_{n,1},\beta_{n,1})}:A(\alpha_{n,1},\beta_{n,1})\to \EC\setminus\overline{D}(\eta_{n,2})$ (if $p_n=1$) with degree $d_{n,1}+\ell$ such that
\begin{enumerate}
\item $L_{n,1}(\xi_{n,1})>L_{n,1}(\alpha_{n,1})>L_{n,1}(\beta_{n,1})>L_{n,1}(\gamma_{n,1})$, $L_{n,2}(\eta_{n,2})>L_{n,2}(\gamma_{n,2})$;
\item $H|_{A(\alpha_{n,1},\beta_{n,1})}$ has $d_{n,1}+\ell$ critical points in $A(\alpha_{n,1},\beta_{n,1})$; and
\item $H|_{A(\alpha_{n,1},\beta_{n,1})}$ can be extended continuously to $\alpha_{n,1}\cup\beta_{n,1}$ by a degree $\ell$ covering $H|_{\alpha_{n,1}}:\alpha_{n,1}\to\eta_{n,2}$ and a degree $d_{n,1}$ covering $H|_{\beta_{n,1}}:\beta_{n,1}\to\eta_{n,2}$.
\end{enumerate}
\end{lema}

Now $H$ is defined on the Riemann sphere except on the annuli $A(\xi_{n,1},\alpha_{n,1})$ and $A(\beta_{n,1},\gamma_{n,1})$. Since all of the connected components of the boundaries of these two annuli, together with their images $\gamma_{n,2}$ and $\eta_{n,2}$, are quasicircles, one can make an interpolation such that the resulting map $H$ satisfies
\begin{enumerate}
\item $H|_{A(\xi_{n,1},\alpha_{n,1})}:A(\xi_{n,1},\alpha_{n,1})\to A(\gamma_{n,2},\eta_{n,2})$ is a degree $\ell$ covering map;
\item $H|_{A(\beta_{n,1},\gamma_{n,1})}:A(\beta_{n,1},\gamma_{n,1})\to A(\eta_{n,2},\gamma_{n,2})$ is a degree $d_{n,1}$ covering map; and
\item $H|_{A(\xi_{n,1},\alpha_{n,1})}$ and $H|_{A(\beta_{n,1},\gamma_{n,1})}$ are local quasiconformal.
\end{enumerate}

\subsection{Uniformization II}

Now we have a quasiregular map $H$ defined from the Riemann sphere to itself whose dynamics is sketched in Figure \ref{Fig_surgery-2}. As before, we will apply Shishikura's fundamental lemma for quasiconformal surgery to conjugate $H$ to a rational map.

\begin{cor}\label{cor:home-2}
There is a quasiconformal map $\varphi_2:\EC\to\EC$ such that
\begin{equation*}
g:=\varphi_2\circ H\circ \varphi_2^{-1}
\end{equation*}
is rational map satisfying $\varphi_2(0)=0$, $\varphi_2(\infty)=\infty$ and $\deg(g)=d+\ell+\sum_{i=1}^{n-1}m_i$.
\end{cor}

\begin{proof}
Recall that $\gamma_{n,j}$ is an equipotential defined in $B_{n,j}$, where $1\leq j\leq p_n+1$ (In particular, $\gamma_{n,p_n+1}$ is contained in $B_{n,1}$).
For the quasiregular map $H$, we define an open set
\begin{equation}\label{equ-E-2}
E_2:=
\left\{
\begin{array}{ll}
D(\eta_{1,1})\cup(\EC\setminus \overline{D}(\gamma_{n,p_n+1}))\cup\bigcup_{j=2}^{p_n}D(\gamma_{n,j})  &~~~~~~~\text{if}~p_n\geq 2, \\
D(\eta_{1,1})\cup(\EC\setminus \overline{D}(\gamma_{n,2})) &~~~~~~\text{if}~p_n=1.
\end{array}
\right.
\end{equation}
According to Lemma \ref{lema:cover-disk-to-disk-2}, we have $H(\EC\setminus \overline{D}(\gamma_{n,p_n+1}))\subset D(\eta_{1,1})$ and $H(D(\eta_{1,1}))\subset\EC\setminus \overline{D}(\gamma_{n,p_n+1})$. This implies that $H(E_2)\subset E_2$. On the other hand, $A(\xi_{n,1},\alpha_{n,1})\cup A(\beta_{n,1},\gamma_{n,1})$ is contained in $H^{-1}(D(\gamma_{n,2}))$ if $p_n\geq 2$ and is contained in $H^{-1}(\EC\setminus\overline{D}(\gamma_{n,2}))$ if $p_n=1$. Note that $H$ is analytic except on the annuli $A(\xi_{n,1},\alpha_{n,1})$ and $A(\beta_{n,1},\gamma_{n,1})$. Therefore, we have $\partial H/\partial \overline{z}=0$ on $E_2$ and a.e. on $\EC\setminus H^{-1}(E_2)$.
The result then follows immediately from Lemma \ref{lema:qc}.

By Corollary \ref{cor:home}, we know that $\deg(h)=d+\sum_{i=1}^{n-1}m_i$ and hence $0$ has $\deg(h)$ preimages which are not contained in the super-attracting basin of $\infty$. By Lemma \ref{lema:cover-disk-to-disk-2}, $0$ has another one preimage $\infty$ (with multiplicity $\ell$) in $O_n$ under the map $H$, and $0$ has no other preimages by Lemma \ref{lema:annu-disk-2}. This implies that $\deg(g)=\deg(h)+\ell=d+\ell+\sum_{i=1}^{n-1}m_i$.
\end{proof}

\begin{rmk}
The rational map $g$ has a super-attracting cycle $0\leftrightarrow\infty$ with period $2$.
\end{rmk}

\subsection{The buried property}

Recall that $\varphi_1:\EC\to\EC$ and $\varphi_2:\EC\to\EC$ are quasiconformal mappings introduced in Corollaries \ref{cor:home} and \ref{cor:home-2}.

\begin{prop}\label{prop-buried}
The set $\varphi_2\circ \varphi_1(J(f))$ is a buried Julia component of $g$.
\end{prop}

\begin{proof}
Applying a similar argument to the first paragraph in the proof of Proposition \ref{prop-homeo}, we know that $\varphi_2(J(h))$ (which is not connected) is contained in the Julia set of $g$. In particular, $J':=\varphi_2\circ \varphi_1(J(f))$ is contained in $J(g)$.

Similar to the proof of Proposition \ref{prop-homeo}, we show that there exists a sequence of Julia components $\{J_k\}_{k\in\N}$ of $g$ in $\varphi_2(B_{n,1})$ which converges to $\varphi_2(\partial B_{n,1})$ in the Hausdorff metric. In the following we assume the period of $B_{n,1}$ is $p_n\geq 2$ without loss of generality (the case for $p_n=1$ is completely similar). By Lemma \ref{lema:cover-disk-to-disk-2}, $H:\C\setminus \overline{D}(\xi_{n,1})\to\EC\setminus(\overline{D}(\gamma_{n,2})\cup\{0\})$ is a covering map with degree $\ell\geq 1$ and $H(\gamma_{n,p_n+1})=\eta_{1,1}\subset U_{1,1}$, it follows that the annulus $\varphi_2(A(\gamma_{n,p_n+1},\xi_{n,1}))$ contains a Julia component $J_0$ of $g$ separating $\infty$ from $\varphi_2(\partial B_{n,1})$. On the other hand, by the surgery construction in \S\ref{subsec-annu-to-disk-2}, it follows that the annulus $\varphi_2(A(\xi_{n,1},\gamma_{n,1}))$ is contained in a Fatou component $W_0$ of $g$ and it separates $\infty$ from $\varphi_2(\partial B_{n,1})$ also.

Since $H^{\circ p_n}=h^{\circ p_n}:A(\gamma_{n,1},\partial B_{n,1})\to A(\gamma_{n,p_n+1},\partial B_{n,1})$ is a covering map with degree $d_n$, it follows that the annulus $\varphi_2(A(\gamma_{n,1},\partial B_{n,1}))$ contains a Julia component $J_1:=g^{-p_n}(J_0)\cap \varphi_2(A(\gamma_{n,1},\partial B_{n,1}))$ and a Fatou component $W_1:=g^{-p_n}(W_0)\cap \varphi_2(A(\gamma_{n,1},\partial B_{n,1}))$ such that both $J_1$ and $W_1$ separate $\infty$ from $\varphi_2(\partial B_{n,1})$. Inductively, one can obtain a sequence of Julia components $\{J_k\}_{k\geq 1}$ and a sequence of Fatou components $\{W_k\}_{k\geq 1}$ in $\varphi_2(A(\gamma_{n,1},\partial B_{n,1}))$ converging to $\varphi_2(\partial B_{n,1})$ as $k\to\infty$ such that $g^{\circ p_n}(J_k)=J_{k-1}$, $g^{\circ p_n}(W_k)=W_{k-1}$ and each $J_k$ and $W_k$ separates $\infty$ from $\varphi_2(\partial B_{n,1})$.

Let $U$ be a Fatou component of $h$, which is simply or doubly connected. There exists $j=j(U)\in\N$ such that $h^{\circ j}(U)=B_{n,1}$. Let $\partial_0 U$ be any connected component of the boundary $\partial U$. Then there exist a sequence of Julia components $\{J_k^U\}_{k\in\N}$ and a sequence of Fatou components $\{W_k^U\}_{k\in\N}$ of $g$ in $\varphi_2(U)$ such that they converge to $\varphi_2(\partial_0 U)$ in the Hausdorff metric as $k\to\infty$. Recall that $\MV=F(f)\setminus\MA_n$ is the union of the Fatou components which are not iterated to the cycle $O_n$ eventually. By Proposition \ref{prop-homeo}, for each component $V$ of $\varphi_1(\MV)$, there is a sequence of Fatou components $\{U_l\}_{l\in\N}$ of $h$ in $V$ converging to $\partial V$ in the Hausdorff metric.
Therefore, the sequence of Julia components $\{J_0^{U_l}\}_{l\in\N}$ and Fatou components $\{W_0^{U_l}\}_{l\in\N}$ in $\varphi_2(V)$ converge to $\varphi_2(\partial V)$ as $l\to\infty$.

Up to now, we have proven that for each component $V$ of $\varphi_2\circ \varphi_1(F(f))$, there exist a sequence of different Julia components $\{J_i\}_{i\in\N}$ and a sequence of different Fatou components $\{W_i\}_{i\in\N}$ of $g$ in $V$ such that they converge to $\partial V$ in the Hausdorff metric as $i\to\infty$. This implies that $J'=\varphi_2\circ \varphi_1(J(f))$ is a Julia component of $g$ and $J'$ is semi-buried from every component of $\varphi_2\circ \varphi_1(F(f))$. Hence $J'$ is a buried Julia component of $g$.
\end{proof}

\begin{rmk}
By the construction of $H$, it follows that $f|_{J(f)}:J(f)\to J(f)$ is quasicomformally conjugate to the restriction of $g$ on the buried Julia component $\varphi_2\circ \varphi_1(J(f))$.
Actually, one can obtain the same result by performing a ``big" surgery by combining the surgeries in \S\ref{sec-semi-buried} and in the present section.
\end{rmk}

We now prove Theorem \ref{thm-main} under the assumption that $f$ has at least two attracting periodic orbits.

\begin{proof}[{Proof of Theorem \ref{thm-main} in the case $n\geq 2$}]
The sufficiency of the first assertion follows from Proposition \ref{prop-buried}. For the necessity, we assume that $f$ contains a parabolic periodic point or a Siegel disk. Suppose that there is a rational map $g$ such that $g$ has a buried Julia component $J'$ on which $g$ is conjugate to $f$ on $J(f)$ by a restriction of a quasiconformal map $\varphi$. Since topological conjugacy preserves the multiplier at indifferent periodic points (see \cite{Nai82} or \cite{PM97}), it follows that $g$ has also a parabolic periodic point on $J'$ if $f$ does. But this is impossible since $J'$ is a buried Julia component. On the other hand, suppose that $f$ has a Siegel disk $U$ with period $p\geq 1$. Note that the quasiconformal map $\varphi$ is defined in a neighborhood of $J(f)$. This implies that $U$ contains an annular neighborhood $A$ such that $\varphi(A\cap U)$ is contained in a rotation domain of $g$. However, this is also impossible since $\varphi(J(f))$ is a buried Julia component of $g$.

Now we only need to verify the statements on the degrees. By Corollary \ref{cor:home-2}, the proof is reduced to find a upper bound of $d+\ell+\sum_{i=1}^{n-1}m_i$ as small as possible. By Lemmas \ref{lema:cover-disk-to-disk} and \ref{lema:cover-disk-to-disk-2}, the integers $m_i$ and $\ell$ should satisfy
\begin{equation*}
m_i>\frac{d_{i,1}}{d_i-1} \quad\text{and}\quad \ell>\frac{d_{n,1}+\tfrac{d_n}{m_1}}{d_n-1},
\end{equation*}
where $1\leq i\leq n-1$. Since $d_{i,1}\leq d_i$ and $d_i\geq 2$ for all $1\leq i\leq n$, we have
\begin{equation*}
\frac{d_{i,1}}{d_i-1}\leq \frac{d_i}{d_i-1}\leq 2 \quad\text{and}\quad
\frac{d_{n,1}+\tfrac{d_n}{m_1}}{d_n-1}\leq \frac{d_n}{d_n-1}\cdot\frac{m_1+1}{m_1}\leq 2+\frac{2}{m_1}.
\end{equation*}
This implies that we can choose $m_i=3$ and $\ell=3$, where $1\leq i\leq n-1$. Since $n\leq 2d-2$, the degree of $g$ is at most $d+3+3(2d-3)=7d-6$. In particular, if $f$ is a polynomial, then $n\leq d$ and we have $\deg(g)\leq d+3+3(d-1)=4d$.
\end{proof}

\section{The first case of exactly one attracting cycle}\label{sec-one-orbit}

In this section, we assume that $f$ is a rational map with degree $d\geq 2$ whose Julia set is connected. Furthermore, $f$ is assumed to be post-critically finite in the Fatou set and have exactly one cycle of periodic super-attracting basins. Therefore, all the Fatou components of $f$ will be iterated onto the periodic Fatou components containing this cycle eventually. Let $O$ be this cycle and $p\geq 1$ its period.

We also assume that the Fatou set of $f$ has at least two connected components. The case that $f$ has exactly one connected component will be discussed in the next section (i.e. the dendrite case). Under these assumptions, $f$ cannot be a polynomial since otherwise it has at least two super-attracting cycles or its Fatou set has exactly one connected component.

If $f$ has exactly two Fatou components, then $f^{\circ 2}$ is conjugate to a unicritical polynomial. Indeed, suppose that $D_0$ and $D_\infty$ are the only Fatou components of $f$ containing $0$ and $\infty$ respectively. Then we have $f(D_0)=D_\infty$ and $f(D_\infty)=D_0$ since $f$ is assumed to have exactly one super-attracting cycle. Then $D_0$ and $D_\infty$ are completely invariant under $f^{\circ 2}$. Since $f$ is assumed to be post-critically finite in the Fatou set, without loss of generality, we can assume that $f(0)=\infty$ and $f(\infty)=0$. This implies that $(f^{\circ 2})^{-1}(\infty)=\infty$, $f$ is conjugate to $z\mapsto z^{-d}$ with $d\geq 2$ and the Julia set of $f$ is a round circle.

It is known from \cite{McM88} (see also \cite{DLU05}, \cite{BDGR08}, \cite{GMR13}, \cite{HP12}, \cite{QYY15} and \cite{FY15}) that there exists a rational map $g$ whose Julia set is a set of Cantor circles and contains a buried Julia component which is a Jordan curve. Moreover, on this buried Jordan curve $g$ is quasiconformally conjugate to $z\mapsto z^{-d}$ on the unit circle. Therefore, in the remaining of this section we always assume that $f$ has at least three Fatou components. By \cite[Theorem 5.6.2]{Bea91a}, this implies that $f$ has infinitely many Fatou components.

We denote the immediate super-attracting basin of $f$ as $B_1\mapsto B_2\mapsto\dots\mapsto B_p\mapsto B_1$. Let $U_1\neq B_p$ be a preimage of $B_1$. Without loss of generality, we assume that $\infty\in B_1$ is contained in the super-attracting periodic orbit, $0\in U_1$ and $f(0)=\infty$.

One may find that actually we are in a similar case as \S\ref{sec-second-pert}. The rational map $h$ there and the rational map $f$ here both have exactly one super-attracting basin. The difference is that the Julia set of $f$ is connected while the Julia set of $h$ is disconnected. However, from the construction of the surgery in \S\ref{sec-second-pert}, we know that the connectivity of the Julia set is not an obstruction.

Let $d_0\geq 2$ be the degree of the restriction $f^{\circ p}|_{B_1}:B_1\to B_1$ and denote $d_{0,1}:=\deg(f|_{B_1}:B_1\to B_2)$. We denote $m_1:=\deg(f|_{U_1}:U_1\to B_1)\geq 1$. The proof of the following proposition is completely similar to that of Proposition \ref{prop-buried} (compare Lemma \ref{lema:cover-disk-to-disk-2}) and we omit the details.

\begin{prop}\label{prop-buried-p-1}
There exists a rational map $g$ with degree $\deg(g)=d+\ell$, where $\ell\geq 1$ is the minimal integer satisfying
\begin{equation*}
\ell>\frac{d_{0,1}+\tfrac{d_0}{m_1}}{d_0-1},
\end{equation*}
such that its Julia set contains a buried Julia component on which the dynamics is quasiconformally conjugate to that of $f$ on $J(f)$.
\end{prop}

We now prove Theorem \ref{thm-main} under the assumption that $f$ has exactly one attracting periodic orbit and the Fatou set contains infinitely many Fatou components.

\begin{proof}[{Proof of Theorem \ref{thm-main} in the case $n=1$ (the non-dendrite case)}]
The first assertion follows from Proposition \ref{prop-buried-p-1}. We only need to verify the statements on the degrees.
Since $d_0\geq d_{0,1}$, $d_0\geq 2$ and $m_1\geq 1$, we have
\begin{equation*}
\frac{d_{0,1}+\tfrac{d_0}{m_1}}{d_0-1}\leq \frac{d_0}{d_0-1}\cdot\frac{m_1+1}{m_1}\leq 4.
\end{equation*}
This implies that $\ell\leq 5$. The degree of $g$ is at most $d+5\leq 4d$ if $d\geq 2$.
\end{proof}

\section{The second case of exactly one attracting cycle: dendrite}\label{sec-dendrite}

The Julia set of a rational map $f$ is called a \textit{dendrite} if the Julia set is connected and the Fatou set of $f$ is non-empty and contains exactly one component. In this section, we assume that $f$ is a rational map whose Julia set is a dendrite. Since the unique Fatou component of $f$ is completely invariant, without loss of generality, we suppose that $f$ is a polynomial with degree $d\geq 2$ and $0$ is contained in the filled-in Julia set of $f$.

Let $B$ be the unique super-attracting basin of $f$ centered at the infinity. We will construct a quasiregular map $G$ based on $f$ such that the straightened map of $G$ contains a buried Julia component which is homeomorphic to $J(f)$. Let $\phi:B\to\EC\setminus\overline{\D}$ be the B\"{o}ttcher map conjugating $f:B\to B$ holomorphically to $z\mapsto z^d$ from $\EC\setminus\overline{\D}$ to itself, where $\phi(\infty)=\infty$. Recall that an \textit{equipotential} $\gamma$ in $B$ is the preimage by $\phi$ of an Euclidean circle in $\EC\setminus\overline{\D}$ centered at 0. In this section, the \textit{level} of an equipotential $\gamma$ in $B$ will be denoted by $L(\gamma)\in(1,+\infty)$.

\begin{lema}\label{lema:dendrite}
Let $m\geq 5$ be an integer. For any $R\geq 2$, there are five different equipotentials $\gamma_1$, $\gamma_2$, $\xi_1$, $\xi_2$, $\alpha_1$ in $B$, and a holomorphic branched mapping $G|_{\EC\setminus\overline{D}(\xi_1)}: \EC\setminus\overline{D}(\xi_1)\to\EC\setminus\overline{D}(\alpha_1)$ satisfying the following conditions:
\begin{enumerate}
\item $R=L(\gamma_1)<L(\alpha_1)<L(\xi_1)<L(\gamma_2)=R^d<L(\xi_2)$;
\item $G(\infty)=\infty$ and $G:\C\setminus\overline{D}(\xi_1)\to\C\setminus\overline{D}(\alpha_1)$ is a degree $m$ covering map; and
\item $G(\xi_1)=\alpha_1$, $G(\gamma_2)=\xi_2$ and $f(\gamma_1)=\gamma_2$.
\end{enumerate}
\end{lema}

\begin{proof}
Let $R\geq 2$ and let $S_1$, $S_2$, $S_3$ be the following constants satisfying $R<S_1<S_2<R^d<S_3$, where
\begin{equation}\label{equ-S-1-2}
S_1=R^{(d+4)/5}, \quad S_2=R^{(4d+1)/5} \quad \text{and} \quad S_3=S_1 R^{dm}/S_2^m>R^d.
\end{equation}
We define $\gamma_1$, $\alpha_1$, $\xi_1$, $\gamma_2$ and $\xi_2$ as the equipotentials in $B$ with the levels $R$, $S_1$, $S_2$, $R^d$ and $S_3$ respectively. See Figure \ref{Fig_surgery-3}.

\begin{figure}[!htpb]
  \setlength{\unitlength}{1mm}
  \centering
  \includegraphics[width=0.95\textwidth]{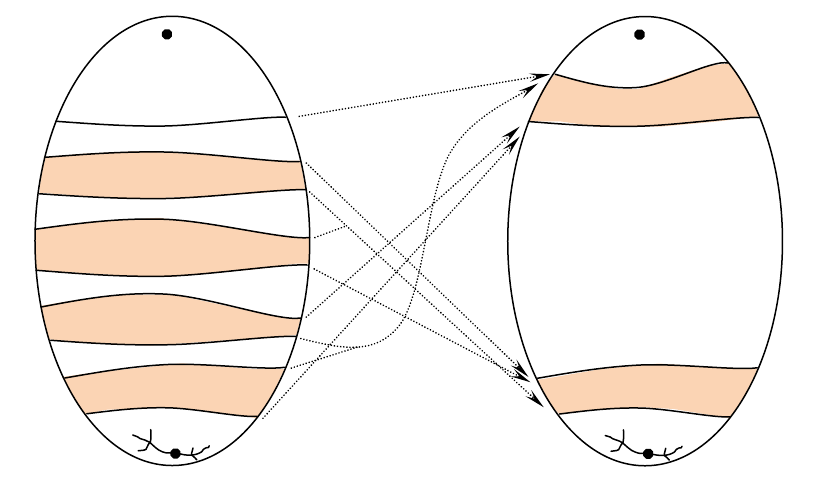}
  \put(-95,6){$0$}
  \put(-94,64){$\infty$}
  \put(-112,10){$\gamma_1$}
  \put(-115,14){$\alpha_1$}
  \put(-117,20){$\beta_1$}
  \put(-119,25){$\alpha_2$}
  \put(-119,30){$\beta_2$}
  \put(-120,36){$\alpha_3$}
  \put(-119,41){$\beta_3$}
  \put(-118,47){$\xi_1$}
  \put(-117,53){$\gamma_2$}
  \put(-12,9){$\gamma_1$}
  \put(-9,15){$\alpha_1$}
  \put(-8,53){$\gamma_2$}
  \put(-12,61){$\xi_2$}
  \put(-25,6){$0$}
  \put(-24,64){$\infty$}
  \put(-62,45){$G$}
  \caption{The sketch of the definition of the map $G|_{\EC\setminus\overline{D}(\xi_1)}: \EC\setminus\overline{D}(\xi_1)\to\EC\setminus\overline{D}(\alpha_1)$. For further references, some other equipotentials are drawn also in the basin of infinity and the dendrite is also marked.}
  \label{Fig_surgery-3}
\end{figure}

Define $Q_m(z)=S_1 z^m/S_2^m:\EC\setminus\overline{\D}(0,S_2)\to\EC\setminus\overline{\D}(0,S_1)$.
Then $Q_m(\EC\setminus\overline{\D}(0,R^d))=\EC\setminus\overline{\D}(0,S_3)$. Recall that $\phi:B\to\EC\setminus\overline{\D}$ is the B\"{o}ttcher map conjugating $f:B\to B$ holomorphically to $z\mapsto z^d$, where $\phi(\infty)=\infty$. Then $G:=\phi^{-1}\circ Q_m\circ\phi:\EC\setminus\overline{D}(\xi_1)\to\EC\setminus\overline{D}(\alpha_1)$ is the required holomorphic branched covering map.
\end{proof}

As the surgery performed before, we need to construct some annulus-to-disk branched covering map and annulus-to-annulus covering maps. Recall that $m\geq 5$ is an integer introduced in Lemma \ref{lema:dendrite}.

\begin{lema}\label{lema:annu-disk-3-key}
There exists $R_0\geq 2$ such that if $R\geq R_0$, then there exist two different equipotentials $\alpha_2$, $\beta_2$ in $B$ and a holomorphic covering mapping $G|_{A(\alpha_2,\beta_2)}:A(\alpha_2,\beta_2)\to A(\gamma_2,\alpha_1)$ with degree $m$ such that\footnote{The equipotentials $\gamma_1$, $\alpha_1$, $\xi_1$, $\gamma_2$ and $\xi_2$ are defined as in Lemma \ref{lema:dendrite}. }
\begin{enumerate}
\item $L(\alpha_1)<L(\alpha_2)<L(\beta_2)<L(\xi_1)$;
\item $G|_{A(\alpha_2,\beta_2)}$ can be extended continuously to $\alpha_2\cup\beta_2$ by two degree $m$ covering maps $G|_{\alpha_2}:\alpha_2\to\gamma_2$ and  $G|_{\beta_2}:\beta_2\to\alpha_1$; and
\item $\Mod(A(\alpha_1,\alpha_2))\geq\frac{3}{\pi}\log 2$ and $\Mod(A(\beta_2,\xi_1))\geq\frac{1}{m\pi}\log (4R+2)$.
\end{enumerate}
\end{lema}

\begin{proof}
By \eqref{equ-S-1-2}, we have
\begin{equation}\label{equ-mod-1}
\Mod(A(\alpha_1,\xi_1))=\frac{3(d-1)}{10\pi}\log R \quad \text{and}\quad \Mod(A(\alpha_1,\gamma_2))=\frac{2(d-1)}{5\pi}\log R.
\end{equation}
Since $m\geq 5$, we have
\begin{equation}\label{equ:mm}
\frac{\Mod(A(\alpha_1,\gamma_2))}{m}=\frac{2(d-1)}{5m\pi}\log R<\frac{3(d-1)}{10\pi}\log R=\Mod(A(\alpha_1,\xi_1)).
\end{equation}
This implies that one can choose an essential subannulus\footnote{An annulus $A_1$ is called an \textit{essential subannulus} of the annulus $A_2$ if $A_1$ separates two boundaries of $A_2$.} $A(\alpha_2,\beta_2)$ of $A(\alpha_1,\xi_1)$ with $\Mod(A(\alpha_2,\beta_2))=\frac{2(d-1)}{5m\pi}\log R$ such that $G|_{A(\alpha_2,\beta_2)}:A(\alpha_2,\beta_2)\to A(\gamma_2,\alpha_1)$ is a holomorphic covering mapping with degree $m$. Moreover, $\alpha_2$ and $\beta_2$ can be chosen such that $L(\alpha_1)<L(\alpha_2)<L(\beta_2)<L(\xi_1)$. See Figure \ref{Fig_surgery-3}.

\medskip

Note that $\alpha_1$, $\alpha_2$, $\beta_2$ and $\xi_1$ are equipotentials. We have
\begin{equation*}
\Mod(A(\alpha_1,\alpha_2))+\Mod(A(\alpha_2,\beta_2))+\Mod(A(\beta_2,\xi_1))=\Mod(A(\alpha_1,\xi_1)).
\end{equation*}
In order to obtain (c), it is sufficient to guarantee that
\begin{equation}\label{equ:ineq-1}
\frac{3}{\pi}\log 2+\frac{2(d-1)}{5m\pi}\log R+\frac{1}{m\pi}\log (4R+2)\leq \frac{3(d-1)}{10\pi}\log R.
\end{equation}
Since $d\geq 2$ and $m\geq 5$, we have $(3m-4)(d-1)>10$. This is equivalent to
\begin{equation*}
\frac{2(d-1)}{5m}+\frac{1}{m}<\frac{3(d-1)}{10}.
\end{equation*}
Therefore, there exists a large number $R_0\geq 2$ such that if $R\geq R_0$, then
\begin{equation*}
3\log 2+\frac{\log 5}{m}\leq \left(\frac{3(d-1)}{10\pi}-\frac{2(d-1)}{5m}-\frac{1}{m}\right)\log R.
\end{equation*}
This implies that \eqref{equ:ineq-1} holds and the proof is complete.
\end{proof}

Let $R_0\geq 2$ be the number introduced in Lemma \ref{lema:annu-disk-3-key}.

\begin{lema}\label{lema:annu-disk-3}
For any $R\geq R_0$, there exist three equipotentials $\beta_1$, $\alpha_3$, $\beta_3$ in $B$ and two holomorphic branched covering maps $G|_{A(\alpha_1,\beta_1)}:A(\alpha_1,\beta_1)\to \EC\setminus\overline{D}(\xi_2)$ with degree $d+m$ and $G|_{A(\alpha_3,\beta_3)}:A(\alpha_3,\beta_3)\to D(\gamma_1)$ with degree $2m$ such that\footnote{The equipotentials $\gamma_1$, $\alpha_1$, $\xi_1$, $\gamma_2$ and $\xi_2$ are defined as in Lemma \ref{lema:dendrite}, and $\alpha_2$, $\beta_2$ are defined as in Lemma \ref{lema:annu-disk-3-key}. }
\begin{enumerate}
\item $L(\alpha_1)<L(\beta_1)<L(\alpha_2)<L(\beta_2)<L(\alpha_3)<L(\beta_3)<L(\xi_1)$;
\item $G|_{A(\alpha_1,\beta_1)}$ has $d+m$ critical points in $A(\alpha_1,\beta_1)$ and $G|_{A(\alpha_1,\beta_1)}$ can be extended continuously to $\alpha_1\cup\beta_1$ by a degree $d$ covering $G|_{\alpha_1}:\alpha_1\to\xi_2$ and a degree $m$ covering $G|_{\beta_1}:\beta_1\to\xi_2$;
\item $G|_{A(\alpha_3,\beta_3)}$ has $2m$ critical points in $A(\alpha_3,\beta_3)$ and $G|_{A(\alpha_3,\beta_3)}$ can be extended continuously to $\alpha_3\cup\beta_3$ by two degree $m$ covering maps $G|_{\alpha_3}:\alpha_3\to\gamma_1$ and $G|_{\beta_3}:\beta_3\to\gamma_1$; and
\item $G|_{A(\alpha_3,\beta_3)}$ has exactly two different critical values which are contained in the Julia set of $f$ and $G^{-1}(J(f))\cap A(\alpha_3,\beta_3)$ is a connected compact set separating $\alpha_3$ and $\beta_3$.
\end{enumerate}
\end{lema}

\begin{proof}
By Lemma \ref{lema:annu-disk-3-key}(c), we have $\Mod(A(\alpha_1,\alpha_2))> \tfrac{2}{\pi}\log 2$. Then the existence of  $G|_{A(\alpha_1,\beta_1)}:A(\alpha_1,\beta_1)\to \EC\setminus\overline{D}(\xi_2)$ satisfying the first two conditions is guaranteed by the first half statement of Lemma \ref{lema:ann-to-disk-std}. Hence it is sufficient to prove the existence of $G|_{A(\alpha_3,\beta_3)}:A(\alpha_3,\beta_3)\to D(\gamma_1)$.

Let $\psi_1:D(\gamma_1)\to \D$ be a conformal isomorphism such that $0\in\psi_1(J(f))$. Let $b_0\in\psi_1(J(f))$ be a point such that
\begin{equation*}
|b_0|=\sup\{|z|:\,z\in\psi_1(J(f))\}.
\end{equation*}
Since $\Mod(\D\setminus\psi_1(J(f)))=\Mod(D(\gamma_1)\setminus J(f))=\tfrac{1}{2\pi}\log R$, it follows that $|b_0|\geq 1/R$. Without loss of generality, we assume that $b_0$ is a real positive number. Let $\psi_2:\D\to\D$ be an isomorphism such that
\begin{equation*}
\psi_2(z)=\frac{z-b_1}{1-b_1 z}, \text{ where } b_1=\frac{1-\sqrt{1-b_0^2}}{b_0}\in(0,1).
\end{equation*}
Then $\psi_2$ maps $0$, $b_0$ to $-b_1$, $b_1$ respectively. Moreover, since $b_0\geq 1/R$, we have
\begin{equation*}
b_1\geq R-\sqrt{R^2-1}\geq 1/(2R).
\end{equation*}

By the second conclusion of Lemma \ref{lema:ann-to-disk-std}, there exists a number $r>0$ and a holomorphic branched covering map $\psi:\A_r\to \D$ with degree $2m$ such that
\begin{enumerate}
\item[(1)] $\psi$ has exactly $2m$ critical points in $\A_r$ and exactly $2$ critical values at $\pm\,b_1$;
\item[(2)] $\psi$ can be extended continuously to $\partial\A_r$ by two degree $m$ covering maps $\psi|_{\partial\D}:\partial\D\to\partial\D$ and $\psi|_{\T_r}:\T_r\to\partial\D$;
\item[(3)] $\Mod(\A_r)<\tfrac{1}{m\pi}\log (\tfrac{2}{b_1}+1)\leq \tfrac{1}{m\pi}\log (4R+1)$; and
\item[(4)] $\psi^{-1}(\psi_2(\psi_1(J(f))))\subset\A_r$ is a connected compact set separating $0$ from $\infty$.
\end{enumerate}
Therefore, to prove the existence of $G|_{A(\alpha_3,\beta_3)}:A(\alpha_3,\beta_3)\to D(\gamma_1)$ satisfying (c) and (d), it is sufficient to guarantee that $\Mod(A(\beta_2,\xi_1))\geq \tfrac{1}{m\pi}\log (4R+1)$. This is true by Lemma \ref{lema:annu-disk-3-key}(c).
\end{proof}

In the following we fix the number $R\geq R_0$.
Define $G|_{D(\gamma_1)}:=f|_{D(\gamma_1)}$. Then $G$ is defined on the Riemann sphere except on the annuli $A(\gamma_1,\alpha_1)$, $A(\beta_1,\alpha_2)$, $A(\beta_2,\alpha_3)$, $A(\beta_3,\xi_1)$. Since all of the connected components of the boundaries of these four annuli, together with their images $\gamma_1$, $\alpha_1$, $\gamma_2$ and $\xi_2$, are quasicircles, one can make an interpolation such that the resulting map $G$ satisfies
\begin{enumerate}
\item $G|_{A(\gamma_1,\alpha_1)}:A(\gamma_1,\alpha_1)\to A(\gamma_2,\xi_2)$ is a degree $d$ covering map;
\item $G|_{A(\beta_1,\alpha_2)}:A(\beta_1,\alpha_2)\to A(\xi_2,\gamma_2)$ is a degree $m$ covering map;
\item $G|_{A(\beta_2,\alpha_3)}:A(\beta_2,\alpha_3)\to A(\alpha_1,\gamma_1)$ is a degree $m$ covering map;
\item $G|_{A(\beta_3,\xi_1)}:A(\beta_3,\xi_1)\to A(\gamma_1,\alpha_1)$ is a degree $m$ covering map; and
\item The four annular mapping above are local quasiconformal.
\end{enumerate}

Now we have a quasiregular map $G$ defined from the Riemann sphere to itself whose dynamics is sketched in Figure \ref{Fig_surgery-3}.

\begin{cor}\label{cor:home-2-3}
There is a quasiconformal map $\varphi:\EC\to\EC$ such that $g:=\varphi\circ G\circ \varphi^{-1}$ is rational map satisfying $\varphi(0)=0$, $\varphi(\infty)=\infty$ and $\deg(g)=d+2m$.
\end{cor}

\begin{proof}
For the quasiregular map $G$, we define an open set $E:=\EC\setminus\overline{D}(\gamma_2)$.
According to Lemma \ref{lema:dendrite}, we have $G(E)\subset E$. By Lemmas \ref{lema:dendrite}, \ref{lema:annu-disk-3-key} and \ref{lema:annu-disk-3}, and the annulus-to-annulus construction of $G$, it follows that $G$ is holomorphic outside of $G^{-2}(\overline{E})$. This implies that $\partial G/\partial \overline{z}=0$ on $E$ and a.e. on $\EC\setminus G^{-2}(E)$. Then the result follows immediately from Lemma \ref{lema:qc}.
\end{proof}

\begin{prop}\label{prop-buried-3}
The set $\varphi(J(f))$ is a buried Julia component of $g$ and the restriction of $g$ on $\varphi(J(f))$ is quasiconformally conjugate to $f$ on $J(f)$.
\end{prop}

\begin{proof}
The proof is similar to that of Propositions \ref{prop-homeo} and \ref{prop-buried}. Hence we only give a sketch of the proof here. Firstly $\varphi(J(f))$ is contained in the Julia set of $g$ by a similar argument as in the first paragraph of the proof of Proposition \ref{prop-homeo}. Then we show that there exists a sequence of Julia components $\{J_k\}_{k\in\N}$ of $g$ in $\varphi(D(\gamma_2))$ which converges to $\varphi(J(f))$ in the Hausdorff metric. By the annulus-to-annulus construction of $G$ above, it follows that the image of the annuli $A(\gamma_1,\alpha_1)$, $A(\beta_1,\alpha_2)$, $A(\beta_2,\alpha_3)$ and $A(\beta_3,\xi_1)$ under $\varphi$ are contained in the Fatou set of $g$. According to Lemma \ref{lema:annu-disk-3}(d), the annulus $\varphi(A(\alpha_3,\beta_3))\subset \varphi(A(\gamma_1,\gamma_2))$ contains a Julia component $J_0$ of $g$ separating $\infty$ from $\varphi(J(g))$. Then $\varphi(A(\gamma_1,\alpha_1))$ is contained in a Fatou component $U_0$ of $g$ which separates $\infty$ and $\varphi(J(g))$.

Since $f:A(\gamma_1,J(f))\to A(\gamma_2,J(f))$ is a covering map with degree $d$, it follows that the annulus $\varphi(A(\gamma_1,J(f)))$ contains a Julia component $J_1:=G^{-1}(J_0)\cap \varphi(A(\gamma_1,J(f)))$ and a Fatou component $U_1:=G^{-1}(U_0)\cap \varphi(A(\gamma_1,J(f)))$ such that both $J_1$ and $U_1$ separate $\infty$ from $\varphi(J(f))$. Inductively, one can obtain a sequence of Julia components $\{J_k\}_{k\geq 1}$ and a sequence of Fatou components $\{U_k\}_{k\geq 1}$ in $\varphi(A(\gamma_1,J(f)))$ converging to $\varphi(J(f))$ such that $G(J_k)=J_{k-1}$, $G(U_k)=U_{k-1}$ and each $J_k$ and $U_k$ separates $\infty$ from $\varphi(J(f))$. This implies that $\varphi(J(f))$ is a buried Julia component of $g$.
\end{proof}

\begin{proof}[{Proof of Theorem \ref{thm-main} in the case $n=1$ (the dendrite case)}]
The first assertion follows from Proposition \ref{prop-buried-3}. For the statements on the degrees, by Lemmas \ref{lema:dendrite}, \ref{lema:annu-disk-3-key} and \ref{lema:annu-disk-3}, we can choose $m=5$ and then $\deg(g)=d+10\leq 4d+4$ for all $d\geq 2$.
\end{proof}

\begin{proof}[{Proof of Theorem \ref{thm-main}}]
The sufficiency of the first statement has been proven respectively in Propositions \ref{prop-buried}, \ref{prop-buried-p-1} and \ref{prop-buried-3} by dividing the arguments into three cases. Next to the proof of these propositions the results on the degrees have been proven. The necessity of the first statement has been proven at the end of \S\ref{sec-second-pert}.
\end{proof}

\section{The quartic Julia sets containing buried Jordan curves}\label{sec-quartic}

In this section, we aim to construct some quartic rational maps whose Julia sets contain buried components which are Jordan curves. Except Godillon's example of cubic rational map \cite{God15}, the known rational maps whose buried Julia components are Jordan curves have at least degree five. Note that Godillon's example is degree three but his construction requires constructing weighted dynamical trees and solving Hurwitz's problem, which is very complicated. We present here a much simpler construction, which is based on the observation of Lemma \ref{lema:cover-disk-to-disk}.

Recall that in \S\ref{subsec-equi}, $B_{i,1}\mapsto\cdots\mapsto B_{i,p_i}$ is a cycle of super-attracting basins of the given rational map $f$, where $i\geq 1$ is an index. In order to construct a semi-buried Julia component, the degree of the holomorphic branched covering $F$ from a disk to another (see Lemma \ref{lema:cover-disk-to-disk}) is required to satisfy $m_i>d_{i,1}/(d_i-1)$, where $d_{i,1}=\deg(f:B_{i,1}\to B_{i,2})\geq 1$ and $d_i=\deg(f^{\circ p_i}:B_{i,1}\to B_{i,1})\geq 2$. Here the construction of $F$ can be seen as a small perturbation.

As remarked before, if $d_{i,1}=1$ and $d_i=2$, then $m_i$ can be chosen as $2$. If $d_{i,1}=1$ and $d_i\geq 3$, then $m_i$ can be chosen as $1$.
Therefore, it is natural to consider the following two families:
\begin{equation*}
f_\lambda(z)=z^2+c+\frac{\lambda}{(z-c)^2} \quad\text{and}\quad g_\mu(z)=z^q+b+\frac{\mu}{z-b},
\end{equation*}
where $\lambda\neq 0$ (resp. $\mu\neq 0$) and $c\neq 0$ (resp. $b\neq 0$) is the center of a hyperbolic component of the Mandelbrot set (resp. the Multibrot set $M_q$ with $q\geq 3$). This means that the perturbation is made at the \textit{critical value} $c$ of $z\mapsto z^2+c$ (resp. $b$ of $z\mapsto z^q+b$), but not at the \textit{critical point} $0$.

The main content in this section is to prove that if $\lambda\neq 0$ (resp. $\mu\neq 0$) is small enough, then the Julia set of $f_\lambda$ (resp. $g_\mu$) contains some buried Jordan curves and a semi-buried Julia component which is homeomorphic to the Julia set of $z\mapsto z^2+c$ (resp. $z\mapsto z^q+b$). We only give the proof of Theorem \ref{thm-exam} for $f_\lambda$ since the argument for $g_\mu$ is completely similar.

\subsection{The position of the critical orbits}

In the remaining of this section, we always assume that $\lambda\neq 0$ is sufficiently small. We use $B_\lambda$ and $T_\lambda$, respectively, to denote the immediate attracting basin of $f_\lambda$ containing $\infty$ and the Fatou component of $f_\lambda$ containing the pole $c$. Let $p\geq 2$ be the minimal period of the critical point of $P_c(z)= z^2+c$, i.e. $P_c^{\circ p}(0)=0$ and $P_c^{\circ p}(c)=c$.

Let $X(\lambda)> 0$ and $Y(\lambda)>0$ be two numbers depending on $\lambda$. We denote $X(\lambda)\preceq Y(\lambda)$ if there exists a constant $C_0> 0$ independent of $\lambda$ (where $\lambda\neq 0$ is small enough) such that $X(\lambda)\leq C_0Y(\lambda)$. Moreover, we use $X(\lambda)\asymp Y(\lambda)$ to denote the two numbers which are comparable, i.e. there exist two constants $C_1$, $C_2>0$ independent of $\lambda$ (where $\lambda\neq 0$ is small enough) such that $C_1Y(\lambda)\leq X(\lambda)\leq C_2 Y(\lambda)$.

\begin{lema}[See Figure \ref{Fig_diagram-1}]\label{lema:hyper}
There exists $\delta>0$ such that if $0<|\lambda|<\delta$, then there exist two Fatou components $A_\lambda$ and $U_\lambda$ such that
\begin{enumerate}
\item $A_\lambda$ contains three critical points $\{c_j^\lambda:1\leq j\leq 3\}$ of $f_\lambda$ such that $c_j^\lambda\in\{z:|\lambda|^{11/30}\leq |z-c|\leq |\lambda|^{3/10}\}\subset A_\lambda$; In particular, $\lim_{\lambda\to 0}c_j^\lambda=c$;
\item $U_\lambda$ contains one critical point $c_0^\lambda$ of $f_\lambda$ such that $c_0^\lambda\in\overline{\D}(0,|\lambda|^{2/3})\subset U_\lambda$; In particular, $\lim_{\lambda\to 0}c_0^\lambda=0$; and
\item $f_{\lambda}^{\circ p}(A_\lambda)=T_\lambda$ and $f_{\lambda}(U_\lambda)=T_\lambda$.
\end{enumerate}
In particular, $f_\lambda$ is hyperbolic\footnote{In Corollary \ref{cor:different} we will prove that the Fatou components $B_\lambda$, $A_\lambda$, $U_\lambda$ and $T_\lambda$ are pairwise different. Hence $A_\lambda$ contains exactly three critical points and $U_\lambda$ contains exactly one.}.
\end{lema}

\begin{figure}[!htpb]
  \setlength{\unitlength}{1mm}
  \centering
  \includegraphics[width=0.95\textwidth]{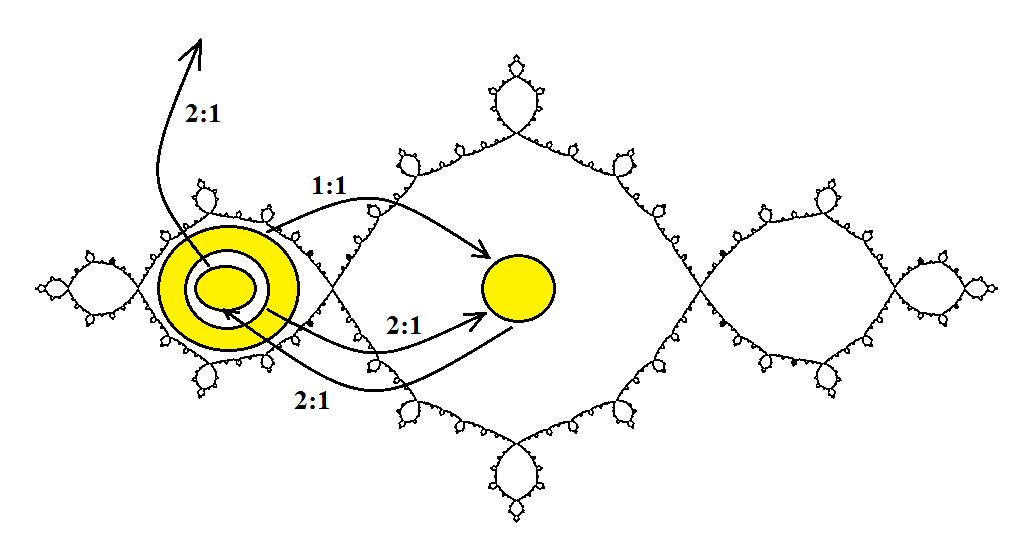}
  \put(-90,59){$B_\lambda$}
  \put(-61.5,30){$U_\lambda$}
  \put(-95.5,30){$T_\lambda$}
  \put(-88.4,30){$A_\lambda$}
  \caption{The sketch of the mapping diagram of $f_\lambda$, where $\lambda\neq 0$ is small, $c=-1$ and $p=2$. Some arrows with mapping degrees between the boundaries of the Fatou components are also marked.}
  \label{Fig_diagram-1}
\end{figure}

\begin{proof}
One can check that $f_\lambda$ has two simple critical points $\infty$ and $c$. A direct calculation shows that
\begin{equation*}
f_\lambda'(z)=2\,\frac{z(z-c)^3-\lambda}{(z-c)^3}.
\end{equation*}
Except $\infty$ and $c$, to obtain the remaining $4$ critical points $c_j^\lambda$ with $0\leq j\leq 3$, it is sufficient to solve the following equation (Now the parameter $c\neq 0$ is seen as a constant):
\begin{equation}\label{equ:deriv-1}
F(z,\lambda):=z(z-c)^3-\lambda=0.
\end{equation}
From the above equation we know that if $\lambda\neq 0$ is small, then one critical point of $f_\lambda$ is close to $0$ (without loss of generality we mark it by $c_0^\lambda$) and the remaining three critical points $c_1^\lambda$, $c_2^\lambda$, $c_3^\lambda$ are close to $c$. By \eqref{equ:deriv-1} and implicit function theorem, if $\lambda\neq 0$ is small enough, then $c_0^\lambda$ and $c_j^\lambda$ with $1\leq j\leq 3$ can be written as
\begin{equation*}
c_0^\lambda=a_0\lambda+o(|\lambda|) \quad\text{ and }\quad c_j^\lambda=c+a_je^{2\pi\ii j/3}\lambda^{1/3}+o(|\lambda|^{1/3}).
\end{equation*}
Substituting the above series to \eqref{equ:deriv-1} and comparing the coefficients, one can obtain:
\begin{equation*}
a_0=-c^{-3} \quad\text{ and }\quad a_j=c^{-1/3} \text{\quad for }1\leq j\leq 3.
\end{equation*}
In particular, we have $\lim_{\lambda\to 0}c_0^\lambda=0$ and $\lim_{\lambda\to 0}c_j^\lambda=c$ for $1\leq j\leq 3$. Since $\infty$ is a super-attracting fixed point of $f_\lambda$, there exists a constant $R>1$ such that $\EC\setminus\D(0,R)$ is contained in $B_\lambda$. This implies that there exists a constant $r_1>0$ independent of small $\lambda$ such that $\D(c,r_1|\lambda|^{1/2})$ is contained in $T_\lambda$.

If $z\in\overline{\D}(0,|\lambda|^{2/3})$, then $|z|^2$ is much smaller than $|\lambda|/|z-c|^2$ if $\lambda\neq 0$ is sufficiently small. Hence we have
\begin{equation*}
|f_\lambda(z)-c|=\left|z^2+\frac{\lambda}{(z-c)^2}\right|\asymp|\lambda|.
\end{equation*} Since $|\lambda|\ll|\lambda|^{1/2}$ and $c_0^\lambda\in\D(0,|\lambda|^{2/3})$ if $\lambda$ is small enough, it implies that $f_\lambda(c_0^\lambda)\in T_\lambda$ and $c_0^\lambda$ escapes to $\infty$ under the iterate of $f_\lambda$ for small $\lambda$. For $|\lambda|^{11/30}\leq |z-c|\leq |\lambda|^{3/10}$, we have
\begin{equation*}
\begin{split}
|f_\lambda(z)-P_c(c)|
=&\left|(z+c)(z-c)+\frac{\lambda}{(z-c)^2}\right|\\
\preceq&\,|z-c|+|\lambda|/|z-c|^2\\
\preceq&\,|\lambda|^{3/10}+|\lambda|^{4/15}\preceq |\lambda|^{4/15}.
\end{split}
\end{equation*}
Inductively, for $|\lambda|^{11/30}\leq |z-c|\leq |\lambda|^{3/10}$ and $2\leq k\leq p-1$ (if $p\geq 3$), we have
\begin{equation*}
\begin{split}
&~|f_\lambda^{\circ k}(z)-P_c^{\circ k}(c)| \\
=&\,\left|\Big(f_\lambda^{\circ (k-1)}(z)+P_c^{\circ (k-1)}(c)\Big)\Big(f_\lambda^{\circ (k-1)}(z)-P_c^{\circ (k-1)}(c)\Big)
+\frac{\lambda}{\big(f_\lambda^{\circ (k-1)}(z)-c\big)^2}\right| \\
\preceq&\,\left|f_\lambda^{\circ (k-1)}(z)-P_c^{\circ (k-1)}(c)\right|+|\lambda|\preceq|\lambda|^{4/15}.
\end{split}
\end{equation*}
Since $P_c^{\circ (p-1)}(c)=0$, we have $|f_\lambda^{\circ (p-1)}(z)|=|f_\lambda^{\circ (p-1)}(z)-P_c^{\circ (p-1)}(c)|\preceq|\lambda|^{4/15}$. This implies that
\begin{equation*}
\begin{split}
|f_\lambda^{\circ p}(z)-c|
=&\,\left|\Big(f_\lambda^{\circ (p-1)}(z)\Big)^2+\frac{\lambda}{\big(f_\lambda^{\circ (p-1)}(z)-c\big)^2}\right| \\
\preceq&~|\lambda|^{8/15}+|\lambda|
\preceq|\lambda|^{8/15}.
\end{split}
\end{equation*}

Since $|\lambda|^{8/15}\ll|\lambda|^{1/2}$ and $|\lambda|^{11/30}<|c_j^\lambda-c|< |\lambda|^{3/10}$ for all $1\leq j\leq 3$ if $\lambda$ is small enough, it implies that $f_\lambda^{\circ p}(c_j^\lambda)\in T_\lambda$ provided $\lambda$ is sufficiently small. Therefore, all the critical points of $f_\lambda$ escape to $\infty$ and hence $f_{\lambda}$ is hyperbolic.
\end{proof}

\begin{rmk}
A small perturbation of a hyperbolic rational map is not necessarily still hyperbolic. For example, let $F(z)=z^2$ and $F_\lambda(z)=z^2+\lambda/z^2$, where $\lambda\in\C$. There exists a sequence $\lambda_n\to 0$ as $n\to\infty$ such that $F_\lambda$ is not hyperbolic. See \cite{DG08}.
\end{rmk}

\subsection{The Julia set moves holomorphically}

According to \cite[Corollary 4.12]{Mil06}, $J(P_c)$ is the boundary of the basin of infinity. If $\lambda\neq 0$ is sufficiently small, it is reasonable to expect that the dynamics near $J(P_c)$ would not be changed too much. Indeed, we will use holomorphic motion to show that $J(P_c)$ moves holomorphically if $\lambda$ varies in a small neighborhood of the origin.

\begin{defi}[{Holomorphic motion, \cite{MSS83}, \cite{Lyu83}}]
Let $E$ be a subset of $\EC$, a map $h:\mathbb{D}\times E\rightarrow\EC$ is called a \textit{holomorphic motion} of $E$ parameterized by the unit disk $\mathbb{D}$ and with base point $0$ if
\begin{enumerate}
\item For every $z\in E$, $\lambda\mapsto h(\lambda,z)$ is holomorphic for $\lambda$ in $\mathbb{D}$;
\item For every $\lambda\in\mathbb{D}$, $z\mapsto h(\lambda,z)$ is injective on $E$; and
\item $h(0,z)=z$ for all $z\in E$.
\end{enumerate}
\end{defi}

The unit disk $\mathbb{D}$ in the above definition can be replaced by any connected complex manifold. See \cite[\S 4.1]{McM94b}.

\begin{lema}[{$\lambda$-lemma, \cite{MSS83}}]\label{lema-lambda}
A holomorphic motion $h:\mathbb{D}\times E\rightarrow\EC$ of $E$ has a unique extension to a holomorphic motion of $\overline{E}$. The extension is a continuous map. Moreover, for each $\lambda\in \mathbb{D}$, the map $h(\lambda,\cdot):E\rightarrow\EC$ extends to a quasiconformal map of $\EC$ to itself.
\end{lema}

Let $B_0$ be the basin of infinity of $f_0:=P_c$ and recall that $B_\lambda$ is the immediate attracting basin of $\infty$ of $f_\lambda$. Let $\delta>0$ be the constant introduced in Lemma \ref{lema:hyper} and define $\Lambda:=\{\lambda\in\C:|\lambda|<\delta\}$.

\begin{lema}\label{holo-motion}
There is a holomorphic motion $h:\Lambda\times J(f_0)\rightarrow\mathbb{C}$ parameterized by $\Lambda$ and with base point $0$ such that
$h(\lambda, \partial B_0)=\partial B_\lambda$ for all $\lambda\in \Lambda$.
\end{lema}

\begin{proof}
We first prove that every repelling periodic point of $f_0$ moves holomorphically in $\Lambda$.  Let $z_0\in J(f_0)$ be such a point with period $p_0$. For small $\lambda$, the map $f_\lambda$ is a small perturbation of $f_0$. By implicit function theorem, there is a neighborhood $U_0$ of $0$ such that $z_0$ becomes a repelling point $z_\lambda$ of $f_\lambda$ with the same period $p_0$, for all $\lambda\in U_0$. On the other hand, by Lemma \ref{lema:hyper}, $f_\lambda$ has exactly one attracting cycle for all $\lambda\in\Lambda\setminus\{0\}$. This implies that each repelling cycle of $f_\lambda$ moves holomorphically throughout $\Lambda\setminus\{0\}$ (see \cite[Theorem 4.2]{McM94b}). In particular, by $\lambda$-lemma, there is a holomorphic motion of $J(f_{\lambda_0})$ parameterized by $\Lambda\setminus\{0\}$ and with any given base point $\lambda_0$.

Since $\Lambda$ is simply connected, there is a holomorphic map $Z:\Lambda\rightarrow \mathbb{C}$ such that $Z(\lambda)=z_\lambda$ for $\lambda\in U_0$. Note that $Z$ is a map depending on the periodic points. Let $\textup{Per}{(f_0)}$ be all repelling periodic points of $f_0$. Then the map $h:\Lambda\times \textup{Per}{(f_0)}\rightarrow\mathbb{C}$ defined by $h(\lambda,z_0)=Z(\lambda)$ is a holomorphic motion. Indeed, it is sufficient to verify that $h(\lambda,\cdot)$ is injective on $\textup{Per}(f_0)$ for any fixed $\lambda\in\Lambda\setminus\{0\}$. Suppose that there exists $\lambda_0\in\Lambda\setminus\{0\}$ such that $h(\lambda_0,z')=h(\lambda_0,z'')$, where $z'$ and $z''$ are different points contained in $\textup{Per}(f_0)$. Since $J(f_\lambda)$ moves holomorphically in $\Lambda\setminus\{0\}$, it implies that $h(\lambda,z')=h(\lambda,z'')$ for all $\lambda\in\Lambda\setminus\{0\}$. But this is a contradiction since $h(\lambda,z')\neq h(\lambda,z'')$ if $\lambda\neq 0$ is small enough. Notice that $J(f_0)=\overline{\textup{Per}{(f_0)}}$. By Lemma \ref{lema-lambda}, we obtain a holomorphic motion which is an extension of $h$, say $h:\Lambda\times J(f_0)\rightarrow\mathbb{C}$. It is obvious that $h(\lambda, J(f_0))$ is a connected subset\footnote{We will see later that $h(\lambda, J(f_0))$ is actually a Julia component of $f_\lambda$.} of $J(f_\lambda)$.

To finish, we show that $h(\lambda, \partial B_0)=\partial B_\lambda$ for all $\lambda\in \Lambda$. Actually, since $J(f_0)$ moves holomorphically in $\Lambda$, we have $h_\lambda\circ f_0(z)=f_\lambda\circ h_\lambda(z)$ for all $z\in J(f_0)$, where $h_\lambda:=h(\lambda,\cdot)$ (see \cite[\S 4.1]{McM94b}). Let $\widetilde{B}_\lambda$ be the connected component of $\EC\setminus h_\lambda(\partial B_0)$ containing $\infty$. Since $h_\lambda(\partial B_0)$ is fixed by $f_\lambda$, it implies that $f_\lambda(\widetilde{B}_\lambda)=\widetilde{B}_\lambda$ or $f_\lambda(\widetilde{B}_\lambda)=\EC$. However, if $\lambda$ is small enough, then $h_\lambda(\partial B_0)$ is close to $\partial B_0$ and $c\not\in f_\lambda(\widetilde{B}_\lambda)$. Indeed, if $f_\lambda(z)=c$, then $z^2(z-c)^2+\lambda=0$. This implies that $z$ is close to $0$ or $c$ and hence $c\not\in f_\lambda(\widetilde{B}_\lambda)$. Therefore, $f_\lambda(\widetilde{B}_\lambda)=\widetilde{B}_\lambda$ for sufficiently small $\lambda$ and hence $\widetilde{B}_\lambda=B_\lambda$ is the immediate attracting basin of $\infty$ of $f_\lambda$. By the uniqueness of the holomorphic motion, we have $f_\lambda(B_\lambda)=B_\lambda$ for all $\lambda\in\Lambda$. This implies $h(\lambda, \partial B_0)=\partial B_\lambda$ for all $\lambda\in\Lambda$.
\end{proof}

\begin{rmk}
The holomorphic motion has also been used in \cite{BDGR08} and \cite{GMR13} to prove the structure of the Julia sets still exists after perturbation for other singularly perturbed families. The authors in those papers proved that the B\"{o}ttcher map defined on the basin of infinity varies holomorphically.

However, our proof here can be adopted to the following two more general cases: the first one is when $\infty$ lies in a super-attracting basin whose period is greater than one, and the second is when the corresponding basins of the perturbed maps are no longer super-attracting where the B\"{o}ttcher map cannot be defined. This observation will be used in the next section.
\end{rmk}

Recall that $A_\lambda$ is the Fatou component of $f_\lambda$ containing $3$ critical points $\{c_j^\lambda:1\leq j\leq 3\}$ and $U_\lambda$ is the Fatou component containing the critical point $c_0^\lambda$.

\begin{cor}\label{cor:different}
The four Fatou components $B_\lambda$, $A_\lambda$, $U_\lambda$ and $T_\lambda$ are pairwise distinct. Moreover,
$A_\lambda$ is an annulus and the forward images of $f_\lambda(A_\lambda)$ are simply connected. In particular, $U_\lambda$, $T_\lambda$ and $B_\lambda$ are all simply connected.
\end{cor}

\begin{proof}
Note that we have the following orbit under $f_\lambda$ (see Figure \ref{Fig_diagram-1}):
\begin{equation}\label{equ:crit-orbit-1}
A_\lambda\to f_\lambda(A_\lambda)\to\cdots\to U_\lambda\to T_\lambda\to B_\lambda\to B_\lambda.
\end{equation}
If $B_\lambda=A_\lambda$, then we have $B_\lambda=A_\lambda=U_\lambda=T_\lambda$ since $f_\lambda(B_\lambda)=B_\lambda$. In this case $B_\lambda$ contains all the critical points of $f_\lambda$ and the Julia set of $f_\lambda$ is a Cantor set \cite[Lemma 8.1]{Mil93}. This is impossible since $J(f_\lambda)$ contains a homeomorphic image of $J(f_0)$ by Lemma \ref{holo-motion}. This implies that $B_\lambda\neq A_\lambda$. Since $A_\lambda$ contains an annulus separating $T_\lambda$ and $B_\lambda$, we have $B_\lambda\neq T_\lambda$. Therefore, by \eqref{equ:crit-orbit-1} we know that $B_\lambda$, $A_\lambda$, $U_\lambda$ and $T_\lambda$ are four pairwise different Fatou components.

According to Lemmas \ref{lema:hyper} and \ref{holo-motion}, it follows that $B_\lambda\setminus\{\infty\}$ does not contain any critical points and critical values. Hence $B_\lambda$ is simply connected. By \eqref{equ:crit-orbit-1}, we know that $f_\lambda:T_\lambda\setminus\{c\}\to B_\lambda\setminus\{\infty\}$ is a degree two covering map. It follows that $T_\lambda$ is simply connected. Similarly, $U_\lambda=f_\lambda^{\circ (p-1)}(A_\lambda)$ is simply connected. Moreover, $f_\lambda(A_\lambda)$, $\cdots$, $f_\lambda^{\circ (p-2)}(A_\lambda)$ are simply connected since they do not contain critical points (if $p\geq 3$). Since $f_\lambda:A_\lambda\to f_\lambda(A_\lambda)$ is a degree $3$ branched covering and $A_\lambda$ contains exactly $3$ critical points, it implies that $A_\lambda$ is an annulus by Riemann-Hurwitz's formula.
\end{proof}

\subsection{The semi-buried component and buried Jordan curves}

Recall that $\Lambda=\D(0,\delta)$ and $h:\Lambda\times J(f_0)\to\C$ is a holomorphic motion introduced in Lemma \ref{holo-motion}. Let us fix $\lambda\in\Lambda\setminus\{0\}$ and define $h_\lambda=h(\lambda,\cdot):J(f_0)\to\C$. Then
\begin{equation*}
J_0(f_\lambda):=h_\lambda(J(f_0))
\end{equation*}
is contained in $J(f_\lambda)$ and $h_\lambda:J(f_0)\to J_0(f_\lambda)$ is a restriction of a quasiconformal homeomorphism by Lemma \ref{lema-lambda}.

Let $T_0$ be the Fatou component of $f_0=P_c$ containing the critical value $c$. Note that $h_\lambda(\partial T_0)$ and $\partial T_\lambda$ are very different. In particular, $\partial T_\lambda$ is compactly contained in the interior of $h_\lambda(T_0)$.

\begin{prop}\label{prop-homeo-11}
Let $0<|\lambda|<\delta$. Then $J_0(f_\lambda)$ is a semi-buried Julia component of $f_\lambda$ and $J(f)$ contains infinitely many buried Julia components which are Jordan curves.
\end{prop}

\begin{proof}
Let $D_1$, $D_2$ and $D_3$, respectively, be the unique unbounded component of $\C\setminus \overline{A}_\lambda$, unique bounded component of $\C\setminus \overline{A}_\lambda$ containing $c$ and the unique unbounded component of $\C\setminus \overline{T}_\lambda$. By Lemma \ref{lema:hyper}, $f_\lambda$ is hyperbolic. This implies that $\partial D_i$ does not contain any critical points, where $i=1,2,3$. Note that $f_\lambda^{\circ p}:\partial D_3\to h_\lambda(\partial T_0)$, $f_\lambda^{\circ p}:\partial D_2\to \partial D_3$ and $f_\lambda^{\circ p}:\partial D_1\to \partial D_3$ are covering maps. It follows that $\partial D_1$, $\partial D_2$ and $\partial D_3$ are Jordan curves since $h_\lambda(\partial T_0)$ is.
For simplicity, we denote
three annuli (see Figure \ref{Fig_diagram-1-local})
\begin{equation*}
V_1=A(\partial D_1,h_\lambda(\partial T_0)),~
V_2=A(\partial D_2,\partial D_3) \text{ and }
V=A(\partial D_3,h_\lambda(\partial T_0)).
\end{equation*}
Then we know that $f_\lambda^{\circ p}:V_1\to V$ and $f_\lambda^{\circ p}:V_2\to V$ are covering maps between annuli with degree $2$ and $4$ respectively.

\begin{figure}[!htpb]
  \setlength{\unitlength}{1mm}
  \centering
  \includegraphics[width=0.8\textwidth]{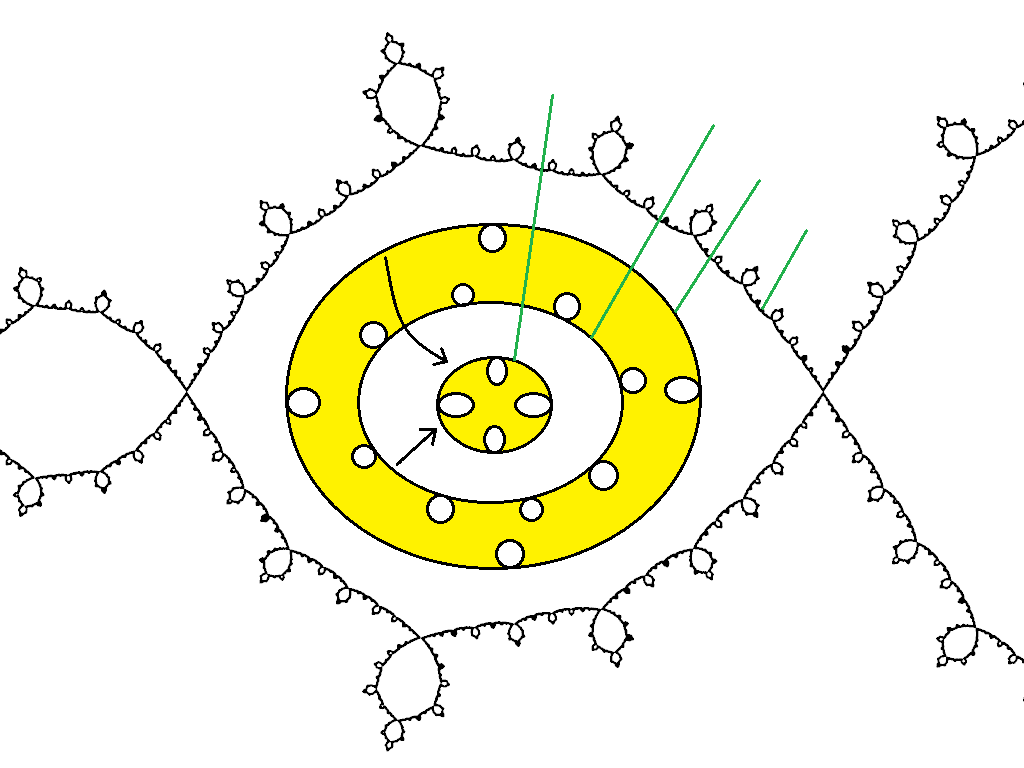}
  \put(-90,59){$B_\lambda$}
  \put(-61.5,47){$f_\lambda^{\circ p}$}
  \put(-60.2,28.9){$f_\lambda^{\circ p}$}
  \put(-54,35){$T_\lambda$}
  \put(-45.5,35){$V_2$}
  \put(-38,33){$A_\lambda$}
  \put(-29,35){$V_1$}
  \put(-26,54.5){$h_\lambda(\partial T_0)$}
  \put(-25.6,59){$\partial D_1$}
  \put(-29,64){$\partial D_2$}
  \put(-49,68){$\partial D_3$}
  \caption{The sketch of the locations of two annuli $V_1$ and $V_2$. Some Fatou components and mapping relations are also marked.}
  \label{Fig_diagram-1-local}
\end{figure}

By a similar argument as in the proof of Proposition \ref{prop-homeo}, it follows that $V_1$ contains a sequence of Julia components and a sequence of annular Fatou components of $f_\lambda$, such that they separate $c$ from $\infty$ and converge to $h_\lambda(\partial T_0)$ in the Hausdorff metric. This implies that $h_\lambda(\partial T_0)$ is semi-buried from $h_\lambda(T_0)$. Taking preimages of $h_\lambda(T_0)$ under $f_\lambda$, it is easy to see that $J_0(f_\lambda)$ is a semi-buried Julia component of $f_\lambda$.

Let $(j_0,j_1,\cdots,j_k,\cdots)$ be any infinite sequence satisfying $j_k\in\{1,2\}$. We denote
\begin{equation*}
J_{j_0 j_1\cdots j_k\cdots}:=\{z\in \overline{V}:~f_\lambda^{\circ kp}(z)\in \overline{V}_{j_k}, \forall\,k\in\N\}.
\end{equation*}
Note that $V_j\subset V$ and the identity $\textup{id}:V_j\hookrightarrow V$ is not homotopic to a constant map. By \cite[Lemma 2.4 (Case 2)]{PT00} and \cite[Proposition (Case 2)]{PT00}, every $J_{j_0 j_1\cdots j_k\cdots}$ is a Jordan curve surrounding the point $c\in T_\lambda$. Indeed, an intuitive observation is that each $J_{j_0 j_1\cdots j_k\cdots}$ is the intersection of a sequence of nested closed annuli surrounding $c$ by taking preimages of $f_\lambda^{\circ p}$ in $V$. In fact, according to \cite[Proposition 7.2]{McM88} and \cite[\S 3]{DLU05}, the set
\begin{equation*}
\mathcal{J}_0:=\{J_{j_0 j_1\cdots j_k\cdots}:j_k\in\{1,2\} \text{ and }k\in\N\}
\end{equation*}
is a Cantor set of circles which is contained in the Julia set of $f_\lambda$.

Except $J_{1 1\cdots 1\cdots}=h_\lambda(\partial T_0)$, any other $J_{j_0 j_1\cdots j_k\cdots}$ separates $c$ from $h_\lambda(\partial T_0)$. Since $V$ contains a doubly connected Fatou component $A_\lambda$, it implies that every pair of different $J_{j_0 j_1\cdots j_k\cdots}$'s are contained in different Julia components of $f_\lambda$. Moreover, $J_{j_0 j_1\cdots j_k\cdots}$ is buried in $\mathcal{J}_0$ if and only if the sequence $(j_0,j_1,\cdots,j_k,\cdots)$ does not end with infinitely many successive $1$'s.
Therefore, if $(j_0,j_1,\cdots,j_k,\cdots)$ does not end with infinitely many successive $1$'s, then $J_{j_0 j_1\cdots j_k\cdots}$ is a buried Julia component of $f_\lambda$ which is a Jordan curve. This implies that there are infinitely many buried Julia components which are Jordan curves. This finishes the proof of Proposition \ref{prop-homeo-11} and Theorem \ref{thm-exam}.
\end{proof}

\begin{rmk}
In fact one can obtain a precise classification of the topology of the Julia components of $f_\lambda$. Specifically, some Julia components are semi-buried copies of $J(P_c)$, some are buried Jordan curves and some are buried singletons. One can see \cite{BDGR08} and \cite{GMR13} for a similar classification.
\end{rmk}

\section{The cubic Julia sets containing buried Jordan curves}\label{sec-cubic}

In this section, we aim to construct some cubic rational maps whose Julia sets contain buried Jordan curves and more importantly, we show that these cubic rational maps contain a buried Julia component which is homeomorphic to the Julia set of a quadratic rational map. The construction is inspired by combining Lemma \ref{lema:cover-disk-to-disk} and Proposition \ref{prop-buried-p-1}.

For the quadatic rational map $Q_a(z)=1+a/z^2$, where $a\in\C\setminus\{0\}$, there are two critical points $0$ and $\infty$ in the critical orbit of $Q_a$ beginning with $0$. Therefore, similar to the consideration in last section, a natural idea is to make a perturbation of $Q_a$ at the critical value $1$. Let $a\not\in\{0,-(3+\sqrt{5})/2\}$ be the center of a hyperbolic component of $Q_a(z)=1+a/z^2$. We consider the following family:
\begin{equation*}
f_\nu(z)=1+\frac{a}{z^2}+\frac{(1+a+\sqrt{-a})\,\nu}{z-1-\nu},
\end{equation*}
where $\nu\in\C\setminus\{0\}$. This perturbation is made at the critical value $1$ of $Q_a$ so that $f_\nu(1)=-\sqrt{-a}$. Note that $-\sqrt{-a}$ is a preimage of $0$ under $Q_a$. Recall that $a=-(3+\sqrt{5})/2$ is excluded since we need to guarantee that $1+a+\sqrt{-a}\neq 0$.

\begin{lema}\label{lema:Jordan-domain-1}
If $a\neq 0$ is the center of a hyperbolic component of $Q_a$, then every Fatou component of $Q_a$ is a Jordan domain.
\end{lema}

\begin{proof}
Note that if $a\neq 0$ is the center of a hyperbolic component, then $Q_a$ is a critically finite quadratic rational map. According to \cite[Theorem 1.1]{Pil96} all the Fatou components of $Q_a$ are Jordan domains.
\end{proof}

The main content in this section is to prove that if $\nu\neq 0$ is small enough, then the Julia set of $f_\nu$ contains infinitely many buried Jordan curves and a fully buried Julia component which is homeomorphic to the Julia set of $Q_a(z)=1+a/z^2$.
In the remaining of this section, we always assume that $\nu\neq 0$ is sufficiently small. Let $p\geq 3$ be the minimal period of the critical point of $Q_a(z)=1+a/z^2$, i.e. $Q_a^{\circ p}(0)=0$ and $Q_a^{\circ p}(\infty)=\infty$.

\begin{lema}[See Figure \ref{Fig_diagram-2}]\label{lema:hyper-2}
There exists $\sigma>0$ such that if $0<|\nu|<\sigma$, then there exist five different Fatou components $A_\nu$, $U_\nu^1$, $U_\nu^{-1}$, $U_\nu^0$ and $U_\nu^\infty$ such that
\begin{enumerate}
\item $U_\nu^1$, $U_\nu^{-1}$, $U_\nu^0$ and $U_\nu^\infty$ contain $1$, $-\sqrt{-a}$, $0$ and $\infty$ respectively;
\item $A_\nu$ contains exactly two critical points $\{c_j^\nu:j=1,2\}$ of $f_\nu$ such that $c_j^\nu\in\{z:|\nu|^{9/16}\leq |z-1|\leq |\nu|^{7/16}\}\subset A_\nu$; In particular, $\lim_{\nu\to 0}c_j^\nu=1$;
\item $U_\nu^\infty$ contains exactly one critical point $c_\infty^\nu$ of $f_\nu$ such that $\lim_{\nu\to 0}c_\infty^\nu=\infty$;
\item $\overline{\D}(1,|\nu|^{11/8})\subset U_\nu^1$;
\item $f_{\nu}^{\circ (p-2)}(A_\nu)=U_\nu^0$, $f_{\nu}(U_\nu^0)=U_\nu^\infty$, $f_{\nu}(U_\nu^\infty)=U_\nu^1$ and $f_{\nu}(U_\nu^1)=U_\nu^{-1}$ and $f_{\nu}(U_\nu^{-1})=U_\nu^0$; and
\item The Fatou components $U_\nu^1$, $U_\nu^{-1}$, $U_\nu^0$ and $U_\nu^\infty$ contains a cycle of attracting periodic points with period $4$ and $f_\nu$ is hyperbolic.
\end{enumerate}
\end{lema}

\begin{figure}[!htpb]
  \setlength{\unitlength}{1mm}
  \centering
  \includegraphics[width=0.95\textwidth]{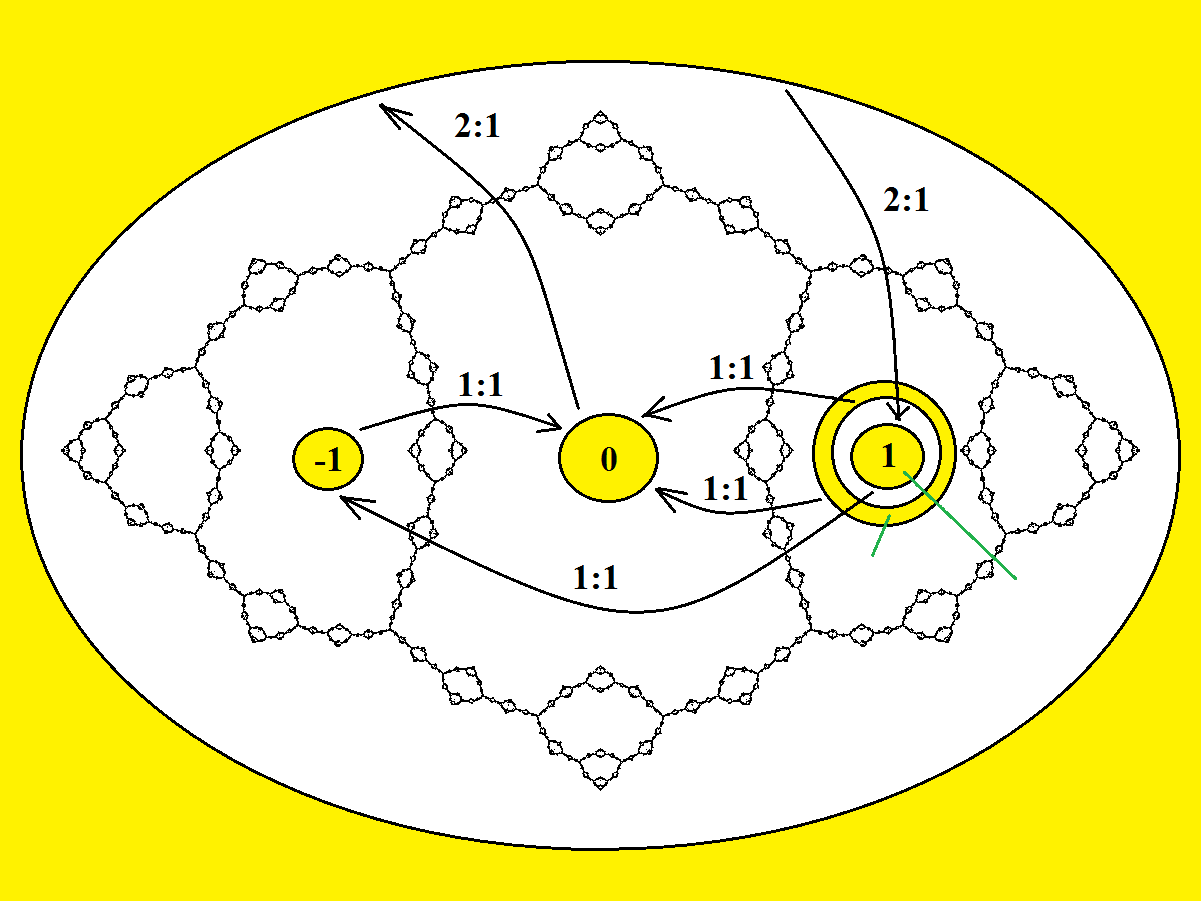}
  \put(-17,30){$U_\nu^1$}
  \put(-35,31){$A_\nu$}
  \put(-65,37){$U_\nu^0$}
  \put(-95,38){$U_\nu^{-1}$}
  \put(-34,83){$U_\nu^{\infty}$}
  \caption{The sketch of the mapping diagram of $f_\nu$, where $\nu\neq 0$ is small, $a=-1$ and $p=3$. Some arrows with mapping degrees between the boundaries of the Fatou components are also marked.}
  \label{Fig_diagram-2}
\end{figure}

\begin{proof}
One can check that $f_\nu$ has a simple critical point $0$. A direct calculation shows that
\begin{equation*}
f_\nu'(z)=-\frac{2a(z-1-\nu)^2+(1+a+\sqrt{-a})\nu z^3}{z^3(z-1-\nu)^2}.
\end{equation*}
Except $0$, to obtain the remaining three critical points $c_1^\nu$, $c_2^\nu$ and $c_\infty^\nu$, it is sufficient to solve the following equation (Now the parameter $a$ is seen as a constant and $1+a+\sqrt{-a}\neq 0$):
\begin{equation}\label{equ:deriv-2}
G(z,\nu):=2a(z-1-\nu)^2+(1+a+\sqrt{-a})\nu z^3=0.
\end{equation}
From \eqref{equ:deriv-2} we know that if $\nu\neq 0$ is small, then one critical point of $f_\nu$ is close to $\infty$ (without loss of generality we assume that it is $c_\infty^\nu$) and the remaining two critical points $c_1^\nu$ and $c_2^\nu$ are close to $1$. By implicit function theorem, if $\nu\neq 0$ is small enough, then $c_\infty^\nu$, $c_j^\nu$ with $j=1,2$ can be written as
\begin{equation*}
c_\infty^\nu=b_\infty\nu^{-1}+o(|\nu|^{-1}) \quad\text{ and }\quad c_j^\nu=1+(-1)^jb_j\nu^{1/2}+o(|\nu|^{1/2}).
\end{equation*}
Substituting the above series to \eqref{equ:deriv-2} and comparing the coefficients, one can obtain:
\begin{equation*}
b_\infty=-\frac{2a}{1+a+\sqrt{-a}}\neq 0 \quad\text{ and }\quad b_j=\sqrt{1/b_\infty}\neq 0\text{\quad for }j=1,2.
\end{equation*}
This implies that
\begin{equation*}
|c_\infty^\nu|\asymp|\nu|^{-1} \quad\text{ and }\quad |c_j^\nu-1|\asymp|\nu|^{1/2}, \text{ where }j=1,2.
\end{equation*}
In particular, $\lim_{\nu\to 0}c_\infty^\nu=\infty$ and $\lim_{\nu\to 0}c_j^\nu=1$ for $j=1,2$.

\medskip

For any $z\in\overline{\D}(1,|\nu|^{1+\varepsilon})$ with $\varepsilon=3/8$, one can write $z=1+\widetilde{z}$ with $|\widetilde{z}|\leq |\nu|^{1+\varepsilon}$ and we have
\begin{equation}\label{equ:est-1}
\begin{split}
&\,|f_\nu(z)+\sqrt{-a}|
=\left|1+\frac{a}{(1+\widetilde{z})^2}+\frac{(1+a+\sqrt{-a})\nu}{\widetilde{z}-\nu}+\sqrt{-a}\right|\\
=&\left|1+a\big(1+O(|\nu|^{1+\varepsilon})\big)-\big(1+a+\sqrt{-a}\big)\big(1+O(|\nu|^{\varepsilon})\big)+\sqrt{-a}\right|
\preceq|\nu|^{\varepsilon}.
\end{split}
\end{equation}
Then
\begin{equation}\label{equ:est-2}
\begin{split}
&\,|f_\nu^{\circ 2}(z)|
=\left|1+\frac{a}{(f_\nu(z))^2}+\frac{(1+a+\sqrt{-a})\nu}{f_\nu(z)-1-\nu}\right|\\
=&\left|1+\frac{a}{\big(-\sqrt{-a}+O(|\nu|^{\varepsilon})\big)^2}+\frac{(1+a+\sqrt{-a})\nu}{-\sqrt{-a}+O(|\nu|^{\varepsilon})-1-\nu}\right|
\preceq|\nu|^{\varepsilon}.
\end{split}
\end{equation}
Since
\begin{equation}\label{equ:est-3}
|f_\nu^{\circ 3}(z)|=\left|1+\frac{a}{(f_\nu^{\circ 2}(z))^2}+\frac{(1+a+\sqrt{-a})\nu}{f_\nu^{\circ 2}(z)-1-\nu}\right|,
\end{equation}
we have $|\nu|^{-2\varepsilon}\preceq|f_\nu^{\circ 3}(z)|$. Then
\begin{equation}\label{equ:est-4}
|f_\nu^{\circ 4}(z)-1|
=\left|\frac{a}{(f_\nu^{\circ 3}(z))^2}+\frac{(1+a+\sqrt{-a})\nu}{f_\nu^{\circ 3}(z)-1-\nu}\right|\preceq|\nu|^{4\varepsilon}.
\end{equation}

Since $|\nu|^{4\varepsilon}\ll|\nu|^{1+\varepsilon}$ if $\nu$ is small enough, it implies that $f_\nu^{\circ 4}$ maps the closed disk $\overline{\D}(1,|\nu|^{1+\varepsilon})$ into its interior. Then $f_\nu$ has an attracting cycle of periodic points with period $4$. Let $U_\nu^1$, $U_\nu^{-1}$, $U_\nu^0$ and $U_\nu^\infty$ be the Fatou components containing $1$, $-\sqrt{-a}$, $0$ and $\infty$ respectively. Then $U_\nu^1$, $U_\nu^{-1}$, $U_\nu^0$ and $U_\nu^\infty$ are pairwise different. Since $\overline{\D}(1,|\nu|^{1+\varepsilon})\subset U_\nu^1$ and $|\nu|^{-2\varepsilon}\ll|\nu|^{-1}$ if $\nu$ is small enough, we have $c_\infty^\nu\in U_\nu^\infty$.

\medskip

For $|\nu|^{9/16}\leq |z-1|\leq |\nu|^{7/16}$, one can write $z=1+\widehat{z}$ with $|\nu|^{9/16}\leq |\widehat{z}|\leq |\nu|^{7/16}$ and we have
\begin{equation*}
|f_\nu(z)-Q_a(1)|
=\left|\frac{a}{(1+\widehat{z})^2}-a+\frac{(1+a+\sqrt{-a})\nu}{\widehat{z}-\nu}\right|
\preceq|\nu|^{7/16}.
\end{equation*}
Similarly and inductively, for $|\nu|^{9/16}\leq |z-1|\leq |\nu|^{7/16}$ and $1\leq k\leq p-2$, we have
\begin{equation}\label{equ:f-nu-1}
|f_\nu^{\circ k}(z)-Q_a^{\circ k}(1)|\preceq|\nu|^{7/16}<|\nu|^{3/8}.
\end{equation}
Since $Q_a^{\circ p}(1)=1$, we have $Q_a^{\circ (p-1)}(1)=\infty$ and $Q_a^{\circ (p-2)}(1)=0$. Then $|f_\nu^{\circ (p-2)}(z)|\preceq|\nu|^{7/16}$ and we have
\begin{equation*}
\left|\frac{1}{f_\nu^{\circ (p-1)}(z)}\right|
=\left|1+\frac{a}{(f_\nu^{\circ (p-2)}(z))^2}+\frac{(1+a+\sqrt{-a})\nu}{f_\nu^{\circ (p-2)}(z)-1-\nu}\right|^{-1}\preceq |\nu|^{7/8}.
\end{equation*}
Therefore, we have
\begin{equation}\label{equ:f-nu-2}
|\nu|^{-7/8}\preceq|f_\nu^{\circ (p-1)}(z)| \text{\quad and\quad}
|f_\nu^{\circ p}(z)-1|\preceq|\nu|^{7/4}.
\end{equation}
Since $|\nu|^{7/4}\ll|\nu|^{11/8}$ and $|\nu|^{9/16}<|c_j^\nu-1|< |\nu|^{7/16}$ for all $j=1,2$ if $\nu$ is small enough, it implies that $f_\nu^{\circ p}(c_j^\nu)\in U_\nu^1$ provided $\nu$ is sufficiently small. Therefore, all the critical points of $f_\nu$ are attracted by an attracting cycle with period $4$ and hence $f_{\nu}$ is hyperbolic.

\medskip

Let $A_\nu$ be the Fatou component containing $\{z:|\nu|^{9/16}\leq |z-1|\leq |\nu|^{7/16}\}$. To conclude the proof, it is sufficient to show that $A_\nu$ is different from any one of $U_\nu^1$, $U_\nu^{-1}$, $U_\nu^0$ and $U_\nu^\infty$.
By applying a completely similar argument as Lemma \ref{holo-motion}, there is a holomorphic motion $h:\Xi\times J(f_0)\rightarrow\mathbb{C}$ parameterized by $\Xi:=\{\nu\in\C:|\nu|<\sigma\}$ with base point $0$, where $\sigma>0$ is a small number. For $\nu\in\Xi\setminus\{0\}$ we define
\begin{equation}\label{equ:h-nu}
h_\nu=h(\nu,\cdot):J(f_0)\to\C.
\end{equation}
Then $J_0(f_\nu):=h_\nu(J(f_0))$ is contained in $J(f_\nu)$ and $h_\nu:J(f_0)\to J_0(f_\nu)$ is a restriction of a quasiconformal homeomorphism. In particular, $J(f_\nu)$ contains a quasiconformal copy of $J(f_0)=J(Q_a)$.

Let  $W_0^k$ be the Fatou components of $f_0:=Q_a$ containing $Q_a^{\circ k}(1)$, where $0\leq k\leq p-1$. Note that we have the following orbit under $Q_a$:
\begin{equation}\label{equ:f-nu-orbit}
\pm\sqrt{-a}\mapsto 0\mapsto \infty\mapsto 1\mapsto 1+a \mapsto\cdots.
\end{equation}
By Lemma \ref{lema:Jordan-domain-1}, each Fatou component of $J(Q_a)$ is a Jordan domain. It follows that $\sqrt{-a}$ and $-\sqrt{-a}$ lie in different Fatou components of $Q_a$ by Riemann-Hurwitz's formula.
Then we have
\begin{equation*}
0\in W_0^{p-2}, \quad \infty\in W_0^{p-1}, \quad 1\in W_0^0, \quad 1+a\in W_0^1,
\end{equation*}
and $\sqrt{-a}\in W_0^{p-3}$ or $-\sqrt{-a}\in W_0^{p-3}$. If $\sqrt{-a}\in W_0^{p-3}$, we use $\widetilde{W}_0^{p-3}$ to denote the Fatou component of $Q_a$ containing $-\sqrt{-a}$. Similarly, If $-\sqrt{-a}\in W_0^{p-3}$, we use $\widetilde{W}_0^{p-3}$ to denote the Fatou component of $Q_a$ containing $\sqrt{-a}$. We have
\begin{equation}\label{equ:W-0-p-3}
W_0^{p-3}\cap \widetilde{W}_0^{p-3}=\emptyset.
\end{equation}

Since $U_\nu^1$, $U_\nu^{-1}$, $U_\nu^0$ and $U_\nu^\infty$ are the Fatou components of $f_\nu$ containing $1$, $-\sqrt{-a}$, $0$ and $\infty$ respectively, we have
\begin{equation}\label{equ:h-nu-W}
U_\nu^0\subset h_\nu(W_0^{p-2}), \quad U_\nu^\infty\subset h_\nu(W_0^{p-1}) \text{\quad and\quad} U_\nu^1\subset h_\nu(W_0^0).
\end{equation}
Moreover, if $p\geq 4$, then $U_\nu^{-1}\subset h_\nu(W_0^{p-3})$ if $Q_a^{\circ (p-3)}(1)=-\sqrt{-a}$; and $U_\nu^{-1}\cap h_\nu(W_0^k)=\emptyset$ for all $0\leq k\leq p-1$ if $Q_a^{\circ (p-3)}(1)=\sqrt{-a}$.
If $p=3$, then $a=-1$ and $U_\nu^{-1}\cap h_\nu(W_0^k)=\emptyset$ for all $0\leq k\leq 2$.
By \eqref{equ:f-nu-1} and \eqref{equ:f-nu-2}, without loss of generality we can assume that $\sigma>0$ is small enough so that
\begin{equation*}
f_\nu^{\circ k}(A_\nu)\subset h_\nu(W_0^k), \text{\quad where\quad}0\leq k\leq p-1.
\end{equation*}

Since $\{h_\nu(W_0^k):0\leq k\leq p-1\}$ are pairwise disjoint, by \eqref{equ:h-nu-W} we know that $A_\nu$ is disjoint with $U_\nu^0$ and $U_\nu^\infty$.
If $p\geq 4$ and $Q_a^{\circ (p-3)}(1)=-\sqrt{-a}$, then $U_\nu^{-1}\subset h_\nu(W_0^{p-3})$ and hence $U_\nu^{-1}\cap h_\nu(W_0^0)=\emptyset$.
If $p\geq 4$ and $Q_a^{\circ (p-3)}(1)=\sqrt{-a}$, then $U_\nu^{-1}\cap h_\nu(W_0^0)=\emptyset$.
If $p=3$, we still have $U_\nu^{-1}\cap h_\nu(W_0^0)=\emptyset$. Therefore, $A_\nu$ is disjoint with $U_\nu^{-1}$ in all cases.

In the following we prove that $A_\nu\cap U_\nu^1=\emptyset$.
Note that we have the following orbit under $f_\nu$:
\begin{equation*}
A_\nu\to f_\nu(A_\nu)\to\cdots\to f_\nu^{\circ(p-2)}(A_\nu)=U_\nu^0\to
U_\nu^\infty\to U_\nu^1\to U_\nu^{-1}\to U_\nu^0.
\end{equation*}
Suppose that $A_\nu\cap U_\nu^1\neq\emptyset$, equivalently, $A_\nu=U_\nu^1$ is the Fatou component contained in $h_\nu(W_0^0)$. Then we have $f_\nu^{\circ 2}(A_\nu)\subset h_\nu(W_0^2)$ and $f_\nu^{\circ 2}(U_\nu^1)=U_\nu^0\subset h_\nu(W_0^{p-2})$. Since $\{h_\nu(W_0^k):0\leq k\leq p-1\}$ are pairwise disjoint, we conclude that $p=4$. By \eqref{equ:f-nu-orbit} we have $1+a=\sqrt{-a}$ since $1+a\neq -\sqrt{-a}$. Then $Q_a(1)=1+a=\sqrt{-a}\in f_\nu(A_\nu)\subset h_\nu(W_0^1)$ by \eqref{equ:f-nu-1} and $-\sqrt{-a}\in f_\nu(U_\nu^1)=U_\nu^{-1}\subset h_\nu(\widetilde{W}_0^1)$. According to \eqref{equ:W-0-p-3}, we have $f_\nu(A_\nu)\cap f_\nu(U_\nu^1)=\emptyset$, which contradicts to the assumption that $A_\nu=U_\nu^1$. Therefore, we have $A_\nu\cap U_\nu^1=\emptyset$.
\end{proof}

\begin{lema}\label{lema:Jordan-domain}
There exists $\sigma'\in(0,\sigma]$ such that if $0<|\nu|<\sigma'$, then the Fatou components $U_\nu^1$, $U_\nu^{-1}$, $U_\nu^0$ and $U_\nu^\infty$ are Jordan domains.
\end{lema}

\begin{proof}
By Lemma \ref{lema:hyper-2}, $U_\nu^1$, $U_\nu^{-1}$, $U_\nu^0$ and $U_\nu^\infty$ form an attracting cycle of Fatou components with period $4$. To prove that they are Jordan domains, the idea is to construct a polynomial-like mapping $f_\nu^{\circ 4}:\Omega_1\to \Omega_2$ and prove that $(f_\nu^{\circ 4},\Omega_1,\Omega_2)$ is conjugated to a quartic polynomial whose Julia set is a Jordan curve.

To this end, we define $\Omega_1':=\D(1,|\nu|^{1+\tau})$ with $\tau=1/4$. By applying a similar calculation from \eqref{equ:est-1} to \eqref{equ:est-4}, if $|z-1|=|\nu|^{1+\tau}$ (i.e. $z\in\partial\Omega_1'$), then we have
\begin{equation*}
\begin{split}
|f_\nu(z)+\sqrt{-a}|\asymp |\nu|^\tau,
&\quad |f_\nu^{\circ 2}(z)|\asymp|\nu|^\tau \text{ and} \\
|f_\nu^{\circ 3}(z)|\asymp|\nu|^{-2\tau},
&\quad |f_\nu^{\circ 4}(z)-1|\asymp|\nu|^{4\tau}.
\end{split}
\end{equation*}
Since $|\nu|^{1+\tau}\ll |\nu|^{4\tau}=|\nu|$, there exists a constant $C>0$ such that $\overline{\Omega}_1'\subset \Omega_2:=\D(1,C|\nu|)\subset f_\nu^{\circ 4}(\Omega_1')$. According to Lemma \ref{lema:hyper-2}, the connected component of $(f_\nu^{\circ 2})^{-1}(\Omega_2)$ containing $0$, and the connected component of $(f_\nu^{\circ 3})^{-1}(\Omega_2)$ containing $\infty$, contain the critical point $0$ and $c_\nu^\infty$ respectively. Let $\Omega_1$ be the connected component of $(f_\nu^{\circ 4})^{-1}(\Omega_2)$ containing $1$. Then $(f_\nu^{\circ 4},\Omega_1,\Omega_2)$ is a polynomial-like mapping with degree $4$. According to \cite[p.\,296]{DH85b}, $(f_\nu^{\circ 4},\Omega_1,\Omega_2)$ is quasiconformally conjugate to a quartic polynomial $P$. By quasiconformal surgery, one can obtain another polynomial $\widetilde{P}$ so that $\widetilde{P}$ is critically finite and $P$, $\widetilde{P}$ are quasiconformally conjugate in an open neighborhood of the corresponding Julia sets. Obviously, $\widetilde{P}$ is conjugate to $z\mapsto z^4$ and hence $J(P)$, $J(\widetilde{P})$ are Jordan curves.
This implies that $\partial U_\nu^\infty$ is a Jordan curve and hence $U_\nu^1$, $U_\nu^{-1}$, $U_\nu^0$ are all Jordan domains.
\end{proof}

Let $A$ be a bounded annulus in $\C$. We use $\partial^+A$ and $\partial^-A$, respectively, to denote the external and internal boundaries of $A$, so that $\partial^-A$ is contained in a bounded component of $\C\setminus\partial^+A$.

\begin{proof}[{Proof of Theorem \ref{thm-exam-2}}]
We only give a sketch here since the proof is almost a copy of that of Proposition \ref{prop-homeo-11}.
By Lemma \ref{lema:hyper-2}, the orbit $U_\nu^1\to U_\nu^{-1}\to U_\nu^0\to U_\nu^\infty\to U_\nu^1$ forms an attracting cycle of Fatou components and it contains exactly two critical points $0\in U_\nu^0$ and $c_\nu^\infty\in U_\nu^\infty$. By Lemma \ref{lema:Jordan-domain}, $U_\nu^1$, $U_\nu^{-1}$, $U_\nu^0$ and $U_\nu^\infty$ are all Jordan domains and hence they are simply connected. Since $A_\nu$ contains exactly two critical points, it follows that $A_\nu$ is an annulus by applying Riemann-Hurwitz's formula to $f_\nu^{\circ (p-2)}: A_\nu\to U_\nu^0$.

\medskip

By Lemma \ref{lema:Jordan-domain-1}, $\partial U_0^1$ is a Jordan curve and hence $h_\nu(\partial U_0^1)$ is also. By Lemma \ref{lema:Jordan-domain}, $\partial ^+ A_\nu$ and $\partial ^- A_\nu$ are Jordan curves. We denote the following three annuli
\begin{equation*}
V_1=A(\partial^+A_\nu,h_\nu(\partial U_0^1)),~
V_2=A(\partial^-A_\nu,\partial U_\nu^1) \text{ and }
V=A(\partial U_\nu^1,h_\nu(\partial U_0^1)).
\end{equation*}
Note that $f_\nu^{\circ p}:V_1\to V$ and $f_\nu^{\circ p}:V_2\to V$ are both covering maps with degree $4$. Similar to the proof of Proposition \ref{prop-homeo-11}, one can show that $h_\nu(\partial U_0^1)$ is semi-buried from $h_\nu(U_0^1)$ first. Taking preimages of $h_\nu(U_0^1)$ under $f_\nu$, one can show that $J_0(f_\nu)=h_\nu(J(f_0))$ is a fully buried Julia component of $f_\nu$. Then the existence of infinitely many buried Julia components which are Jordan curves is completely similar to the proof of Proposition \ref{prop-homeo-11}. %We leave the details to the reader.
\end{proof}

\begin{rmk}
In fact, to prove Theorem \ref{thm-exam-2}, Lemma \ref{lema:Jordan-domain} is not necessary. Even if we don't know $\partial U_0^1$, $\partial ^+ A_\nu$ and $\partial ^- A_\nu$ are Jordan curves, one can still use \cite[Lemma 2.4 (Case 2)]{PT00} and \cite[Proposition (Case 2)]{PT00} to conclude the same result. However, in this case it will be a little complicated to mark the annuli $V_1$, $V_2$ and $V$. Anyway based on Lemma \ref{lema:Jordan-domain} we can understand the topological structures of the Julia components of $f_\nu$ better.
\end{rmk}

%----------------------------------------------------------------------------------------------------------------
\bibliographystyle{amsalpha}
\bibliography{E:/Latex-model/Ref1}

\providecommand{\bysame}{\leavevmode\hbox to3em{\hrulefill}\thinspace}
\providecommand{\MR}{\relax\ifhmode\unskip\space\fi MR }
\providecommand{\MRhref}[2]{%
  \href{http://www.ams.org/mathscinet-getitem?mr=#1}{#2}
}
\providecommand{\href}[2]{#2}
\begin{thebibliography}{CMMR09}

\bibitem[BB99]{BB99}
R.~Bam{\'{o}}n and J.~Bobenrieth, \emph{The rational maps {$z\mapsto 1+1/\omega
  z^d$} have no {H}erman rings}, Proc. Amer. Math. Soc. \textbf{127} (1999),
  no.~2, 633--636.

\bibitem[BDGR08]{BDGR08}
P.~Blanchard, R.~L. Devaney, A.~Garijo, and E.~D. Russell, \emph{A generalized
  version of the {M}c{M}ullen domain}, Internat. J. Bifur. Chaos Appl. Sci.
  Engrg. \textbf{18} (2008), no.~8, 2309--2318.

\bibitem[Bea91a]{Bea91a}
A.~F. Beardon, \emph{The components of a {J}ulia set}, Ann. Acad. Sci. Fenn.
  Ser. A I Math. \textbf{16} (1991), no.~1, 173--177.

\bibitem[Bea91b]{Bea91b}
\bysame, \emph{Iteration of rational functions}, Graduate Texts in Mathematics,
  vol. 132, Springer-Verlag, New York, 1991.

\bibitem[BF14]{BF14}
B.~Branner and N.~Fagella, \emph{Quasiconformal surgery in holomorphic
  dynamics}, Cambridge Studies in Advanced Mathematics, vol. 141, Cambridge
  University Press, Cambridge, 2014.

\bibitem[CG93]{CG93}
L.~Carleson and T.~W. Gamelin, \emph{Complex dynamics}, Universitext: Tracts in
  Mathematics, Springer-Verlag, New York, 1993.

\bibitem[CMMR09]{CMMR09}
C.~P. Curry, J.~C. Mayer, J.~Meddaugh, and J.~T. Rogers, \emph{Any
  counterexample to {M}akienko's conjecture is an indecomposable continuum},
  Ergodic Theory Dynam. Systems \textbf{29} (2009), no.~3, 875--883.

\bibitem[CMT13]{CMT13}
C.~P. Curry, J.~C. Mayer, and E.~D. Tymchatyn, \emph{Topology and measure of
  buried points in {J}ulia sets}, Fund. Math. \textbf{222} (2013), no.~1,
  1--17.

\bibitem[DG08]{DG08}
R.~L. Devaney and A.~Garijo, \emph{Julia sets converging to the unit disk},
  Proc. Amer. Math. Soc. \textbf{136} (2008), no.~3, 981--988.

\bibitem[DH85]{DH85b}
A.~Douady and J.~H. Hubbard, \emph{On the dynamics of polynomial-like
  mappings}, Ann. Sci. \'{E}cole Norm. Sup. (4) \textbf{18} (1985), no.~2,
  287--343.

\bibitem[DLU05]{DLU05}
R.~L. Devaney, D.~M. Look, and D.~Uminsky, \emph{The escape trichotomy for
  singularly perturbed rational maps}, Indiana Univ. Math. J. \textbf{54}
  (2005), no.~6, 1621--1634.

\bibitem[Dur83]{Dur83}
P.~L. Duren, \emph{Univalent functions}, Grundlehren der Mathematischen
  Wissenschaften, vol. 259, Springer-Verlag, New York, 1983.

\bibitem[FY15]{FY15}
J.~Fu and F.~Yang, \emph{On the dynamics of a family of singularly perturbed
  rational maps}, J. Math. Anal. Appl. \textbf{424} (2015), no.~1, 104--121.

\bibitem[GMR13]{GMR13}
A.~Garijo, S.~M. Marotta, and E.~D. Russell, \emph{Singular perturbations in
  the quadratic family with multiple poles}, J. Difference Equ. Appl.
  \textbf{19} (2013), no.~1, 124--145.

\bibitem[God15]{God15}
S.~Godillon, \emph{A family of rational maps with buried {J}ulia components},
  Ergodic Theory Dynam. Systems \textbf{35} (2015), no.~6, 1846--1879.

\bibitem[HP12]{HP12}
P.~Ha{\"{\i}}ssinsky and K.~M. Pilgrim, \emph{Quasisymmetrically inequivalent
  hyperbolic {J}ulia sets}, Rev. Mat. Iberoam. \textbf{28} (2012), no.~4,
  1025--1034.

\bibitem[Lyu83]{Lyu83}
M.~Lyubich, \emph{Some typical properties of the dynamics of rational
  mappings}, Uspekhi Mat. Nauk \textbf{38} (1983), no.~5 (233), 197--198.

\bibitem[Lyu86]{Lyu86}
M.~Yu. Lyubich, \emph{Dynamics of rational transformations: topological
  picture}, Uspekhi Mat. Nauk \textbf{41} (1986), no.~4 (250), 35--95.

\bibitem[Mak87]{Mak87}
P.~M. Makienko, \emph{Iterations of analytic functions in {${\bf C}^*$}}, Dokl.
  Akad. Nauk SSSR \textbf{297} (1987), no.~1, 35--37.

\bibitem[Mar08]{Mar08}
S.~M. Marotta, \emph{Singular perturbations in the quadratic family}, J.
  Difference Equ. Appl. \textbf{14} (2008), no.~6, 581--595.

\bibitem[McM88]{McM88}
C.~T. McMullen, \emph{Automorphisms of rational maps}, Holomorphic functions
  and moduli, {V}ol. {I} ({B}erkeley, {CA}, 1986), Math. Sci. Res. Inst. Publ.,
  vol.~10, Springer, New York, 1988, pp.~31--60.

\bibitem[McM94]{McM94b}
\bysame, \emph{Complex dynamics and renormalization}, Annals of Mathematics
  Studies, vol. 135, Princeton University Press, Princeton, NJ, 1994.

\bibitem[Mil93]{Mil93}
J.~Milnor, \emph{Geometry and dynamics of quadratic rational maps}, Experiment.
  Math. \textbf{2} (1993), no.~1, 37--83, With an appendix by the author and
  Lei Tan.

\bibitem[Mil06]{Mil06}
\bysame, \emph{Dynamics in one complex variable}, third ed., Annals of
  Mathematics Studies, vol. 160, Princeton University Press, Princeton, NJ,
  2006.

\bibitem[Mor97]{Mor97}
S.~Morosawa, \emph{On the residual {J}ulia sets of rational functions}, Ergodic
  Theory Dynam. Systems \textbf{17} (1997), no.~1, 205--210.

\bibitem[Mor00]{Mor00}
\bysame, \emph{Julia sets of subhyperbolic rational functions}, Complex
  Variables Theory Appl. \textbf{41} (2000), no.~2, 151--162.

\bibitem[MSS83]{MSS83}
R.~Ma{\~{n}}{\'{e}}, P.~Sad, and D.~Sullivan, \emph{On the dynamics of rational
  maps}, Ann. Sci. \'{E}cole Norm. Sup. (4) \textbf{16} (1983), no.~2,
  193--217.

\bibitem[Na{\u{\i}}82]{Nai82}
V.~A. Na{\u{\i}}shul', \emph{Topological invariants of analytic and
  area-preserving mappings and their application to analytic differential
  equations in {${\bf C}^{2}$} and {${\bf C}P^{2}$}}, Trudy Moskov. Mat.
  Obshch. \textbf{44} (1982), 235--245.

\bibitem[Pil96]{Pil96}
K.~M. Pilgrim, \emph{Rational maps whose {F}atou components are {J}ordan
  domains}, Ergodic Theory Dynam. Systems \textbf{16} (1996), no.~6,
  1323--1343.

\bibitem[PM97]{PM97}
R.~P{\'{e}}rez-Marco, \emph{Fixed points and circle maps}, Acta Math.
  \textbf{179} (1997), no.~2, 243--294.

\bibitem[Pom75]{Pom75}
C.~Pommerenke, \emph{Univalent functions}, Vandenhoeck \& Ruprecht,
  G\"{o}ttingen, 1975.

\bibitem[PT99]{PT99}
K.~M. Pilgrim and L.~Tan, \emph{On disc-annulus surgery of rational maps},
  Dynamical systems, {P}roceedings of the {I}nternational {C}onference in
  {H}onor of {P}rofessor {S}hantao {L}iao, World Scientific, Singapore, 1999,
  pp.~237--250.

\bibitem[PT00]{PT00}
\bysame, \emph{Rational maps with disconnected {J}ulia set}, Ast\'{e}risque
  (2000), no.~261, xiv, 349--384, G\'{e}om\'{e}trie complexe et syst\`emes
  dynamiques (Orsay, 1995).

\bibitem[Qia95]{Qia95}
J.~Qiao, \emph{The buried points on the {J}ulia sets of rational and entire
  functions}, Sci. China Ser. A \textbf{38} (1995), no.~12, 1409--1419.

\bibitem[Qia97]{Qia97}
\bysame, \emph{Topological complexity of {J}ulia sets}, Sci. China Ser. A
  \textbf{40} (1997), no.~11, 1158--1165.

\bibitem[QYY15]{QYY15}
W.~Qiu, F.~Yang, and Y.~Yin, \emph{Rational maps whose {J}ulia sets are
  {C}antor circles}, Ergodic Theory Dynam. Systems \textbf{35} (2015), no.~2,
  499--529.

\bibitem[Shi87]{Shi87}
M.~Shishikura, \emph{On the quasiconformal surgery of rational functions}, Ann.
  Sci. \'{E}cole Norm. Sup. (4) \textbf{20} (1987), no.~1, 1--29.

\bibitem[SY03]{SY03}
Y.~Sun and C.~C. Yang, \emph{Buried points and {L}akes of {W}ada continua},
  Discrete Contin. Dyn. Syst. \textbf{9} (2003), no.~2, 379--382.

\bibitem[XQY14]{XQY14}
Y.~Xiao, W.~Qiu, and Y.~Yin, \emph{On the dynamics of generalized {M}c{M}ullen
  maps}, Ergodic Theory Dynam. Systems \textbf{34} (2014), no.~6, 2093--2112.

\bibitem[Yan17]{Yan17}
F.~Yang, \emph{Rational maps without {H}erman rings}, Proc. Amer. Math. Soc.
  \textbf{145} (2017), no.~4, 1649--1659.

\bibitem[Yin92]{Yin92}
Y.~Yin, \emph{On the {J}ulia sets of quadratic rational maps}, Complex
  Variables Theory Appl. \textbf{18} (1992), no.~3-4, 141--147.

\end{thebibliography}

\end{document}